\documentclass[a4paper,11pt,reqno]{amsart}

\usepackage{amssymb}
\usepackage[all]{xy}

\usepackage{hyperref}
\usepackage{xcolor}
\usepackage{enumitem}
\usepackage{subcaption}
\usepackage[final]{graphicx}
\usepackage{float}
\usepackage{epic}

\usepackage[all]{xy}
\usepackage{color}

\usepackage{verbatim}

\usepackage{tikz}
\usetikzlibrary[topaths]

\newcount\mycount

\numberwithin{equation}{section}

\numberwithin{equation}{section}
\newtheorem{theorem}{Theorem}[section]

\newtheorem{lemma}[theorem]{Lemma}

\newtheorem{proposition}[theorem]{Proposition}

\newtheorem{corollary}[theorem]{Corollary}

\theoremstyle{definition}
\newtheorem{definition}[theorem]{Definition}

\theoremstyle{remark}
\newtheorem{example}[theorem]{Example}

\theoremstyle{remark}

\newtheorem{remark}[theorem]{Remark}

\newcommand\bp{\begin{proof}}
\newcommand\ep{\end{proof}}

\newcommand{\boldw}{{\bf w}}

\begin{document}

 \title[Topological boundaries of connected graphs and Coxeter groups]{Topological boundaries of connected graphs and Coxeter groups}

\author{Mario Klisse}

\address{TU Delft, EWI/DIAM,
	P.O.Box 5031,
	2600 GA Delft,
	The Netherlands}

\email{m.klisse@tudelft.nl}

\begin{abstract}

We introduce and study certain topological spaces associated with connected rooted graphs. These spaces reflect combinatorial and order theoretic properties of the underlying graph and relate in the case of hyperbolic graphs to Gromov's hyperbolic compactification \cite{Gromov}. They are particularly tractable in the case of Cayley graphs of finite rank Coxeter groups. In that context we speak of the compactification and the boundary of the Coxeter group. As it turns out, the canonical action of the Coxeter group on its Cayley graph induces a natural action on the compactification and the boundary. From this we deduce that in this case our construction coincides with spaces defined in \cite{Caprace} (see also \cite{Lecureux} and \cite{Lam}). We further prove the amenability of the action, we characterize when the compactification is small at infinity and we study classes of Coxeter groups for which the action is a topological boundary action in the sense of Furstenberg.

The second part of the paper deals with the applications of our results to the study of (Iwahori) Hecke algebras. These are certain deformations of group algebras of Coxeter groups. We first study embeddings of Hecke C$^\ast$-algebras and prove property Akemann-Ostrand for a certain class of Hecke-von Neumann algebras. Lastly, we make use of results that are widely related to Kalantar-Kennedy their approach to the C$^\ast$-simplicity problem \cite{KalantarKennedy} to study the simplicity and injective envelopes of operator algebras associated with Hecke algebras.
\end{abstract}

\date{\today. \emph{MSC2010:} 20F55, 20F65, 46L05, 46L10. The author is supported by the NWO project ``The structure of Hecke-von Neumann algebras'', 613.009.125.}

\maketitle


\section*{Introduction}

Hyperbolic graphs are graphs which satisfy a certain negative curvature condition. Intuitively, a hyperbolic graph is a graph whose large-scale geometry looks similar to that of a tree. The concept goes back to Gromov, see \cite{Gromov}. Every hyperbolic graph $K$ admits a topological ``space at infinity'' $\partial_{h}K$, consisting of equivalence classes of certain sequences, and a corresponding ``compactification'' $K\cup\partial_{h}K$. These spaces turn out to have a rich structure which provides an excellent tool to study the underlying graph. Especially in the context of Cayley graphs of groups, the notion of hyperbolicity allows to explore connections between algebraic properties of the group and geometric properties of certain topological spaces (see \cite{Kapovich} for a survey). This led to a number of breakthroughs in the fields of geometric and combinatoric group theory.

Following Gromov's ideas, many similar constructions assigning topological spaces to graphs and groups have been presented. In this paper, our objective is to follow an analogous path by defining certain topological spaces associated with (countable, undirected and simplicial) connected rooted graphs which reflect combinatorial and order theoretic properties. Given a connected rooted graph $(K,o)$ we denote these spaces by $\overline{(K,o)}$ and $\partial(K,o)$. As it turns out, the spaces very naturally identify as spectra of certain unital abelian C$^\ast$-algebras. In particular, $\overline{(K,o)}$ and $\partial(K,o)$ are compact and Hausdorff where $K$ embeds into $\overline{(K,o)}$ as a dense subset. Our construction covers a number of interesting examples. Further, if the graph is hyperbolic, both spaces very nicely relate to Gromov's hyperbolic compactification and boundary.

\begin{theorem} \label{1}
Let $(K,o)$ be a hyperbolic connected rooted graph. Then there exists a continuous surjection $\phi:\partial(K,o)\rightarrow\partial_{h}K$. If the graph is locally finite, then $\phi$ extends to a continuous surjection $\tilde{\phi}:\overline{(K,o)}\rightarrow K\cup\partial_{h}K$ with $\tilde{\phi}|_{K}=\text{id}_{K}$.
\end{theorem}

In general, the structure of the spaces $\overline{(K,o)}$ and $\partial(K,o)$ is much less tractable than that of hyperbolic compactifications and boundaries. However, for certain classes of graphs this can be different. This is the case for Cayley graphs of finite rank Coxeter groups. These are groups freely generated by a finite set $S$ subject to relations of the form $(st)^{m_{st}}=e$ where $m_{ss}=1$, $m_{st}=m_{ts}$ and $m_{st}\geq2$ for $s\neq t$. They were first introduced in \cite{Coxeter} and satisfy a number of strong properties. Today, Coxeter groups find application in many branches of mathematics (for more information see \cite{Bourbaki}, \cite{Davis}, \cite{combinatorics}). For a given Coxeter system $(W,S)$ we denote the corresponding compactification (resp. boundary) associated with its Cayley graph by $\overline{(W,S)}$ (resp. $\partial(W,S)$). The canonical action of $W$ on itself via left multiplication induces an action of $W$ on $\overline{(W,S)}$ and $\partial(W,S)$. These actions turn out to have some desirable properties.

The author is grateful to Sven Raum and Adam Skalski, who pointed out that the space $\overline{(W,S)}$ coincides with the minimal combinatorial compactification associated with the Coxeter complex of $(W,S)$ appearing in \cite{Caprace} (see also \cite{Lecureux}, \cite{Lam}). It has been shown in \cite{Lecureux} that for every building $X$ suitable subgroups of the automorphism group $\text{Aut}(X)$ act amenably on its maximal combinatorial compactification (as defined in \cite{Caprace}) in the sense that there exists a net of continuous, almost equivariant maps from the combinatorial compactification to the space of all probability measures on the group. By identifying $\overline{(W,S)}$ in a suitable way we deduce the amenability of the actions $W\curvearrowright\overline{(W,S)}$ and $W\curvearrowright\partial(W,S)$.

\begin{theorem} \label{2}
Let $(W,S)$ be a finite rank Coxeter system. Then the canonical actions $W\curvearrowright\overline{(W,S)}$ and $W\curvearrowright\partial(W,S)$ are amenable.
\end{theorem}

We further completely characterize when the compactification of a Coxeter system $(W,S)$ is small at infinity in the sense that for every sequence $\left(x_{i}\right)_{i\in\mathbb{N}}\subseteq W$ converging to a boundary point $z\in\partial(W,S)$, $x_{i}\mathbf{w}\rightarrow z$ for every $\mathbf{w}\in W$. In that case, we call the system $(W,S)$ small at infinity. The notion of smallness at infinity has a number of interesting operator algebraic implications that we pick up in the second part of the paper. It implicitly appears in \cite{Kaimanovich}, \cite{Higson} and \cite{HiGu}. The proof of the following theorem crucially depends on Moussong's characterization of word hyperbolic Coxeter groups \cite[Theorem 17.1]{Moussong}.

\begin{theorem} \label{3}
Let $(W,S)$ be a finite rank Coxeter system. For $s\in S$ denote the centralizer of $s$ in $W$ by $C_{W}(s)$. Then the following statements are equivalent:

\begin{enumerate}
\item $(W,S)$ is small at infinity;
\item $\#C_{W}(s)<\infty$ for every $s\in S$;
\item $W$ is word hyperbolic and the map $\tilde{\phi}$ from Theorem \ref{1} is a homeomorphism.
\end{enumerate}
\end{theorem}
 Reflection centralizers of Coxeter groups have been studied in \cite{Allcock} and \cite{Brink}. The main theorem in \cite{Brink} gives the description of the centralizer $C_{W}(s)$ of a generator $s$ in a Coxeter group $W$ as a semidirect product of its reflection subgroup by the fundamental group of the connected component of the odd Coxeter diagram of $W$ containing $s$. From the equivalence of the first two statements in Theorem \ref{3} we deduce that the only irreducible affine type Coxeter system that is small at infinity is the infinite dihedral group. We further prove that an irreducible right-angled Coxeter system is small at infinity if and only if the group decomposes as a free product of finite Coxeter groups.

Another notion that was famously used in \cite{KalantarKennedy} is that of topological boundary actions. Topological and measurable boundary actions have been introduced by Furstenberg in \cite{Fu1}, \cite{Fu2} in the context of rigidity of Lie groups. Compared to its measurable counterpart, the notion of topological boundary actions initially received much less attention. That changed when its relevance for rigidity and simplicity problems of reduced group C$^\ast$-algebras was discovered by Kalantar and Kennedy in \cite{KalantarKennedy}. We will study under which circumstances Coxeter systems give rise to a boundary action on their boundary. For Coxeter systems $(W,S)$ that are right-angled or small at infinity we can completely characterize when the canonical action $W\curvearrowright\partial(W,S)$ admits a (topologically free) boundary action.

\begin{theorem} \label{4}
Let $\left(W,S\right)$ be a finite rank Coxeter system such that $W$ is infinite. Then the following statements hold:
\begin{itemize}
\item Assume that the system is right-angled and irreducible. Then the action of $W$ on $\partial(W,S)$ is a (topologically free) boundary action if and only if $\left|S\right|\geq3$.
\item Assume that the system is small at infinity. Then the action of $W$ on $\partial(W,S)$ is a (topologically free) boundary action if and only if $W$ is non-amenable.
\end{itemize}
\end{theorem}

The second half of the paper gives applications of our earlier studies in the case of operator algebras associated with (Iwahori) Hecke algebras.

(Iwahori) Hecke algebras, first studied in the 1950s, are deformations of group algebras of a Coxeter group $W$ depending on a specific deformation parameter $q\in\mathbb{R}_{>0}^{S}$. They are intimately related to the representation theory of algebraic groups (see e.g. \cite{IwahoriMatsumoto}, \cite{Bernstein}, \cite{KL2}) and received attention in the context of buildings and Kac-Moody groups acting on them \cite{Remy}. Hecke algebras can be naturally represented on the Hilbert space $\ell^{2}(W)$ and thus complete to C$^\ast$-algebras (resp. von Neumann algebras) that we denote by $C_{r,q}^{\ast}(W)$ (resp. $\mathcal{N}_{q}(W)$). For spherical or affine type Coxeter systems these operator algebras have been studied in \cite{Matsumoto}, \cite{KL}, \cite{Lusztig}. Much later, the study of weighted $L^{2}$-cohomology of Coxeter groups led to the exploration of general Hecke von Neumann algebras \cite{Dymara}. Recently, the investigation of the ideal structure of both Hecke C$^\ast$-algebras and Hecke-von Neumann algebras made some progress. Motivated by Garncarek's characterization of factoriality of single-parameter right-angled Hecke-von Neumann algebras \cite{Gar}, Caspers, Larsen and the author obtained results on the (non-)simplicity and trace-uniqueness of (right-angled) Hecke C$^\ast$-algebras \cite{Mario}. Later, Garncarek's characterization was extended by Raum and Skalski in \cite{Raum} to multi-parameters which completely settles the question for factoriality in the right-angled case. In this context, it should also be mentioned that the reduced group C$^\ast$-algebra of an irreducible Coxeter group has been proven to be simple if and only if the corresponding Coxeter system is of non-affine type (see \cite{Fe}, \cite{DeLaHarpe}, \cite{Cornulier}). Other relevant references about Hecke operator algebras are \cite{Dymara2}, \cite{Caspers} and \cite{CSW}.

Our present work is motivated by the approach in \cite[Section 5]{Caspers}. As it turns out, the C$^\ast$-subalgebra of $\mathcal{B}(\ell^{2}(W))$ generated by all Hecke C$^\ast$-algebras of a given Coxeter system $(W,S)$ naturally identifies with the reduced crossed product $C(\overline{(W,S)})\rtimes_{r}W$ of the group $W$ by the continuous functions on $\overline{(W,S)}$. We prove that the Hecke-von Neumann algebras $\mathcal{N}_{q}(W)$ trivially intersect with the compact operators on $\ell^{2}(W)$ if and only if the parameter $q$ is inside a certain region of $\mathbb{R}^{S}$. This immediately implies that for those $q$ the C$^\ast$-algebras $C_{r,q}^{\ast}(W)$ naturally embed into $C(\partial(W,S))\rtimes_{r}W$.

\begin{theorem} \label{5}
Let $(W,S)$ be a finite rank Coxeter system with $W$ being infinite. Then, $\mathcal{N}_{q}(W)\cap\mathcal{K}\neq0$ if and only if $q\in\mathcal{R}^{\prime}$. Here $\mathcal{R}^{\prime}$ is a certain set of parameters associated with the region of convergence of the multi-variate growth series $\sum_{\mathbf{w}\in W}z_{\mathbf{w}}$ of $W$ (see Section \ref{applications} for details).
\end{theorem}

\begin{corollary} \label{6}
Let $(W,S)$ be a finite rank Coxeter system. For $q\in\mathbb{R}_{>0}^{(W,S)}\setminus\mathcal{R}^{\prime}$ the Hecke C$^\ast$-algebra $C_{r,q}^{\ast}(W)$ naturally embeds into the reduced crossed product $C(\partial (W,S))\rtimes_{r}W$.
\end{corollary}
These observations provide a direct link between the topological spaces $\overline{(W,S)}$ and $\partial(W,S)$ (and the action of $W$ on them) and the Hecke operator algebras of the system which allows us to apply our earlier results.

As mentioned before, the notion of smallness at infinity has a number of interesting operator algebraic implications. Higson and Guentner used it in \cite{HiGu} to prove that for a word hyperbolic group $G$ the map $C_{r}^{\ast}(G)\odot JC_{r}^{\ast}(G)J\rightarrow\mathcal{B}(\ell^{2}(G))/\mathcal{K}(\ell^{2}(G))$, $x\otimes y\mapsto xy+\mathcal{K}(\ell^{2}(G))$ is continuous in the minimal tensor norm where $J$ denotes the modular conjugation operator. The same statement has earlier been shown by Akemann and Ostrand \cite{AO} for free groups by a different method. The notion of the property Akemann-Ostrand (property $(\mathcal{AO}))$ was introduced in \cite{Ozawa} and was famously applied by Ozawa to rigidity questions of von Neumann algebras. Variations of property $(\mathcal{AO})$ have later been introduced in \cite{Isono2} and \cite{Isono}. Ozawa proved in \cite{Ozawa} that finite von Neumann algebras that satisfy property $(\mathcal{AO})$ are solid in the sense that the relative commutant of any diffuse von Neumann subalgebra is injective. Using the similar notion of strong solidity, Ozawa and Popa \cite{OzawaPopa} were able to find classes of von Neumann algebras that have no (von Neumann algebraic) Cartan subalgebras. Their approach has been advanced by Popa and Vaes in \cite{PopaVaes} (see also Chifan-Sinclair \cite{ChifanSinclair}). Isono \cite{Isono2} later proved that factors with the weak-$\ast$ completely bounded approximation property that satisfy condition $(\mathcal{AO})^{+}$ are strongly solid. Using a method similar to that of Higson and Guentner (see also \cite[Section 5]{Caspers}), we prove that Hecke-von Neumann algebras of Coxeter systems that are small at infinity satisfy Isono's strong condition $(\mathcal{AO})$ (see \cite{Isono}). The same statement was claimed in \cite{Caspers} in the case of right-angled hyperbolic Coxeter groups, but unfortunately the argument presented there contains a gap.

\begin{theorem} \label{7}
Let $(W,S)$ be a Coxeter system that is small at infinity. Then for every $q\in\mathbb{R}_{>0}^{(W,S)}$ the Hecke-von Neumann algebra $\mathcal{N}_{q}(W)\subseteq\mathcal{B}(\ell^{2}(W))$ satisfies the strong condition $\left(\mathcal{AO}\right)$.
\end{theorem}

Using Garncarek's observation \cite[Section 6]{Gar} that Dykema's interpolated free group factors \cite{Dykema} can be seen as Hecke-von Neumann algebras of free products of finite right-angled Coxeter groups, Theorem \ref{7} in particular implies that interpolated free group factors satisfy the strong condition $(\mathcal{AO})$. We can further use Theorem \ref{7} to show that for Coxeter systems which are small at infinity the intersection $C_{r,q}^{\ast}(G)\cap JC_{r,q}^{\ast}(G)J$ is trivial for all $q\in\mathbb{R}_{>0}^{(W,S)}\setminus\mathcal{R}^{\prime}$.

In the final part of the paper we pick up implications related to the work of Kalantar and Kennedy. In \cite{KalantarKennedy} the authors built a connection between Hamana's theory of (C$^\ast$-dynamical) injective envelopes (see \cite{Hamana1}, \cite{Hamana2}, \cite{Hamana3}, \cite{Hamana4}) and Furstenberg's notion of topological boundary actions. They used this connection to reformulate the longstanding open problem to determine which discrete groups are C$^\ast$-simple (meaning that the corresponding reduced group C$^\ast$-algebra is simple) in terms of the structure of the action of the group on its Furstenberg boundary. They further partially answered a conjecture by Ozawa \cite{Ozawa2} about tight embeddings of exact C$^\ast$-algebras affirmatively. A series of breakthrough works followed (see e.g. \cite{BKKO}, \cite{Haagerup}, \cite{Adam} and also \cite{BK}, \cite{HK}, ...). Using Theorem \ref{4} we will apply some of the results in the context of operator algebras associated with Hecke algebras.\\

\noindent \emph{Structure.} The paper is organized as follows. In Section \ref{section 1} we recall basic notions about partially ordered sets, graphs and trees and their hyperbolic compactifications. Further, we will introduce Coxeter groups and multi-parameter Hecke algebras (resp. their operator algebras). In Section \ref{section 2} we define the topological spaces $\overline{(K,o)}$ and $\partial(K,o)$ associated with (countable, undirected and simplicial) rooted graphs $(K,o)$ and study their basic properties. We will call these spaces the compactification and the boundary of the graph and prove that in the case of hyperbolic graphs Gromov's hyperbolic boundary arises as a quotient of our boundary. Section \ref{section 3} discusses our construction in the context of Cayley graphs of finite rank Coxeter groups. We will show that the canonical action of a Coxeter group on itself by left multiplication induces a natural action on its boundary (resp. its compactification). It is well-behaved in the sense that the action is amenable. We will further fully characterize when the boundary is small at infinity and we prove for certain classes of Coxeter groups that the action is a boundary action in the sense of Furstenberg. Finally, in Section \ref{section 4} we pick up the results of the earlier sections and apply them to (operator algebras associated with) Hecke algebras by studying embedding questions, the Akemann-Ostrand property for certain Hecke-von Neumann algebras and results which are widely related to the theory of injective envelopes of Hecke C$^\ast$-algebras.\\

\noindent \emph{Recent development}: After completion of the first draft of this paper
in October 2020, in \cite{Mario3} the author has found applications of
his construction to the simplicity question of right-angled Hecke
C$^{\ast}$-algebras. Based on the results of Section \ref{section 3} and Section \ref{section 4}
he completely characterized the simplicity of this class of C$^{\ast}$-algebras.
This extends previous results by Caspers, Larsen and the author (see
\cite{Mario}).


\section{Preliminaries and notation} \label{section 1}

\subsection{General notation}

The scalar product of Hilbert spaces is linear in the first variable. The symbol $\odot$ denotes the algebraic tensor product of C$^\ast$-algebras, $\otimes$ denotes the minimal tensor product and the maximal tensor product is denoted by $\otimes_{max}$. We write $\rtimes_{r}$ for reduced crossed products. Further, the neutral element of a group is always denoted by $e$.


\subsection{Partially ordered sets}

A \emph{partial order} on a set $\mathcal{S}$ is a binary relation $\leq$ which is \emph{reflexive}, \emph{antisymmetric} and \emph{transitive}. A set with a partial order is called a \emph{partially ordered set} (\emph{poset}). If existent, the \emph{join} of a subset $T\subseteq\mathcal{S}$, denoted by $\vee T$, is the least upper bound of $T$, meaning that $\vee T\geq y$ for every $y\in T$ and $\vee T\leq x$ for every $x\in\mathcal{S}$ with $x\geq y$ for every $y\in T$. In the same way, the \emph{meet} of $T$, denoted by $\wedge T$, is the greatest lower bound of $T$, meaning that $\wedge T\leq y$ for every $y\in T$ and $\wedge T\geq x$ for every $x\in\mathcal{S}$ with $x\leq y$ for every $y\in T$.

In general, the join and the meet of a set do not necessarily exist. A poset in which all pairs of elements have a join is called a \emph{join-semilattice}. If every non-empty subset has a join, it is called a \emph{complete join-semilattice}. Dually, a poset in which all pairs of elements have a meet is called a \emph{meet-semilattice}. If every non-empty subset has a meet, it is called a \emph{complete meet-semilattice}.

\begin{lemma}
Let $\mathcal{S}$ be a complete meet-semilattice and let $T\subseteq\mathcal{S}$ be a set. If $T$ has an upper bound (i.e. an element $x\in\mathcal{S}$ with $x\geq y$ for all $y\in T$), then the join $\vee T$ exists.

Dually, if $\mathcal{S}$ is a complete join-semilattice and $T\subseteq\mathcal{S}$ is a set with a lower bound (i.e. an element $x\in\mathcal{S}$ with $x\leq y$ for all $y\in T$), then the meet $\wedge T$ exists.
\end{lemma}

\begin{proof}
Let $\mathcal{S}$ be a complete meet-semilattice and $T\subseteq\mathcal{S}$ a set having an upper bound. Then, the set $T^{\prime}:=\left\{ x\in\mathcal{S}\mid y\leq x\text{ for all }y\in T\right\}$  is non-empty. Since $\mathcal{S}$ is a complete meet-semilattice, the meet $x:=\wedge T^{\prime}$ exists. It satisfies $y\leq x$ for all $y\in T$ and $x^{\prime}\geq x$ for all $x^{\prime}\in T^{\prime}$, i.e. $x$ is the join of the set $S$. The second statement follows analogously.
\end{proof}


\subsection{Graphs and trees}

A \emph{graph} $K$ is a pair $K=(V,E)$ consisting of a \emph{vertex set} $V$ and an \emph{edge set} $E\subseteq V\times V$. In the case where the vertex and edge set of the graph $K$ are not designated, we will often write $x\in K$, meaning that $x$ is a vertex of the graph $K$. A graph is called \emph{countable} if $V$ is countable, it is called \emph{undirected} if for every element $\left(x,y\right)\in E$ also $\left(y,x\right)\in E$ and it is called \emph{simplicial} if $\left(x,x\right)\notin E$ for every $x\in V$ and if the graph contains no double edges. We will always assume that the graphs appearing in this paper are countable, undirected and simplicial.

Two vertices $x,y\in V$ are called \emph{adjacent} if $(x,y)\in E$. A path $\alpha=(\alpha_{i})_{i}$ of length $n\in\mathbb{N}\cup\left\{ \infty\right\}$  is a sequence $\alpha_{0}...\alpha_{n}$ of $n$ vertices for which $(\alpha_{i},\alpha_{i+1})\in E$ for every $0\leq i<n$. We call $K$ \emph{connected} if every two vertices of $K$ can be connected by a path. This induces a natural metric $d_{K}$ on $K$ via \begin{eqnarray} \nonumber d_{K}(x,y):=\min\left\{ n\mid\text{there is a path of length }n\text{ that connects }x\text{ and }y\right\} \text{.} \end{eqnarray} We call this the \emph{graph metric} on $K$. A path $\alpha$ is called \emph{geodesic} if $d_{K}(\alpha_{i},\alpha_{j})=\left|i-j\right|$ for all $i,j$. Without further comments we will often extend a finite geodesic path $\alpha_0...\alpha_n$ to an infinite path via $\alpha_0...\alpha_n \alpha_n \alpha_n ...$ and still call it (finite) geodesic. Further, we say that $K$ is a \emph{tree} if there is no finite path $\alpha=(\alpha_{i})_{i=0,...,n}$ with $\alpha_{0}=\alpha_{n}$ for which the vertices $\alpha_{0},...,\alpha_{n}$ are pairwise distinct. Trees have the useful property that every two vertices $x,y$ are connected by a unique geodesic path that we will often denote by $\left[x,y\right]$.

A vertex of a graph is said to have \emph{finite degree} if the number of vertices that are adjacent to it is finite. A graph whose vertices all have finite degree is called \emph{locally finite}. If there is a uniform bound on the degree of vertices, we say that the graph is \emph{uniformly locally finite}.

\begin{example}
Let $G$ be a group generated by a set $S$ where $e\notin S$. Set $S^{-1}:=\left\{ g^{-1}\mid g\in S\right\}$. The \emph{Cayley graph} $\text{Cay}(G,S)$ of $G$ with respect to $S$ is the graph defined by the vertex set $G$ and the edge set $\left\{ \left(g,h\right)\in G\times G\mid g^{-1}h\in S\cup S^{-1}\right\}$. If $S$ is finite, the corresponding Cayley graph is  countable, undirected, simplicial and uniformly locally finite.
\end{example}


\subsection{Hyperbolic compactifications}

In this subsection, we review the notion of Gromov's hyperbolic boundary of hyperbolic graphs. The results and definitions are as they appear in \cite[Chapter 5.3]{BrownOzawa} and \cite{Kapovich}. They go back to Gromov's original work \cite{Gromov}.

Let $K$ be a connected graph. A \emph{geodesic triangle} $\Delta$ consists of three points $x,y,z\in K$ and three geodesic paths connecting them. If there exists a number $\delta>0$ for which each of the paths is contained in the open $\delta$-tubular neighborhood of the union of the other two paths, such a triangle is called \emph{$\delta$-slim}. We say that the graph $K$ is \emph{hyperbolic} if there exists $\delta>0$ such that every geodesic triangle is $\delta$-slim. Note that trees are always hyperbolic.

The \emph{Gromov product} of a graph $K$ with base point $o\in K$ is defined by \begin{eqnarray} \nonumber \left\langle x,y\right\rangle _{o}:=\frac{1}{2}(d_{K}(o,x)+d_{K}(o,y)-d_{K}(x,y))\text{.} \end{eqnarray} If $K$ is hyperbolic, one can define an \emph{equivalence relation} $\sim_{h}$ on the set of all sequences $\mathbf{x}:=(x_i)_{i \in \mathbb{N}} \subseteq K$ which \emph{converge to infinity} in the sense that $\liminf_{i,j} \langle x_i, x_j\rangle_o =\infty$ by declaring two such sequences $\mathbf{x}$ and $\mathbf{y}$ to be equivalent if and only if $\liminf_{i,j} \langle x_i, y_j \rangle_o=\infty$. This definition does not depend on choice of the base point $o$. We write $\left[\mathbf{x}\right]_{h}$ for the equivalence class of $\mathbf{x}$. The \emph{hyperbolic boundary} (or \emph{Gromov boundary}) $\partial_{h}K$ of $K$ is the set of all equivalence classes of sequences in $K$ which converge to infinity. The union $K\cup\partial_{h}K$ is called the \emph{hyperbolic bordification} (or \emph{Gromov  bordification}). For $z\in\partial_{h}K$, $R>0$ define the sets
\begin{eqnarray}
\nonumber
U\left(z,R\right):=\{z^{\prime}\in \partial_{h}K & \mid & \text{there are sequences } \mathbf{x},\mathbf{y} \text{ converging to infinity}\\
\nonumber
& & \text{ with } z=[\mathbf{x}]_{h}, z^{\prime}=[\mathbf{y}]_{h} \text{ and }  \liminf_{i,j\rightarrow \infty} \left\langle x_{i}\text{, }y_{j}\right\rangle _{o}>R\}
\end{eqnarray}
and
\begin{eqnarray}
\nonumber
U^\prime \left(z,R\right):=U(z,R) \cup \{z^{\prime}\in K & \mid & \text{there is a sequence } \mathbf{x} \text{ converging to infinity}\\
\nonumber
& & \text{ with } z=[\mathbf{x}]_{h} \text{ and }  \liminf_{i \rightarrow \infty} \left\langle x_{i}\text{, }y \right\rangle _{o}>R\}\text{.}
\end{eqnarray}

One can topologize $\partial_h K$ by declaring the basis of neighborhoods for a point $z \in \partial_h K$ to be $\{U(z,R) \mid R>0\}$. One can further topologize $K\cup\partial_{h}K$ by endowing $K$ with the metric (i.e. discrete) topology and by declaring the basis of neighborhoods for a point $z \in \partial_h K$ to be $\{U^\prime (z,R) \mid R>0\}$. Again, these topologies are independent of the choice of the base point $o$. If we assume $K$ to be locally finite, then $K\cup\partial_{h}K$ is a compact space that contains $K$ as a dense open subset. In that context we also speak of $K\cup \partial_h K$ as the \emph{hyperbolic compactification} of $K$. Further, every automorphism (i.e. isometric bijection) of $K$ uniquely extends to a homeomorphism of $K\cup\partial_{h}K$.

A group $G$ generated by a finite set $S$ whose Cayley graph $\text{Cay}(G,S)$ is hyperbolic, is called \emph{word hyperbolic}. As it turns out, both the hyperbolicity of $\text{Cay}(G,S)$ and the hyperbolic boundary $\partial_{h}G:=\partial_{h}\text{Cay}(G,S)$, do not depend on the choice of the finite generating set $S$. The left action of $G$ on its Cayley graph induced by left multiplication extends to an amenable action on the compactification $G\cup\partial_{h}G$ and the boundary $\partial_{h}G$ (see e.g. \cite[Theorem 5.3.15]{BrownOzawa}).


\subsection{Coxeter groups}

A \emph{Coxeter group} $W$ is a group freely generated by a set $S$ subject to relations of the form $(st)^{m_{st}}=e$ where $m_{st}\in\left\{ 1,2,...,\infty\right\}$  with $m_{ss}=1$, $m_{st}\geq2$ for all $s\neq t$ and $m_{st}=m_{ts}$. The condition $m_{st}=\infty$ means that $s$ and $t$ are free with respect to each other, i.e. no relation of the form $(st)^{m}=e$ with $m\in\mathbb{N}$ is imposed. The pair $(W,S)$ is then called a \emph{Coxeter system}. It is \emph{right-angled} if $m_{st}\in\left\{ 2,\infty\right\}$  for all $s \neq t$. We will usually assume that the set $S$ is finite. In that case we say that the system $(W,S)$ has \emph{finite rank}. The \emph{Coxeter diagram} of $(W,S)$ is given by the vertex set $S$ and the edge set $\left\{ (s,t)\mid m_{st}\geq3\right\}$  where every edge connecting vertices $s,t\in S$ is labeled by the corresponding exponent $m_{st}$. It encodes the data of the system $(W,S)$. The \emph{odd Coxeter diagram} is obtained from the Coxeter diagram by removing all edges whose labels are even or infinite.

If $T\subseteq S$ is a subset, the subgroup $W_{T}$ of $W$ generated by $T$ is called a \emph{special subgroup}. It is also a Coxeter group with the same exponents as $W$, see \cite[Theorem 4.1.6]{Davis}. The system $(W,S)$ is called \emph{irreducible} if its Coxeter diagram is connected. This is equivalent to $W$ not having a non-trivial decomposition into a direct product of special subgroups. One usually distinguishes three classes of irreducible Coxeter systems.

\begin{definition}
Let $\left(W,S\right)$ be an irreducible Coxeter system.
\begin{itemize}
\item It is of \emph{spherical type} if it is locally finite, i.e. every finitely generated subgroup of $W$ is finite.
\item It is of \emph{affine type} if it is infinite, virtually abelian and has finite rank.
\item It is of \emph{non-affine type} if it is neither spherical nor affine.
\end{itemize}
\end{definition}

Both, spherical type and affine type Coxeter systems are entirely classified by their Coxeter diagrams. The corresponding diagrams are called $(A_{n})_{n\geq1}$, $(B_{n})_{n\geq2}$, $(D_{n})_{n\geq3}$, $(E_{n})_{6\leq n\leq8}$, $F_{4}$, $G_{2}$, $(H_{n})_{2\leq n\leq4}$, $(I_{n})_{n\geq3}$, $A_{\infty}$, $A_{\infty}^{\prime}$, $B_{\infty}$, $D_{\infty}$ and $(\tilde{A}_{n})_{n\geq1}$, $(\tilde{B}_{n})_{n\geq3}$, $(\tilde{C}_{n})_{n\geq2}$, $(\tilde{D}_{n})_{n\geq4}$, $(\tilde{E}_{n})_{6\leq n\leq8}$, $\tilde{F}_{4}$, $\tilde{G}_{2}$, $\tilde{I}_1$.

By \cite[Theorem 14.1.2 and Proposition 17.2.1]{Davis} a Coxeter group is amenable if and only if it decomposes as a direct product of spherical type and affine type Coxeter systems.

Every element $\mathbf{w}\in W$ can be written as a product $\mathbf{w}=s_{1}...s_{n}$ with generators $s_{1},...,s_{n}\in S$. Such an expression is called \emph{reduced} if it has minimal length, meaning that $n\leq m$ for every other expression $\mathbf{w}=t_{1}...t_{m}$ with $t_{1},...,t_{m}\in S$. The set of letters appearing in a reduced expression is independent of the choice of the reduced expression (see \cite[Proposition 4.1.1]{Davis}). With $\left|\mathbf{w}\right|:=n$ we define a \emph{word length} on $W$. For $\mathbf{v},\mathbf{w}\in W$ with $\left|\mathbf{v}^{-1}\mathbf{w}\right|=\left|\mathbf{w}\right|-\left|\mathbf{v}\right|$ (resp. $\left|\mathbf{w}\mathbf{v}^{-1}\right|=\left|\mathbf{w}\right|-\left|\mathbf{v}\right|$) we say that $\mathbf{w}$ \emph{starts} (resp. \emph{ends}) \emph{with $\mathbf{w}$} and write $\mathbf{v}\leq_{R}\mathbf{w}$ (resp. $\mathbf{v}\leq_{L}\mathbf{w}$). This defines a partial order that is called the \emph{weak right} (resp. \emph{weak left}) \emph{Bruhat order}. For convenience, we will usually write $\leq$ instead of $\leq_{R}$. The weak Bruhat orders have the important property that they define complete meet-semilattices on $W$ (see \cite[Theorem 3.2.1]{combinatorics}).

All Coxeter groups share three important (equivalent) conditions. We use the convention that $\widehat{s}$ means that s is removed from an expression.

\begin{theorem}[{\cite[Theorem 3.2.16 and Theorem 3.3.4]{Davis}}] \label{cancellation}
Let $(W,S)$ be a Coxeter system, $\mathbf{w}=s_{1}...s_{n}$ an expression for an element $\mathbf{w}\in W$ and $s,t\in S$. Then, the following conditions hold:

\begin{itemize}
\item \emph{Deletion condition}: If $s_{1} \ldots s_{n}$ is not a reduced expression for $\boldw$, then there exist $i<j$ such that $s_{1} \ldots \widehat{s_{i}} \ldots \widehat{s_{j}} \ldots s_{n}$ is also an expression for $\boldw$.
\item  \emph{Exchange condition}: If $\boldw=s_{1} \ldots s_{n}$ is reduced, then either $\left|s\boldw\right|=n+1$ or there exists $1\leq i\leq n$ with $s\boldw=s_{1} \ldots \widehat{s_{i}} \ldots s_{n}$.
\item \emph{Folding condition}: If $\left|s\boldw\right|=n+1$ and $\left|\boldw t\right|=n+1$, then either $\left|s \boldw t\right|=n+2$ or $\left|s \boldw t\right|=n$.
\end{itemize}
\end{theorem}

In the right-angled case, if we have cancellation of the form $s_{1}...s_{n}=s_{1}...\widehat{s_{i}}...\widehat{s_{j}}...s_{n}$ for $s_{1},...,s_{n}\in S$, then $s_{i}=s_{j}$ and $s_{i}$ commutes with every letter in the reduced expression for $s_{i+1} ...s_{j-1}$.

In \cite{Moussong} Moussong characterized those Coxeter groups that are word hyperbolic.

\begin{theorem}[{\cite[Theorem 17.1]{Moussong}}]
For every Coxeter system $(W,S)$ the following statements are equivalent:

\begin{itemize}
\item $W$ is word hyperbolic;
\item $W$ has no subgroups isomorphic to $\mathbb{Z}\times\mathbb{Z}$;
\item There is no subset $T\subseteq S$ such that $(W_{T},T)$ is an affine Coxeter system of rank $\geq3$ or that there exists a pair of disjoint subsets $T_{1},T_{2}\subseteq T$ with $(W_{T},T)=(W_{T_{1}}\times W_{T_{2}},T_{1}\cup T_{2})$ where $W_{T_1}$, $W_{T_2}$ are infinite.
\end{itemize}
\end{theorem}


\subsection{Multi-parameter Hecke algebras}

The operator algebras associated with (Iwahori) Hecke algbras of spherical and affine Coxeter groups have been studied early in \cite{Matsumoto}, \cite{KL} and \cite{Lusztig} whereas non-affine Hecke-von Neumann algebras first appear in \cite{Dymara} and \cite{Dymara2} (see also \cite[Chapter 19]{Davis}). Note that we use a different normalization of the generators than in the named references. Our notation is mainly taken from \cite{Gar} and \cite{Mario}. It coincides with the ones in \cite{Caspers}, \cite{CSW} and \cite{Raum}.

For a Coxeter system $(W,S)$ let $\mathbb{R}_{>0}^{(W,S)}$ (resp. $\mathbb{C}^{(W,S)}$ and $\left\{ -1,1\right\} ^{S}$) be the set of tuples $q:=(q_{s})_{s\in S}$ in $\mathbb{R}_{>0}^{S}$ (resp. in $\mathbb{C}^{S}$ and $\left\{ -1,1\right\} ^{S}$) of positive real numbers (resp. complex numbers and numbers $-1$ or $1$) with the property that $q_{s}=q_{t}$ whenever $s$ and $t$ are conjugate in $W$. For every such $q$ and every reduced expression $\mathbf{w}=s_{1}...s_{n}$ of $\mathbf{w}\in W$ we define \begin{equation} \nonumber q_{\mathbf{w}}:=q_{s_{1}}...q_{s_{n}}\text{, }\qquad p_{s}(q):=q_{s}^{-\frac{1}{2}}(q_{s}-1)\text{.} \end{equation} Note that $q_{\mathbf{w}}$ does not depend on the choice of the reduced expression for $\mathbf{w}$. The \emph{(Iwahori) Hecke algebra} $\mathbb{C}_{q}\left[W\right]$ associated with the Coxeter system $(W,S)$ and the multi-parameter $q\in\mathbb{R}_{>0}^{(W,S)}$ is the (unique) $\ast$-algebra spanned by a linear basis $\{ T_{\mathbf{w}}^{(q)}\mid\mathbf{w}\in W\}$  with
\begin{eqnarray} \label{multiplication}
T_{s}^{\left(q\right)}T_{\mathbf{w}}^{\left(q\right)}=\begin{cases}
T_{s\mathbf{w}}^{\left(q\right)} & \text{, if }\left|s\mathbf{w}\right|>\left|\mathbf{w}\right|\\
T_{s\mathbf{w}}^{\left(q\right)}+p_{s}(q)T_{\mathbf{w}}^{\left(q\right)} & \text{, if }\left|s\mathbf{w}\right|<\left|\mathbf{w}\right|
\end{cases}\text{,}
\end{eqnarray}
and \begin{eqnarray} \nonumber (T_{\mathbf{w}}^{(q)})^{\ast}=T_{\mathbf{w}^{-1}}^{(q)} \end{eqnarray} where $s\in S$, $\mathbf{w}\in W$ (see \cite[Proposition 19.1.1]{Davis}). Similarly, there exists a \emph{right-handed (Iwahori) Hecke algebra} $\mathbb{C}_{q}^{r}\left[W\right]$ spanned by a linear basis $\{ T_{\mathbf{w}}^{(q),r}\mid\mathbf{w}\in W\}$  satisfying an analogue to \eqref{multiplication} with the order of $s$ and $\mathbf{w}$ reversed. For every $s\in S$ the operators $T_{s}^{(q)}$ and $T_{s}^{(q),r}$ naturally act on the Hilbert space $\ell^{2}(W)$ of square-summable functions on $W$ with the canonical orthonormal basis $(\delta_{\mathbf{w}})_{\mathbf{w}\in W}$ via
\begin{eqnarray} \nonumber
T_{s}^{\left(q\right)}\delta_{\mathbf{w}}=\begin{cases}
\delta_{s\mathbf{w}} & \text{, if }\left|s\mathbf{w}\right|>\left|\mathbf{w}\right|\\
\delta_{s\mathbf{w}}+p_{s}(q)\delta_{\mathbf{w}} & \text{, if }\left|s\mathbf{w}\right|<\left|\mathbf{w}\right|\text{,}
\end{cases}
\end{eqnarray}
and
\begin{eqnarray}
\nonumber
T_{s}^{\left(q\right),r}\delta_{\mathbf{w}}=\begin{cases}
\delta_{\mathbf{w}s} & \text{, if }\left|\mathbf{w}s\right|>\left|\mathbf{w}\right|\\
\delta_{\mathbf{w}s}+p_{s}(q)\delta_{\mathbf{w}} & \text{, if }\left|\mathbf{w}s\right|<\left|\mathbf{w}\right|
\end{cases}\text{.}
\end{eqnarray}
This defines faithful $\ast$-representations $\mathbb{C}_{q}\left[W\right]\rightarrow\mathcal{B}(\ell^{2}(W))$ and $\mathbb{C}_{q}^{r}\left[W\right]\rightarrow\mathcal{B}(\ell^{2}(W))$. We will usually identify both $\ast$-algebras with their images under this representation. The corresponding \emph{(reduced) Hecke C$^\ast$-algebra} $C_{r,q}^{\ast}(W)$ is defined to be the norm closure of $\mathbb{C}_{q}\left[W\right]$ in $\mathcal{B}(\ell^{2}(W))$ and the corresponding \emph{Hecke-von Neumann algebra} is $\mathcal{N}_{q}(W):=(C_{r,q}^{\ast}(W))^{\prime\prime}$. Similarly, we define the \emph{right-handed (reduced) Hecke C$^\ast$-algebra} $C_{r,q}^{\ast,r}(W)$ and the \emph{right-handed (reduced) Hecke-von Neumann algebra} $\mathcal{N}_{q}^{r}(W)$. It follows from \cite [Proposition 19.2.1]{Davis} that the commutant of $\mathcal{N}_{q}(W)$ is $\mathcal{N}_{q}^{r}(W)$ and vice versa. Note that for $q_{s}=1$, $s\in S$ we get that $\mathbb{C}_{q}\left[W\right]=\mathbb{C}\left[W\right]$, $C_{r,q}^{\ast}(W)=C_{r}^{\ast}(W)$ and $\mathcal{N}_{q}(W)=\mathcal{L}(W)$ are respectively the group algebra, reduced group C$^\ast$-algebra and group von Neumann algebra of $W$. 

For every $q\in\mathbb{R}_{>0}^{(W,S)}$ the vector state $\tau$ defined by $\tau(x):=\left\langle x\delta_{e},\delta_{e}\right\rangle$  for $x\in\mathcal{B}(\ell^{2}(W))$ restricts to a faithful \emph{tracial state} $\tau_{q}$ (resp. $\tau_{q}^{r}$) on $C_{r,q}^{\ast}(W)$ and $\mathcal{N}_{q}(W)$ (resp. $C_{r,q}^{\ast,r}(W)$ and $\mathcal{N}_{q}^{r}(W)$) with $\tau_{q}(T_{\mathbf{w}}^{(q)})=0$ (resp. $\tau_{q}^{r}(T_{\mathbf{w}}^{(q)})=0$) for all $\mathbf{w}\in W\setminus\left\{ e\right\}$. Finally, for $\mathbf{w}\in W$ define $P_{\mathbf{w}}$ to be the orthogonal projection of $\ell^{2}(W)$ onto the subspace $\overline{\text{Span}\left\{ \delta_{\mathbf{v}}\mid\mathbf{v}\in W\text{ with }\mathbf{w}\leq_{R}\mathbf{v}\right\} }$ and $P_{\mathbf{w}}^{r}$ to be the orthogonal projection of $\ell^{2}(W)$ onto the subspace $\overline{\text{Span}\left\{ \delta_{\mathbf{v}}\mid\mathbf{v}\in W\text{ with }\mathbf{w}\leq_{L}\mathbf{v}\right\} }$.


\section{Boundaries of connected graphs} \label{section 2}

In the following section we will define certain topological spaces associated with connected graphs and study their basic properties as well as the connection of our construction with Gromov's hyperbolic boundary. As mentioned before, we will always assume that the graphs appearing in this paper are countable, undirected and simplicial. The construction presented here works in greater generality which might be the content of the author's future research.


\subsection{Construction and basic properties}

\begin{definition}
A \emph{rooted graph} $\left(K,o\right)$ is a graph $K$ equipped with a root $o\in K$. If $K$ is connected, we impose a partial order $\leq_o$ on $K$ by declaring $x\leq_o y$ if and only if there exists a geodesic path starting in $o$ and ending in $y$ which passes $x$. If the root $o$ is clear, we often just write $\leq$ instead of $\leq_o$. We call this the \emph{graph order} on $\left(K,o\right)$. Further, define relations $\geq_o$, $<_o$ and $>_o$ (resp. $\geq$, $<$ and $>$) in the natural way. If the \emph{join} or \emph{meet} (with respect to the partial order) of two elements $x,y\in K$ exists, we denote it by $x\vee_o y$ (resp. $x \vee y$) or $x\wedge_o y$ (resp. $x \wedge y$).
\end{definition}

One easily checks that the graph order indeed defines a partial order. Based on it, we define a topological space associated with the connected rooted graph $\left(K,o\right)$ into which $K$ naturally embeds as a dense subset. Let $\mathbf{x}=(x_i)_{i \in \mathbb{N}}$ be a sequence in $K$. We say that $\mathbf{x}$ \emph{$o$-converges} if for every $x\in K$ one either has $x \leq x_i$ for all large enough $i$ or $x \nleq x_i$ for all large enough $i$. Write $x\leq_o \mathbf{x}$ (resp. $x\leq \mathbf{x}$) in the first case and $x\nleq_o \mathbf{x}$ (resp. $x\nleq \mathbf{x}$) in the second one. Note that constant sequences in $K$ necessarily $o$-converge. We say that $\mathbf{x}$ \emph{$o$-converges to infinity} if $\sup_{x \leq \mathbf{x}} d_K(x,o)=\infty$. One easily checks that infinite geodesic paths always $o$-converge to infinity. On the set of all sequences in $K$ which $o$-converge we define an equivalence relation $\sim_o$ (resp. $\sim$) by declaring $\mathbf{x}\sim_o\mathbf{y}$ if and only if for every $x\in K$ the implications $x \leq \mathbf{x} \Leftrightarrow x\leq \mathbf{y}$ hold. Denote by $[\mathbf{x}]_{o}$ (resp. $[ \mathbf{x} ]$) the corresponding equivalence class of a sequence $\mathbf{x}$ and write $\partial{(K,o)}$ for the set of all equivalence classes of sequences which $o$-converge to infinity. We call this set the \emph{boundary of $(K,o)$}. Similarly, the \emph{bordification $\overline{(K,o)}$} is the set of all equivalence classes of sequences in $K$ which $o$-converge. We will view $K$ as a subset of its bordification by identifying the elements in $K$ with equivalence classes of constant sequences.

The following lemma is easy to check.

\begin{lemma} \label{poset}
Let $(K,o)$ be a connected rooted graph. Then the graph order on $(K,o)$ extends to a partial order on $\overline{(K,o)}$ via
\[
[\mathbf{x}]\leq_o [ \mathbf{y}]:\Leftrightarrow x\leq_o \mathbf{y}\text{ for every } x\in K \text{ with }x\leq_o \mathbf{x}
\]
for $ [\mathbf{x}], [\mathbf{y} ]\in \overline{(K,o)}$.
\end{lemma}

As before, we often just write $\leq$ instead of $\leq_o$ for the extended graph order. We equip $\overline{(K,o)}$ with the topology generated by the subbase of sets of the form
\begin{eqnarray}
\nonumber
\mathcal{U}_{x}:=\left\{ z\in\overline{(K,o)}\mid x\leq z\right\} \: \text{ and  } \: \mathcal{U}_{x}^{c}:=\left\{ z\in\overline{(K,o)}\mid x\nleq z\right\}
\end{eqnarray}
where $x\in K$. In particular, $\mathcal{U}_{x}$ is \emph{clopen} (closed and open) in $\overline{(K,o)}$. Further, we impose the subspace topology on $\partial(K,o)$.

\begin{lemma} \label{discreteness}
Let $(K,o)$ be a connected rooted graph. Then the following statements hold:
\begin{itemize}
\item $\overline{(K,o)}$ contains $K$ as a dense subset;
\item For $x\in K$ the one point set $\left\{ x\right\}$  is clopen if $x$ has finite degree;
\item If the graph is locally finite, then $K$ is a discrete subset of $\overline{(K,o)}$ and $\partial(K,o)=\overline{(K,o)}\setminus K$.
\end{itemize}
\end{lemma}

\begin{proof}
The density of $K\subseteq\overline{(K,o)}$ is clear. If $x\in K$ has finite degree, then $\left\{ x\right\}$ is open since either $\left\{ x\right\} =\bigcap_{y\in K:x<y,d_{K}(x,y)=1}(\mathcal{U}_{x}\cap\mathcal{U}_{y}^{c})$ or $\left\{ x\right\} =\mathcal{U}_{x}$. In particular, if the graph is locally finite, $K$ is a discrete subset of $\overline{(K,o)}$. It remains to show that $\partial(K,o)=\overline{(K,o)}\setminus K$.
For this, let $z\in\overline{(K,o)}$ be a point represented by a
sequence $\mathbf{x}=(x_{i})_{i\in\mathbb{N}}\subseteq K$ which $o$-converges
but which does not $o$-converge to infinity. Then $l:=\sup\{d_{K}(y,o)\mid y\in K\text{ with }y\leq z\}$
is finite. Since $K$ is locally finite there exists $i_{0}\in\mathbb{N}$
such that $x_{i}\notin\bigcup_{y\in K:d_{K}(y,o)=l+1}\mathcal{U}_{y}$
for all $i\geq i_{0}$. But then $d_{K}(x_{i},o)\leq l$ for all $i\geq i_{0}$.
In particular, again by the local finiteness of $K$, there exists
a subsequence of $\mathbf{x}$ which is constant. But $\mathbf{x}$
$o$-converges, so $z=y$ for some $y\in K$. This implies the claim. 
\end{proof}

\begin{remark} \label{no-geodesic}
\emph{(a)} It is in general not true that for a connected rooted graph $\left(K,o\right)$ the set $K\subseteq\overline{(K,o)}$ is open. Indeed, if we consider the first graph in Figure \ref{Figure} with the indicated sequence $(z_{i})_{i\in\mathbb{N}}$ of boundary points represented by infinite geodesic paths, then $z_{i}\rightarrow z$. The example in particular demonstrates that the boundary $\partial(K,o)$ is not necessarily compact.

\noindent \emph{(b)} Other than in the context of trees, it is in general not true that for a connected rooted graph $(K,o)$ and an element $x\in K$ the openness of the one point set $\left\{ x\right\}$  implies that $x$ has finite degree. Indeed, consider the second graph in Figure \ref{Figure}. Its vertex $z$ does not have finite degree but the one point set $\left\{ z\right\}$ is open since $\left\{ z\right\} =\mathcal{U}_{z}\cap\mathcal{U}_{z^{\prime}}^{c}$.

\noindent \emph{(c)} Other than for the Gromov compactification of a hyperbolic graph, in general not every element of the bordification $\overline{(K,o)}$ is represented by a (possibly finite) geodesic path; not even in the locally finite case. Consider for instance the sequence $(x_i)_{i \in \mathbb{N}}$ indicated in the third graph of Figure \ref{Figure}. It $o$-converges but the corresponding equivalence class can not be represented by a geodesic path.
\vspace{3mm}

\begin{figure}[h]
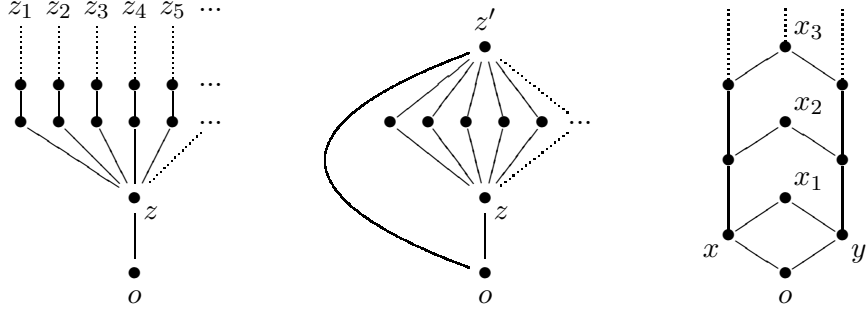

\centerline{
\xygraph{
!{<0cm, 0cm>; <1cm,0cm>: <0cm,1cm>::}
!{(0,0) }*+{{z_1}}="a1"
!{(0.5,0) }*+{{z_2}}="b1"
!{(1,0) }*+{{z_3}}="c1"
!{(1.5,0) }*+{{z_4}}="d1"
!{(2,0) }*+{{z_5}}="e1"
!{(2.5,0) }*+{...}="f1"
!{(0,-1) }*{\bullet}="g1"
!{(0.5,-1) }*{\bullet}="h1"
!{(1,-1) }*{\bullet}="i1"
!{(1.5,-1) }*{\bullet}="j1"
!{(2,-1) }*{\bullet}="k1"
!{(2.5,-1) }*+{...}="l1"
!{(0,-1.5) }*{\bullet}="m1"
!{(0.5,-1.5) }*{\bullet}="n1"
!{(1,-1.5) }*{\bullet}="o1"
!{(1.5,-1.5) }*{\bullet}="p1"
!{(2,-1.5) }*{\bullet}="q1"
!{(2.5,-1.5) }*+{...}="r1"
!{(1.5,-2.5) }*+{\bullet}="s1" *+!LU{z}
!{(1.5,-3.5) }*+{\bullet}="t1" *++!U{o}
!{(4,0) }*{ }
"a1"-@{.}"g1"
"b1"-@{.}"h1"
"c1"-@{.}"i1"
"d1"-@{.}"j1"
"e1"-@{.}"k1"
"g1"-"m1"
"h1"-"n1"
"i1"-"o1"
"j1"-"p1"
"k1"-"q1"
"m1"-"s1"
"n1"-"s1"
"o1"-"s1"
"p1"-"s1"
"q1"-"s1"
"s1"-"t1"
"s1"-@{.}"r1"
}
\xygraph{
!{<0cm, 0cm>; <1cm,0cm>: <0cm,1cm>::}
!{(1.25,-0.5) }*+{\bullet}="a2" *++!D{z^\prime}
!{(0,-1.5) }*{\bullet}="b2"
!{(0.5,-1.5) }*{\bullet}="c2"
!{(1,-1.5) }*{\bullet}="d2"
!{(1.5,-1.5) }*{\bullet}="e2"
!{(2,-1.5) }*{\bullet}="f2"
!{(2.5,-1.5) }*+{...}="g2"
!{(1.25,-2.5) }*+{\bullet}="h2" *+!LU{z}
!{(1.25,-3.5) }*+{\bullet}="i2" *++!U{o}
!{(4,-0.5) }*{ }
"a2"-"b2"
"a2"-"c2"
"a2"-"d2"
"a2"-"e2"
"a2"-"f2"
"a2"-@{.}"g2"
"h2"-"b2"
"h2"-"c2"
"h2"-"d2"
"h2"-"e2"
"h2"-"f2"
"h2"-@{.}"g2"
"h2"-"i2"
"i2"- @/^5pc/ "a2"
}
\xygraph{
!{<0cm, 0cm>; <1cm,0cm>: <0cm,1cm>::}
!{(0,0) }*{}="z4"
!{(0,-0.5) }*{\bullet}="z3" *+!LD{x_3}
!{(0,-1.5) }*{\bullet}="z2" *+!LD{x_2}
!{(0,-2,5) }*{\bullet}="z1" *+!LD{x_1}
!{(0,-3,5) }*{\bullet}="z0" *++!U{o}
!{(-0.75,0) }*{}="a4"
!{(-0.75,-1) }*{\bullet}="a3"
!{(-0.75,-2) }*{\bullet}="a2"
!{(-0.75,-3) }*{\bullet}="a1" *+!RU{x}
!{(0.75,0) }*{}="b4"
!{(0.75,-1) }*{\bullet}="b3"
!{(0.75,-2) }*{\bullet}="b2"
!{(0.75,-3) }*{\bullet}="b1" *+!LU{y}
"z0"-"a1"
"a1"-"a2"
"a2"-"a3"
"a3"-@{.}"a4"
"z0"-"b1"
"b1"-"b2"
"b2"-"b3"
"b3"-@{.}"b4"
"z1"-"a1"
"z1"-"b1"
"z2"-"a2"
"z2"-"b2"
"z3"-"a3"
"z3"-"b3"
"z3"-@{.}"z4"
}
}
\caption{Example \emph{(a)}, \emph{(b)} and \emph{(c)}} \label{Figure}
\end{figure}

\end{remark}

\begin{definition}
Let $(K,o)$ be a connected rooted graph and let \begin{eqnarray} \nonumber \pi:\mathcal{B}(\ell^{2}(K))\rightarrow\mathcal{B}(\ell^{2}(K))/\mathcal{K}(\ell^{2}(K)) \end{eqnarray} be the quotient map where $\mathcal{K}(\ell^{2}(K))$ denotes the compact operators in $\mathcal{B}(\ell^{2}(K))$. For every element $x\in K$ let $P_{x} \in\ell^{\infty}(K)\subseteq\mathcal{B}(\ell^{2}(K))$ be the orthogonal projection onto the subspace
\begin{eqnarray}
\nonumber
\overline{\text{Span}\left\{ \delta_{y}\mid y\in K\text{ with } y\leq x \right\} }\subseteq\ell^{2}(K)\text{.}
\end{eqnarray}
Denote by $\mathcal{D}(K,o)$ the (commutative) unital C$^\ast$-algebra generated by all $P_{x}$, $x\in K$. Note that $P_o=1$.
\end{definition}

\begin{proposition} \label{C*}
Let $(K,o)$ be a connected rooted graph. Then, $\text{Spec}(\mathcal{D}(K,o))\cong\overline{(K,o)}$ where $\text{Spec}(\mathcal{D}(K,o))$ denotes the character spectrum of $\mathcal{D}(K,o)$. In particular, $\overline{(K,o)}$ is a compact Hausdorff space.
Further, $\text{Spec}(\pi(\mathcal{D}(K,o)))\cong \overline{\partial(K,o)}$ where $\text{Spec}(\pi(\mathcal{D}(K,o)))$ denotes the character spectrum of $\pi(\mathcal{D}(K,o))$ and $\overline{\partial(K,o)}$ is the closure of $\partial (K,o)$ in $\overline{(K,o)}$.
\end{proposition}

\begin{proof}
Let $\mathbf{x}$ be a sequence in $K$ which $o$-converges. It is clear that $\lim_i \langle(\cdot)\delta_{{x}_{i}},\delta_{{x}_{i}}\rangle\in\text{Spec}(\mathcal{D}(K,o))$ is well-defined where the limit is taken in the weak-$\ast$ topology. Define a map $\psi:\overline{(K,o)}\rightarrow\text{Spec}(\mathcal{D}(K,o))$ by $z\mapsto\lim\langle(\cdot)\delta_{x_{i}},\delta_{x_{i}}\rangle$ for $z\in \overline{\left(K,o\right)}$ where $\mathbf{x}$ is a sequence representing $z$. The image of $z$ does not depend on the choice of the representing sequence for $z$. Indeed, let $\mathbf{x}$ and $\mathbf{y}$ be sequences in $K$ which $o$-converge and are equivalent to each other. For all $x\in K$ one has
\begin{eqnarray} \nonumber
\lim_{i}\langle P_{x}\delta_{x_{i}},\delta_{x_{i}}\rangle=\begin{cases}
1 & \text{, if }x\leq\mathbf{x}\\
0 & \text{, else}
\end{cases}=\begin{cases}
1 & \text{, if }x\leq \mathbf{y}\\
0 & \text{, else}
\end{cases}=\lim_{i}\langle P_{x}\delta_{y_{i}},\delta_{y_{i}}\rangle\text{,}
\end{eqnarray}
implying that $\lim\langle(\cdot)\delta_{x_{i}},\delta_{x_{i}}\rangle$ and $\lim\langle(\cdot)\delta_{y_{i}},\delta_{y_{i}}\rangle$ coincide on $\ast\text{-Alg}\left(\left\{ P_{x}\mid x\in K\right\} \right)$. Hence, $\psi$ is well-defined.

We proceed by showing that $\psi$ is continuous, injective, surjective and closed.
\begin{itemize}
\item \emph{Continuity:} The continuity follows in the same way as the well-definedness above.
\item \emph{Injectivity:} Let $\mathbf{x}$ and $\mathbf{y}$ be sequences in $K$ which $o$-converge and which are not equivalent to each other. Without loss of generality we can assume that there exists $x\in K$ with $x\leq\mathbf{x}$ and $x\nleq\mathbf{y}$. Then,
\begin{eqnarray} \nonumber
\qquad \qquad \lim_{i}\langle P_{x}\delta_{x_{i}},\delta_{x_{i}}\rangle=1 \: \text{ and } \: \lim_{i}\langle P_{x}\delta_{y_{i}},\delta_{y_{i}}\rangle=0\text{,}
\end{eqnarray}
which implies that $\psi\left(\left[\mathbf{x}\right]\right)\neq\psi\left(\left[\mathbf{y}\right]\right)$.
\item \emph{Surjectivity:} Let $\chi\in\text{Spec}(\mathcal{D}(K,o))$ be
a character on $\mathcal{D}(K,o)$. Define the set $\mathcal{S}:=\{x\in K\mid\chi(P_{x})=1\}$
and choose an enumeration $y_{1},y_{2},...$ of $\mathcal{S}$ (where
we assume that the sequence becomes constant if $\mathcal{S}$ is
finite) and an enumeration $y_{1}^{\prime},y_{2}^{\prime},...$ of
$K\setminus\mathcal{S}$ (where we assume that the sequence becomes
constant if $K\setminus\mathcal{S}$ is finite). For every $i\in\mathbb{N}$
the intersection $\mathfrak{I}_{i}=K\cap(\mathcal{U}_{y_{1}}\cap\mathcal{U}_{y_{1}^{\prime}}^{c})\cap...\cap(\mathcal{U}_{y_{i}}\cap\mathcal{U}_{y_{i}^{\prime}}^{c})$
must be non-empty because otherwise $P_{y_{1}}(1-P_{y_{1}^{\prime}})...P_{y_{i}}(1-P_{y_{i}^{\prime}})=0$
and hence 
\[
\qquad 1=\chi(P_{y_{1}})\chi(1-P_{y_{1}})...\chi(P_{y_{i}})\chi(1-P_{y_{i}^{\prime}})=\chi(P_{y_{1}}(1-P_{y_{1}^{\prime}})...P_{y_{i}}(1-P_{y_{i}^{\prime}}))=0\text{.}
\]
So choose for every $i\in\mathbb{N}$ an element $x_{i}\in\mathfrak{I}_{i}$
and consider the sequence $\mathbf{x}:=(x_{i})_{i\in\mathbb{N}}$
in $K$. By construction the
sequence $o$-converges and for $z:=[\mathbf{x}]\in \overline{(K,o)}$
we have $\psi(z)=\chi$. The surjectivity
follows.
\item \emph{Closedness:} It suffices to show that for every $x\in K$ the sets $\psi(\mathcal{U}_{x})$ and $\psi(\mathcal{U}_{x}^{c})$ are closed in $\text{Spec}(\mathcal{D}(K,o))$. Fix $x\in K$, let $(z^{i})_{i\in I}\subseteq\mathcal{U}_{x}$ be a net and let $z\in\overline{\left(K,o\right)}$ with $\psi(z^{i})\rightarrow\psi(z)$. We have $\left(\psi(z)\right)(P_{x})=\lim\left(\psi(z^{i})\right)(P_{x})=1$, so $z\in\mathcal{U}_{x}$. Hence, $\psi(\mathcal{U}_{x})$ is closed in $\text{Spec}(\mathcal{D}(K,o))$. The closedness of $\psi(\mathcal{U}_{x}^{c})$ follows in the same way.
\end{itemize}

We have shown that $\psi$ is a homeomorphism. The existence of a homeomorphism between $\text{Spec}(\pi(\mathcal{D}(K,o)))$ and $\overline{\partial(K,o)}$ follows in a similar way.
\end{proof}

Motivated by Proposition \ref{C*} we will often speak about $\overline{(K,o)}$ as the \emph{compactification of the graph $K$}.

\begin{remark} \label{remark}
\emph{(a)} The maps in Proposition \ref{C*} induce isomorphisms $\mathcal{D}(K,o)\cong C(\overline{(K,o)})$ via $P_{x}\mapsto\chi_{\mathcal{U}_{x}}$ and $\pi(\mathcal{D}(K,o))\cong C(\overline{\partial(K,o)})$ via $\pi(P_{x}) \mapsto\chi_{\mathcal{U}_{x}\cap \overline{\partial(K,o)}}$, where $\chi_{\mathcal{U}_{x}}$ (resp. $\chi_{\mathcal{U}_{x}\cap \overline{\partial(K,o)}}$) denotes the characteristic function on $\mathcal{U}_{x}$ (resp. $\mathcal{U}_{x}\cap \overline{\partial(K,o)}$).\\
\noindent \emph{(b)} The C$^\ast$-algebra $\mathcal{D}(K,o)$ appearing in Proposition \ref{C*} is separable. This implies that the topological space $\overline{(K,o)}$ is metrizable. The same holds for $\overline{\partial(K,o)}$ and hence for $\partial(K,o)$.\\
\noindent \emph{(c)} If $(K,o)$ is a locally finite connected rooted graph, then Lemma \ref{discreteness} and Proposition \ref{C*} imply that $\partial(K,o) = \overline{\partial(K,o)} \cong  \text{Spec}(\pi(\mathcal{D}(K,o)))$ is a (metrizable) compact Hausdorff space.
\end{remark}

As mentioned in Remark \ref{no-geodesic}, in general not every element of the compactification $\overline{(K,o)}$ is represented by a (possibly finite) geodesic path. Proposition \ref{characterization}  characterizes, when this is the case. Its proof requires the following simple lemma.

\begin{lemma} \label{monotone}
Let $(K,o)$ be a connected rooted graph and $(x_{i})_{i\in\mathbb{N}}\subseteq K$ a sequence with $x_{1}\leq x_{2}\leq...$. Then the sequence converges to a point in $\overline{(K,o)}$ which can be represented by a (possibly finite) geodesic path.
\end{lemma}

\begin{proof}
Choose a geodesic path starting in $o$ and ending in $x_{1}$ and denote it by $\left[o,x_{1}\right]$. Since $x_{1}\leq x_{2}$, there exists a geodesic path starting in $o$ and ending in $x_{2}$ which passes $x_{1}$. Denote by $\left[x_{1},x_{2}\right]$ its tail starting in $x_{1}$ and ending in $x_{2}$. Further, let $\left[o,x_{1}\right]\left[x_{1},x_{2}\right]$ be the concatenation of $\left[o,x_{1}\right]$ and $\left[x_{1},x_{2}\right]$. It is geodesic as well. Proceeding like this we get a geodesic path $\alpha:=\left[o,x_{1}\right]\left[x_{1},x_{2}\right]\left[x_{2},x_{3}\right]...$. If the path is finite, then the convergence of the sequence $(x_{i})_{i\in\mathbb{N}}$ is clear, so assume that $\alpha$ is infinite. We claim that $x_{i}\rightarrow\left[\alpha\right]$. Indeed, if $y\leq\left[\alpha\right]$, then $y\leq\alpha_{i}$ for all large enough $i$ and hence $y\leq x_{n}$ for all large enough $n$. If $y\nleq\left[\alpha\right]$, then $y\nleq\alpha_{i}$ for all large enough $i$ and then also $y\nleq x_{n}$ for all large enough $n$. Hence, $x_{i}\rightarrow\left[\alpha\right]$.
\end{proof}

\begin{proposition} \label{characterization}
Let $(K,o)$ be a connected rooted graph. Then the following statements
are equivalent:
\begin{enumerate}
\item Every element in $\overline{(K,o)}$ is represented by a (possibly
finite) geodesic path in $K$;
\item For every clopen subset $S\subseteq\overline{(K,o)}$ the number of
minimal elements in $S$ is finite;
\item For every $x,y\in K$ the number of minimal elements in $\mathcal{U}_{x}\cap\mathcal{U}_{y}$
is finite.
\end{enumerate}
In particular, if $(K,o)$ satisfies one (and hence all) of the conditions
above, then $\partial(K,o)=\overline{(K,o)}\setminus K$.
\end{proposition}

\begin{proof}
``$\text{(1)}\Rightarrow\text{(2)}$'' Assume that
every element in $\overline{(K,o)}$ is represented by a (possibly
finite) geodesic path in $K$ and that there exists a clopen subset
$S\subseteq\overline{(K,o)}$ of $K$ for which $\#T=\infty$, where
$T:=\left\{ z\in S\mid z\text{ minimal element in }S\right\} $. By
the first assumption we have $\overline{(K,o)}=K\cup\partial(K,o)$.
Further, for every boundary point $z\in S\cap\partial(K,o)$ represented
by an infinite geodesic path $\alpha=(\alpha_{i})_{i\in\mathbb{N}}$
with $\alpha_{0}=o$ we find $i\in\mathbb{N}$ with $\alpha_{i}\in S$ (as $S$ is open).
But then $\alpha_{i}\leq z$ and hence $z\notin T$. This implies
that $T\subseteq K$, so $\bigcup_{x\in T}\mathcal{U}_{x}$ is an
infinite open cover of $S$ which has no finite subcover. But by the
compactness of $\overline{(K,o)}$ the subset $S$ must be compact
as well. This leads to a contradiction. 

``$\text{(2)}\Rightarrow\text{(3)}$'' This implication is clear.

``$\text{(3)}\Rightarrow\text{(1)}$'' Assume that for every $x,y\in K$
the number of minimal elements in $\mathcal{U}_{x}\cap\mathcal{U}_{y}$
is finite and let $z\in\overline{(K,o)}$. Define the subset $\mathcal{S}:=\left\{ x\in K\mid x\leq z\right\} $
of $K$ and choose a (possibly finite) enumeration $y_{1},y_{2},...$
of $\mathcal{S}$. We inductively define elements $x_{1}\leq x_{2}\leq...$
in $\mathcal{S}$ with $y_{i}\leq x_{i+1}<z$ for all $i$. Set $x_{1}:=o\in\mathcal{S}$
and assume that for $i\in\mathbb{N}$ elements $x_{1},...,x_{i}\in\mathcal{S}$
with $o=x_{1}\leq...\leq x_{i}$ and $y_{j}\leq x_{j+1}<z$ for $j=1,...,i-1$
have been defined. The intersection $K\cap\mathcal{U}_{x_{i}}\cap\mathcal{U}_{y_{i}}$
must be non-empty and $\mathcal{U}_{x_{i}}\cap\mathcal{U}_{y_{i}}\subseteq\bigcup_{y\in\mathcal{U}_{x_{i}}\cap\mathcal{U}_{y_{i}}\text{ minimal}}\mathcal{U}_{y}$
where the union is finite by our assumption. We hence find $y\in\mathcal{U}_{x_{i}}\cap\mathcal{U}_{y_{i}}$
with $z\in\mathcal{U}_{y}$. Set $x_{i+1}:=y$, then this element
satisfies the condition $x_{1}\leq...\leq x_{i+1}$ and $y_{i}\leq x_{i+1}<z$.
Now, Lemma \ref{monotone} implies that the sequence $(x_{i})_{i\in\mathbb{N}}$
converges to a boundary point $z^{\prime}$ which can be represented
by an infinite geodesic path. By construction, for every $x\in\mathcal{S}$
one has $x\leq z^{\prime}$ and for $x\in K\setminus\mathcal{S}$
one has $x\nleq z^{\prime}$. Hence, $z^{\prime}=z$.

We have shown the equivalence of the statements (1), (2) and (3).
If $(K,o)$ satisfies one (hence all) of these conditions, it is clear
that $\partial(K,o)=\overline{(K,o)}\setminus K$. The claim follows.
\end{proof}

\begin{remark} \label{rootdependence}
\noindent Let $(K,o)$ be a connected rooted graph. One can show that the set of equivalence classes of infinite geodesic paths in $K$ does not depend on the choice of the root $o\in K$. Indeed, let $o^{\prime}\in K$ be a second root. Assume that $\alpha$, $\beta$ are infinite geodesic paths which are equivalent with respect to $o$. One finds $M\in\mathbb{N}$ such that for all $n\geq M$ there exist $k_{n},l_{n}\in\mathbb{N}$ with $\alpha_{n}\leq_{o}\beta_{k_{n}}\leq_{o}\alpha_{l_{n}}$. That in particular implies that we find a geodesic path starting in $\alpha_{n}$, passing $\beta_{k_{n}}$ and ending in $\alpha_{l_{n}}$. Denote this path by $\left[\alpha_{n},\alpha_{l_{n}}\right]$. Now, there is a geodesic path $\alpha^{\prime}$ starting in $o^{\prime}$ which eventually flows into $\alpha$, i.e. there exist $N\in\mathbb{N}$, $i\in\mathbb{Z}$ such that $\alpha_{n}^{\prime}=\alpha_{i+n}$ for all $n\geq N$ (see e.g. \cite[Lemma E.2]{BrownOzawa}). For $n\geq N-i$ write $\left[o^{\prime},\alpha_{n}\right]$ for the corresponding head of this path starting in $o^{\prime}$ and ending in $\alpha_{n}$. Then, for $n\geq\max\left\{ M,N+i\right\}$  we have that the concatenation $\left[o^{\prime},\alpha_{n}\right]\left[\alpha_{n},\alpha_{l_{n}}\right]$ is a geodesic path starting in $o^{\prime}$, passing $\alpha_{n}$, passing $\beta_{k_{n}}$ and ending in $\alpha_{l_{n}}$. We get that $\alpha_{n}\leq_{o^{\prime}}\beta_{k_{n}}\leq_{o^{\prime}}\alpha_{l_{n}}$ for all $n\geq\max\left\{ M,N+i\right\}$. It is then obvious that $\alpha$ and $\beta$ are equivalent with respect to $o^{\prime}$. However, even though the set of equivalence classes of infinite geodesic paths in $K$ does not depend on the choice of the root $o\in K$, the topology of $\overline{(K,o)}$ can; even in the setting of Proposition \ref{characterization}. Consider for instance the third graph in Figure \ref{Figure}. Then the limit of the indicated sequence $(x_i)_{i \in \mathbb{N}} \subseteq K$ depends on whether one views it as a sequence in $\overline{(K,x)}$ or as a sequence in $\overline{(K,y)}$. Note that the connected rooted graphs $(K,x)$ and $(K,y)$ satisfy the equivalent conditions in Proposition \ref{characterization} whereas $(K,o)$ does not.
\end{remark}

In general, the graph order of a connected rooted graph $(K,o)$ does not necessarily define a (complete) meet-semilattice. However, the graphs that we are mainly interested in in this paper satisfy this condition (see Example \ref{Example}). The meet-semilattice property has the technical advantage that $\mathcal{U}_{x}\cap\mathcal{U}_{y}=\mathcal{U}_{x\vee y}$, $P_{x}P_{y}=P_{x\vee y}$ if $\left\{ x,y\right\}$  has an upper bound and $\mathcal{U}_{x}\cap\mathcal{U}_{y}=\emptyset,P_{x}P_{y}=0$ else. This in particular implies that
\[
\ast\text{-Alg}\left(\left\{ P_{x}\mid x\in K\right\} \right)=\text{Span}\left(\left\{ P_{x}\mid x\in K\right\} \right)\text{.}
\]

\begin{example} \label{Example}
\emph{(a)} The graph orders of connected rooted trees define complete meet-semilattices.\\
\noindent \emph{(b)} Let $(W,S)$ be a Coxeter group, let $\text{Cay}(W,S)$ be the Cayley graph of $W$ with respect to the generating set $S$ and identify elements in $W$ with the corresponding vertices in the graph. It is easy to check that under this identification the graph order of the connected rooted graph $(\text{Cay}(W,S),e)$ coincides with the weak right Bruhat order of $(W,S)$. In particular, by  \cite[Theorem 3.2.1]{combinatorics}, the graph order of $(\text{Cay}(W,S),e)$ defines a complete meet-semilattice.
\end{example}

In combination with the discussion above Proposition \ref{characterization} immediately implies the following statement.

\begin{corollary} \label{cor}
Let $(K,o)$ be a connected rooted graph
whose graph order defines a complete meet-semilattice. Then $\partial(K,o)=\overline{(K,o)}\setminus K$
and every element in $\partial(K,o)$ is represented by an infinite
geodesic path in $K$.
\end{corollary}

\begin{theorem} \label{universal}
Let $(K,o)$ be a connected rooted graph for which the graph order defines a complete meet-semilattice. Then, $\mathcal{D}(K,o)$ is the universal C$^\ast$-algebra generated by projections $(P_{x})_{x\in K}$ with $P_{x}P_{y}=P_{x\vee y}$ for all $x,y\in K$ where we will by convention assume that $P_{x\vee y}=0$ if the join $x\vee y$ does not exist.
\end{theorem}

\begin{proof}
Let $\mathcal{A}$ be the universal C$^\ast$-algebra generated by projections $(\tilde{P}_{x})_{x\in K}$ with $\tilde{P}_{x}\tilde{P}_{y}=\tilde{P}_{x\vee y}$ for all $x,y\in K$ and let $\chi$ be a character on $\mathcal{A}$. It suffices to show that the map $P_{x}\mapsto\chi(\tilde{P}_{x})$ defines a character on $\mathcal{D}(K,o)$. As in the proof of Proposition \ref{C*}, define the set $\mathcal{S}:=\{ x\in K\mid\chi(\tilde{P}_{x})=1\}$. For $i\in\mathbb{N}$ define $x_{i}:=\bigvee_{x\in\mathcal{S}:d_{K}(x,o)\leq i}x$ and let $z\in \overline{(K,o)}$ be the point this sequence converges to. Further, let $\psi$ be the homeomorphism appearing in the proof of Proposition \ref{C*}. Then, $(\psi(z))(P_{x})=1=\chi(\tilde{P}_{x})$ if $x\in\mathcal{S}$ and $(\psi(z))(P_{x})=0=\chi(\tilde{P}_{x})$ if $x\notin\mathcal{S}$. The claim follows.
\end{proof}

We finish this subsection with three spaces which arise as special cases from our construction and which serve as additional motivation.

\begin{example}
\emph{(a)} Let $S$ be a countable
(discrete) set and let $\left\{ \bullet\right\} $ be the one-point
set. Define a graph $K=(V,E)$ via $V:=S\cup\left\{ \bullet\right\} $
and $E:=\{(\bullet,s)\mid s\in S\}\cup\{(s,\bullet)\mid s\in S\}$.
Then, by Proposition \ref{C*}, $\overline{(K,\bullet)}$ is a compact Hausdorff
space. It is easy to check that it identifies with the one-point compactification
of the discrete set $S$.

\noindent \emph{(b)} Let $\mathbf{n}:=(n_{i})_{i\in\mathbb{N}}\subseteq\mathbb{N}$
be a sequence of natural numbers with $n_{i}>1$, define $N_{0}:=1$,
$N_{1}:=n_{1}$, $N_{2}:=n_{1}n_{2}$, ... and consider the set $\mathcal{G}_{\mathbf{n}}$
consisting of all formal sums of the form $\mathbf{x}:=\sum_{i=1}^{\infty}x_{i}N_{i-1}$
with $x_{i}\in\{0,1,...,n_{i}-1\}$ where only finitely many of the
coefficients $x_{i}$ are non-zero. Set $0:=\sum_{i=1}^{\infty}0N_{i-1}$.
The set $\mathcal{G}_{\mathbf{n}}$ induces a locally finite graph
$K=(V,E)$ via $V:=\mathcal{G}_{\mathbf{n}}$ and where distinct vertices
$\mathbf{x}=\sum_{i=1}^{\infty}x_{i}N_{i-1}$ and $\mathbf{y}=\sum_{i=1}^{\infty}y_{i}N_{i-1}$
are adjacent to each other if and only if there exists $L\in\mathbb{N}$
such that 
\[
x_{1}=y_{1},...,x_{L}=y_{L}\text{, }x_{L+1}=0\neq y_{L+1}\text{ and }x_{L+2}=y_{L+2}=x_{L+3}=y_{L+3}=...=0
\]
or 
\[
x_{1}=y_{1},...,x_{L}=y_{L}\text{, }x_{L+1}\neq0=y_{L+1}\text{ and }x_{L+2}=y_{L+2}=x_{L+3}=y_{L+3}=...=0\text{.}
\]
The graph order of $(K,0)$ gives rise to a complete meet-semilattice.
Hence, by Proposition \ref{C*} the compactification
$\overline{\mathcal{G_{\mathbf{n}}}}:=\overline{(K,0)}$ is a metrizable
compact space. It identifies with the Cantor set $\prod_{i=1}^{\infty}\mathbb{Z}/n_{i}\mathbb{Z}$
where the product is equipped with the product topology and where
the $\mathbb{Z}/n_{i}\mathbb{Z}$ are viewed as discrete topological
spaces. Note that the map $\mathcal{G}_{\mathbf{n}}\rightarrow\mathcal{G}_{\mathbf{n}}$,
$\mathbf{x}\mapsto\mathbf{x}+1N_{0}$ (addition with carryover) induces a homeomorphism of $\overline{\mathcal{G}_{\mathbf{n}}}$.
The induced action of $\mathbb{Z}$ on $\overline{\mathcal{G_{\mathbf{n}}}}$
is the well-studied \emph{odometer action} with respect to $\mathbf{n}$ and the corresponding crossed product
$C(\overline{\mathcal{G}_{\mathbf{n}}})\rtimes_{r}\mathbb{Z}$ is
the \emph{Bunce-Deddens algebra}
with respect to $\mathbf{n}$ (see \cite{BunceDeddens}
and \cite[Sections VIII.4 and V.3]{Davidson}).
In the case where $n_{i}$ is the $i$-th prime number, $\overline{\mathcal{G}_{\mathbf{n}}}$
identifies with the well-known  \emph{profinite completion}
of the integers.

\noindent \emph{(c)} The positive integers $\mathbb{N}_{\geq1}$ induce a locally
finite graph $K=(V,E)$ with $V:=\mathbb{N}_{\geq1}$ and
\[
E:=\{(m,n)\in\mathbb{N}_{\geq1}\times\mathbb{N}_{\geq1}\mid m=pn\text{ or }n=pm\text{ for a prime number }p\}\text{.}
\]
In the graph order with respect to the root $o:=1$ one has $m\leq n$
for $m,n\in\mathbb{N}_{\geq1}$ if and only if $m\mid n$. It hence
defines a complete meet-semilattice where the meet of a subset $T\subseteq K$
is given by the greatest common divisor $\text{gcd}(T)$. By Theorem
\ref{C*} the corresponding compactification $\overline{(\mathbb{N}_{\geq1},1)}$
is a metrizable compact space and by Proposition \ref{characterization} every element in $\overline{(\mathbb{N}_{\geq1},1)}$ can be represented by a (possibly infinite) geodesic path. The space can hence be viewed as a compactification
of the set of \emph{supernatural numbers}. Further, the semigroup action of $\mathbb{N}_{\geq1}$ on itself by
multiplication extends to a continuous semigroup action on $\overline{(\mathbb{N}_{\geq1},1)}$.
\end{example}


\subsection{Hyperbolic graphs and trees}

In the following subsection we will see that for a hyperbolic connected rooted graph $(K,o)$ the topological spaces $\partial(K,o)$ and $\overline{(K,o)}$ behave well with respect to the hyperbolic (Gromov) boundary $\partial_{h}K$ and the corresponding compactification $K\cup\partial_{h}K$ of $K$. In the case where $K$ is a tree, both spaces turn out to be homeomorphic to each other.

\begin{theorem} \label{hyperbolic}
Let $(K,o)$ be a hyperbolic connected rooted graph. Then the map $\phi:\partial(K,o)\rightarrow\partial_{h}K$ given by $\phi\left(\left[\mathbf{x}\right]\right)=\left[\mathbf{x}\right]_{h}$ for a sequence $\mathbf{x}$ which $o$-converges to infinity is well-defined, continuous and surjective. If the graph is locally finite, then $\phi$ extends to a continuous surjection $\tilde{\phi}:\overline{(K,o)}\rightarrow K\cup\partial_{h}K$ with $\tilde{\phi}|_{K}=\text{id}_{K}$.
\end{theorem}

\begin{proof}
\emph{Well-defined:} Let $\mathbf{x}$, $\mathbf{y}$ be equivalent sequences which $o$-converge to infinity. Assume that $\mathbf{x} \nsim_{h}\mathbf{y}$. Then, by definition, $\liminf_{m,n\rightarrow\infty}\langle x_{m},y_{n}\rangle_{o}<\infty$, i.e. there exist strictly increasing sequences $(m_{i})_{i\in\mathbb{N}}, (n_{i})_{i\in\mathbb{N}}\subseteq\mathbb{N}$ such that $\lim_{i\rightarrow\infty}\langle x_{m_{i}},y_{n_{i}}\rangle_{o}<\infty$. Since $\mathbf{x}$ and $\mathbf{y}$ are equivalent, for every $i\in\mathbb{N}$ there exists an element $z_{i}\in K$ such that $z_{i}\leq x_{m_{i}}$, $z_{i}\leq y_{n_{i}}$. Further, we can choose $z_{i}$ in such a way that $d_{K}(z_{i},o)\rightarrow\infty$. This implies
\begin{eqnarray} \nonumber
\left\langle x_{m_{i}},y_{n_{i}}\right\rangle _{o}&=&\frac{1}{2}\left(d_{K}(x_{m_{i}},o)+d_{K}(y_{n_{i}},o)-d_{K}(x_{m_{i}},y_{n_{i}})\right)\\
\nonumber
&\geq& \frac{1}{2}\left(d_{K}(x_{m_{i}},o)+d_{K}(y_{n_{i}},o)-(d_{K}(z_{i},x_{m_{i}})+d_{K}(z_{i},y_{n_{i}}))\right)\\
\nonumber
&=& \frac{1}{2}\left(d_{K}(x_{m_{i}},o)+d_{K}(y_{n_{i}},o)-(d_{K}(x_{m_{i}},o)-d_{K}(z_{i},o)+d_{K}(y_{n_{i}},o)-d_{K}(z_{i},o))\right)\\
\nonumber
&=& d_{K}(z_{i},o)\\
\nonumber
&\rightarrow& \infty
\end{eqnarray}
in contradiction to our assumption. Hence, $\mathbf{x}$ and $\mathbf{y}$ must have been equivalent. In particular, the map $\phi$ is well-defined.

\emph{Continuity:} Let $(\left[\mathbf{x}^{i}\right])_{i\in I}\subseteq\partial(K,o)$ be a net of equivalence classes of sequences $\mathbf{x}^{i}$, $i\in I$ which $o$-converge to infinity that converges to a point $z\in\partial(K,o)$. Let $\mathbf{x}$ be a sequence which $o$-converges to infinity and which represents $z$. We claim that $\left[\mathbf{x}^{i}\right]_{h}\rightarrow \phi(z)=\left[\mathbf{x}\right]_{h}$. As the sets $\left\{ U\left(\phi(z),R\right)\right\} _{R>0}$ with
\begin{eqnarray}
\nonumber
U\left(\phi(z),R\right):=\{z^{\prime}\in \partial_h K & \mid &\text{there are sequences }\mathbf{y}^{1},\mathbf{y}^{2}\text{ converging to infinity }\\
\nonumber
& & \text{ with } \phi(z)=\left[\mathbf{y}^{1}\right]_{h}\text{, }z^{\prime}=\left[\mathbf{y}^{2}\right]_{h}\text{ and }\liminf_{i,j\rightarrow\infty}\left\langle y_{i}^{1},y_{j}^{2}\right\rangle _{o}>R\}
\end{eqnarray}
define a neighborhood basis of $\phi(z)$, it suffices to show that for every $R>0$, $\left[\mathbf{x}^{i}\right]_{h}\in U(\phi(z),R)$ for $i$ large enough. For $R>0$ we find $y_{R}\in K$ with $y_{R}\leq z$ and $d_{K}(y_{R},o)>R$. Further, as $\left[\mathbf{x}^{i}\right]\rightarrow z$, there exists $i_{0}(R)\in I$ such that $y_{R}\leq\left[\mathbf{x}^{i}\right]$ for every $i\geq i_{0}(R)$. We claim that $\left[\mathbf{x}^{i}\right]_{h}\in U(\phi(z),R)$ for every $i\geq i_{0}(R)$. Assume that this is not the case. Then, as above, for fixed $i\geq i_{0}(R)$ we find strictly increasing sequences $(m_{j})_{j\in\mathbb{N}}, (n_{j})_{j\in\mathbb{N}}\subseteq\mathbb{N}$ such that $\lim_{j\rightarrow\infty}\langle x_{m_{j}}, x_{n_{j}}^{i}\rangle_{o}\leq R$. Without loss of generality we can assume that $x_{R}\leq x_{m_{j}}, x_{n_{j}}^{i}$ for every $j\in\mathbb{N}$. Then,
\begin{eqnarray}
\nonumber
\langle x_{m_{j}},x_{n_{j}}^{i}\rangle_{o} &=& \frac{1}{2}\left(d_{K}(x_{m_{j}},o)+d_{K}(x_{n_{j}}^{i},o)-d_{K}(x_{m_{j}},x_{n_{j}}^{i})\right)\\
\nonumber
&\geq& \frac{1}{2}\left(d_{K}(x_{m_{j}},o)+d_{K}(x_{n_{j}}^{i},o)-(d_{K}(y_{R},x_{m_{j}})+d_{K}(y_{R},x_{n_{j}}^{i}))\right)\\
\nonumber
&=& d_{K}\left(y_{R},o\right)\\
\nonumber
&>& R
\end{eqnarray}
in contradiction to our assumption. This implies that $\left[x^{i}\right]_{h}\in U(\phi(z),R)$ for $i\geq i_{0}(R)$, so $\left[x^{i}\right]_{h}\rightarrow \phi(z)=\left[\mathbf{x}\right]_{h}$.

\emph{Surjectivity:} That is clear.

We have shown that the map $\phi$ is well-defined, continuous and surjective. If the graph is locally finite, $K$ is an open subset of $\overline{(K,o)}$. Using this, one checks in the same way as above that the identity map on $K$ continuously extends to a surjection $\tilde{\phi}:\overline{(K,o)}\rightarrow K\cup\partial_{h}K$ with $\tilde{\phi}|_{\partial(K,o)}=\phi$.
\end{proof}

In the case of a tree Theorem \ref{hyperbolic} can be strengthened.

\begin{corollary} \label{homeomorphism}
Let $(\mathcal{T},o)$ be a connected rooted tree. Then the identity on $\mathcal{T}$ extends to a homeomorphism $\partial(\mathcal{T},o)\cong\partial_{h}\mathcal{T}$.
\end{corollary}

\begin{proof}
It suffices to show that the map $\left[\mathbf{x}\right]\mapsto\left[\mathbf{x}\right]_{h}$ is injective. By Proposition \ref{characterization} it further suffices to consider equivalence classes of infinite geodesic rays. So let $\alpha$, $\beta$ be infinite geodesic paths with $\left[\alpha\right]_{h}=\left[\beta\right]_{h}$, i.e. $\sup_{i}d_{\mathcal{T}}(\alpha_{i},\beta_{i})<\infty$. Then, since $\mathcal{T}$ is a tree, $\alpha$ and $\beta$ must eventually flow together which implies that $\left[\alpha\right]=\left[\beta\right]$.
\end{proof}

Besides from Corollary \ref{homeomorphism} the compactification of trees has another useful property.

\begin{lemma} \label{trees}
Let $(\mathcal{T},o)$ be a rooted tree. Then, every element $z\in\partial(\mathcal{T},o)$ is \emph{maximal} in the sense that if $z^{\prime}\in\partial(\mathcal{T},o)$ is another element with $z\leq z^{\prime}$ (in the partial order from Lemma \ref{poset}), then $z=z^{\prime}$.
\end{lemma}

\begin{proof}
Let $z,z^{\prime}\in\partial(\mathcal{T},o)$ be elements with $z\leq z^{\prime}$. By Proposition \ref{characterization} there exist infinite geodesic paths $\alpha$, $\beta$ which represent $z$, $z^{\prime}$. Assume that $\alpha_{0}=\beta_{0}=o$. We have $\alpha_{1}\leq z$ and hence $\alpha_{1}\leq z^{\prime}$. Since geodesic paths between two points of a tree are unique, $\beta$ passes $\alpha_{1}$ and hence $\alpha_{1}=\beta_{1}$. By the same argument we get $\alpha_{2}=\beta_{2}$, $\alpha_{3}=\beta_{3}$, ..., therefore $z=z^{\prime}$.
\end{proof}


\section{Boundaries of Coxeter groups} \label{section 3}

The most important graphs that we consider in this paper are Cayley graphs of Coxeter systems. Even though some of the results hold in greater generality we restrict to finite rank Coxeter groups to avoid technical subtleties and to keep the statements consistent with each other.

\begin{definition}
Let $(W,S)$ be a finite rank Coxeter system. As before, let $\text{Cay}(W,S)$ be the Cayley graph of $W$ with respect to the generating set $S$ and view it as a rooted graph with root $e\in W$. We call $\partial(W,S):=\partial(\text{Cay}(W,S),e)$ the \emph{boundary of $\left(W,S\right)$} and $\overline{(W,S)}:=\overline{(\text{Cay}(W,S),e)}$ the \emph{compactification of $\left(W,S\right)$}. For convenience, we will often write $\partial W$ and $\overline{W}$ if the generating set $S$ is clear.
\end{definition}

By what we have seen so far, the spaces $\partial(W,S)$, $\overline{(W,S)}$ are metrizable compact spaces and $W\subseteq\overline{(W,S)}$ is both dense and discrete with $\partial(W,S)=\overline{(W,S)}\setminus W$. Further, by Corollary \ref{cor} every element in $\partial(W,S)$ is represented by an infinite geodesic path. In the following we will make use of these facts without any further mention.


\subsection{Left actions of Coxeter groups on their compactification} \label{2.1}

\begin{proposition}[{\cite[Proposition 3.1.2 (vi)]{combinatorics}}] \label{help}
Let $\left(W,S\right)$ be a Coxeter system, $\mathbf{v},\mathbf{w}\in W$ and $s\in S$ with $s\leq\mathbf{v}$, $s\leq\mathbf{w}$. Then, $\mathbf{v}\leq\mathbf{w}$ if and only if $s\mathbf{v}\leq s\mathbf{w}$.
\end{proposition}

\begin{theorem} \label{coxeteraction}
Let $(W,S)$ be a finite rank Coxeter system. Then, the action $W\curvearrowright W$ by left multiplication extends to a continuous action $W\curvearrowright\overline{(W,S)}$ with $W.(\partial(W,S))=\partial(W,S)$.
\end{theorem}

\begin{proof}
If suffices to show that for every $s\in S$ the map $\mathbf{w}\mapsto s\mathbf{w}$ continuously extends to the boundary. First, let $\alpha$, $\beta$ be equivalent infinite geodesic paths. It is clear that $(s.\alpha_{n})_{n\in\mathbb{N}}$, $(s.\beta_{n})_{n\in\mathbb{N}}$ are infinite geodesic paths as well, hence the elements $\left[s.\alpha\right]:=\left[(s.\alpha_{n})_{n\in\mathbb{N}}\right]$ and $\left[s.\beta\right]:=\left[(s.\beta_{n})_{n\in\mathbb{N}}\right]\in\partial W$ are well-defined. Without loss of generality we can assume that $\alpha_{0}=\beta_{0}=e$. Then, for every $n\in\mathbb{N}$ there exist minimal $k_{n},l_{n}\in\mathbb{N}$ with $\alpha_{n}\leq\beta_{k_{n}}$ and $\beta_{n}\leq\alpha_{l_{n}}$.

\begin{itemize}
\item \emph{Case 1:} Assume that $s\leq\left[\alpha\right]=\left[\beta\right]$ and let $\mathbf{v}\leq\left[s.\alpha\right]$. Then there exists $N\in\mathbb{N}$ such that $s\leq\alpha_{n}$ and $\mathbf{v}\leq s\alpha_{n}$ for all $n\geq N$. Since then $s\leq\alpha_{n}\leq\beta_{k_{n}}$, Proposition \ref{help} implies that $\mathbf{v}\leq s\alpha_{n}\leq s\beta_{k_{n}}\leq\left[s.\beta\right]$ for all $n\geq N$. The same argument can be used to show that if $\mathbf{v}\leq\left[s.\beta\right]$, then $\mathbf{v}\leq\left[s.\alpha\right]$ which implies that $\left[s.\alpha\right]=\left[s.\beta\right]$.
\item \emph{Case 2:} Assume that $s\nleq\left[\alpha\right]=\left[\beta\right]$ and let $\mathbf{v}\leq s.\left[\alpha\right]$. Then, there exists $N\in\mathbb{N}$ such that $s\nleq\alpha_{n}$, $s\nleq\beta_{k_{n}}$ and $\mathbf{v}\leq s\alpha_{n}$ for all $n\geq N$. Again, an application of Proposition \ref{help} to $s\alpha_n$ and $s \beta_{k_n}$ implies that $\mathbf{v}\leq s\alpha_{n}\leq s\beta_{k_{n}}\leq\left[s.\beta\right]$ for all $n\geq N$. The same argument implies that if $\mathbf{v}\leq\left[s.\alpha\right]$, then $\mathbf{v}\leq\left[s.\beta\right]$. We get $\left[s.\alpha\right]=\left[s.\beta\right]$.
\end{itemize}

We have shown that the map $\mathbf{w}\mapsto s\mathbf{w}$ extends to the boundary via $s.\left[\alpha\right]:=\left[(s.\alpha_{n})_{n\in\mathbb{N}}\right]\in\partial W$ for $\left[\alpha\right]\in\partial W$. It remains to show that the extension is continuous. Since $\overline{W}$ is metrizable, it suffices to consider sequences. Let $(z^{i})_{i\in\mathbb{N}}\subseteq\overline{W}$ be a sequence converging to a boundary point $z\in\partial W$ and let $\alpha^{i}$ (resp. $\alpha$) be (possibly finite) geodesic paths representing $z^{i}$ (resp. $z$). Again, we can assume that $\alpha_{0}^{i}=\alpha_{0}=e$.

\begin{itemize}
\item \emph{Case 1:} Assume that $s\leq z=\left[\alpha\right]\in\partial W$ and let $\mathbf{v}\in W$ with $\mathbf{v}\leq s.z$. There exists $N\in\mathbb{N}$ with $s\leq\alpha_{n}$ and $\mathbf{v}\leq s\alpha_{n}$ for all $n\geq N$. Further, for $n\geq N$ there exists $i_{0}\left(n\right)\in\mathbb{N}$ with $s\leq\alpha_{n}\leq z^{i}$ for $i\geq i_{0}\left(n\right)$. Proposition \ref{help} implies that $\mathbf{v}\leq s\alpha_{n}\leq s.z^{i}$ for all $n\geq N$, $i\geq i_{0}\left(n\right)$, so in particular $\mathbf{v}\leq s.z^{i}$ for all $i\geq i_{0}\left(N\right)$. Now, let $\mathbf{v}\in W$ with $\mathbf{v}\nleq s.z$. We have to show that $\mathbf{v}\nleq s.z^{i}$ for $i$ large enough. Assume without loss of generality that $\mathbf{v}\leq s.z^{i}$ for all $i\in\mathbb{N}$. There exists $i_{0}\in\mathbb{N}$ with $s\leq z^{i}$ for all $i\geq i_{0}$ and hence $s\nleq\mathbf{v}$. We get with Proposition \ref{help} that $s\mathbf{v}\leq z^{i}$ for all $i\geq i_{0}$. But $z^{i}\rightarrow z$, so $s\mathbf{v}\leq z$ as well. Again, using Proposition \ref{help} we get $\mathbf{v}\leq s.z$ in contradiction to our choice of $\mathbf{v}$. This implies that $s.z^{i}\rightarrow s.z$.
\end{itemize}

We have hence shown that if $z\in\partial W$ with $s\leq z$, then $s.z^{i}\rightarrow s.z$ for every sequence $(z^{i})_{i\in\mathbb{N}}\subseteq\overline{W}$ with $z^{i}\rightarrow z$.

\begin{itemize}
\item \emph{Case 2:} Assume that $s\nleq z=\left[\alpha\right]\in\partial W$ and let $\mathbf{v}\in W$ with $\mathbf{v}\leq s.z$. There exists $N\in\mathbb{N}$ with $s\nleq\alpha_{n}$ and $\mathbf{v}\leq s\alpha_{n}$ for all $n\geq N$. Further, for $n\in\mathbb{N}$ there exists $i_{0}\left(n\right)\in\mathbb{N}$ with $\alpha_{n}\leq z^{i}$ and $s\nleq z^{i}$ for all $i\geq i_{0}\left(n\right)$. Proposition \ref{help} implies that $\mathbf{v}\leq s\alpha_{n}\leq s.z^{i}$ for all $n\geq N$, $i\geq i_{0}\left(n\right)$, so in particular $\mathbf{v}\leq s.z^{i}$ for all $i\geq i_{0}\left(N\right)$. Now, let $\mathbf{v}\in W$ with $\mathbf{v}\nleq s.z$. Again, we have to show that $\mathbf{v}\nleq s.z^{i}$ for $i$ large enough. Assume without loss of generality that $\mathbf{v}\leq s.z^{i}$ for all $i\in\mathbb{N}$. Since $\overline{W}$ is (sequentially) compact, we find a subsequence $(s.z^{i_{k}})_{k\in\mathbb{N}}$ of $(s.z^{i})_{i\in\mathbb{N}}$ converging to a boundary point $z^{\prime}\in\partial W$. Then, $s\leq z^{\prime}$ and $\mathbf{v}\leq z^{\prime}$. By what we have shown in \emph{Case 1}, we get that $z^{i_{k}}\rightarrow s.z^{\prime}$ which implies $s.z^{\prime}=z$. But then $\mathbf{v}\leq z^{\prime}=s.z$ in contradiction to our choice of $\mathbf{v}$. This implies that $s.z^{i}\rightarrow s.z$.
\end{itemize}

The claim follows.
\end{proof}

An immediate implication of Theorem \ref{hyperbolic} is the following.

\begin{corollary} \label{hyperbolic 2}
Let $\left(W,S\right)$ be a word hyperbolic Coxeter system. Then the map $\tilde{\phi} :\overline{\left(W,S\right)}\rightarrow W\cup\partial_{h}W$ given by $\tilde{\phi}\left(\mathbf{w}\right)=\mathbf{w}$ for $\mathbf{w}\in W$ and $\tilde{\phi}\left(\left[\alpha\right]\right)=\left[\alpha\right]_{h}$ for an infinite geodesic path $\alpha$ is well-defined, continuous, $W$-equivariant and surjective with $\tilde{\phi} (\partial (W,S)) = \partial_h W$.
\end{corollary}

In particular, in the setting of Corollary \ref{hyperbolic 2}, the action of $W$ on the compactification $\overline{W}$ is (topologically) amenable. But we can do better, as we will see in the following subsection.


\subsection{Combinatorial compactifications and horofunction compactifications} \label{2.2}

The author is grateful to Sven Raum and Adam Skalski, who contacted
him after receiving an earlier draft of this paper to point out that
the space $\overline{(W,S)}$ coincides with Caprace-Lécureux's minimal
combinatorial compactification associated with the Coxeter complex
of the system $(W,S)$ (see \cite{Caprace}, \cite{Lecureux})
and hence relates to Lam-Thomas' results in \cite{Lam}.
Let us elaborate on this, for details see \cite{Caprace}.
Let $X$ be a locally finite building of type $(W,S)$ with chamber
set $\text{Ch}(X)$ and denote the corresponding set of spherical
residues by $\text{Res}_{\text{sph}}(X)$. Given a spherical residue
$\sigma\in\text{Res}_{\text{sph}}(X)$ the associated \emph{combinatorial
projection} $\text{proj}_{\sigma}:\text{Ch}(X)\rightarrow\text{St}(\sigma)$
associates to a chamber $C$ the chamber of the star $\text{St}(\sigma)$
of $\sigma$ (i.e. the set of all residues containing $\sigma$ in
their boundaries) closest to $C$. It may be extended to a map on
the set of all (spherical) residues of $X$ and hence induces a map
\[
\pi_{\text{Res}}:\text{Res}_{\text{sph}}(X)\rightarrow\prod_{\sigma\in\text{Res}_{\text{sph}}(X)}\text{St}(\sigma),R\mapsto(\text{proj}_{\sigma}(R))_{\sigma\in\text{Res}_{\text{sph}}(X)}.
\]
Equip $\prod_{\sigma\in\text{Res}_{\text{sph}}(X)}\text{St}(\sigma)$
with the product topology where each star is discrete. Then the \emph{minimal
combinatorial compactification $\mathcal{C}_{1}(X)$} of $X$ can
be defined as the closure $\mathcal{C}_{1}(X):=\overline{\pi_{\text{Res}}(\text{Ch}(X))}$
(see \cite[Proposition 2.12]{Caprace}) and the
\emph{maximal combinatorial compactification} of $X$ is $\mathcal{C}_{\text{sph}}(X):=\overline{\pi_{\text{Res}}(\text{Res}_{\text{sph}}(X))}$.
In particular, $\mathcal{C}_{1}(X)$ is a closed subset of $\mathcal{C}_{\text{sph}}(X)$.
One can show that the $\text{Aut}(X)$-action on $X$ extends in a
canonical way to continuous actions on $\mathcal{C}_{1}(X)$ and $\mathcal{C}_{\text{sph}}(X)$.

The following theorem builds a connection between Caprace-Lécureux's
combinatorial compactifications and our construction. It holds in
greater generality, but we restrict to compactifications of finite
rank Coxeter systems. The theorem states that for a finite rank Coxeter
system $(W,S)$ and the Coxeter complex $\Sigma$ Caprace-Lécureux's
minimal combinatorial compactification $\mathcal{C}_{1}(W,S):=\mathcal{C}_{1}(\Sigma)$
coincides with the space $\overline{(W,S)}$. Its proof is based on
the characterization \cite[Theorem 3.1]{Caprace} of $\mathcal{C}_{1}(W,S)$ as the horofunction
compactification of the chamber graph (i.e. the set of chambers with
the gallery distance) of the locally finite building $\Sigma$. Recall
that the chamber graph of $\Sigma$ is just the Cayley graph $\text{Cay}(W,S)$
with the usual metric.

The horofunction compactificiation is constructed as follows. Following
\cite[Chapter II.8]{BridsonHaefliger}, let $(Y,d)$ be a metric space and consider the space $C(Y)$
of continuous functions on $Y$ equipped with the topology of uniform
convergence on bounded sets. Given a base point $y_{0}\in Y$ define
the subspace $C(Y,y_{0}):=\{f\in C(Y)\mid f(y_{0})=0\}$. It is homeomorphic
to the quotient $C_{\ast}(Y)$ of $C(Y)$ by the 1-dimensional subspace
of constant functions, so in particular $C(Y,y_{0})$ does not depend
on the choice of $y_{0}\in Y$. The space $Y$ (continuously and injectively)
embeds into $C(Y,y_{0})$ via $y\mapsto f_{y}:=d(y,\cdot)-d(y,y_{0})$.
We can hence view $Y$ as a subspace of $C(Y,y_{0})$. The closure
of $Y$ in $C(Y,y_{0})$ is then denoted by $\widehat{Y}$. If $Y$
is proper, $\widehat{Y}$ is a compact Hausdorff space (see \cite[Exercise 8.15]{BridsonHaefliger}) which
is called the \emph{horofunction compactification} of $Y$.

Note that the chamber graph $\text{Cay}(W,S)$ of $\Sigma$ is a proper
metric space and that in this case the topology of uniform convergence
on bounded sets on $C(W,e):=C(\text{Cay}(W,S),e)$ coincides with
the topology of pointwise convergence. By \cite[Theorem
3.1]{Caprace} the minimal combinatorial compactification $\mathcal{C}_{1}(W,S)$
is $\text{Aut}(\Sigma)$-equivariantly homeomorphic to the horofunction
compactification of $\text{Cay}(W,S)$ via $\pi_{\text{Res}}(\text{Ch}(\Sigma))\ni\pi_{\text{Res}}(\mathbf{w})\mapsto f_{\mathbf{w}}=\left|\mathbf{w}^{-1}(\cdot)\right|-\left|\mathbf{w}\right|$
where $\mathbf{w}\in\text{Ch}(\Sigma)=W$.

\begin{theorem} \label{CombinatorialCompactification}
Let $(W,S)$ be a finite rank Coxeter system. Then
the map
\[
W\rightarrow C(W,e), \mathbf{w}\mapsto f_{\mathbf{w}}:=\left|\mathbf{w}^{-1}(\cdot)\right|-\left|\mathbf{w}\right|
\]
extends to a $W$-equivariant homeomorphism between $\overline{(W,S)}$
and $\widehat{\text{Cay}(W,S)}$. In particular, $\overline{(W,S)}$
is $W$-equivariantly homeomorphic to the minimal combinatorial compactification
$\mathcal{C}_{1}(W,S)$.
\end{theorem}

\begin{proof}
By the compactness of $\overline{W}$ it suffices
to show that the $W$-equivariant map $W\rightarrow C(W,e)$, $\mathbf{w}\mapsto f_{\mathbf{w}}:=\left|\mathbf{w}^{-1}(\cdot)\right|-\left|\mathbf{w}\right|$
extends to a well-defined, bijective and continuous map $\phi:\overline{W}\rightarrow\widehat{\text{Cay}(W,S)}$
via $(\phi(z))(\mathbf{v}):=\lim_{i}\left(\left|\alpha_{i}^{-1}\mathbf{v}\right|-\left|\alpha_{i}\right|\right)$,
where $\mathbf{v}\in W$ and $\alpha$ is a (possibly finite) geodesic
path representing $z\in\overline{W}$.

\emph{Well-defined:} Let $\alpha$ and $\beta$ be equivalent infinite
geodesic paths and $\mathbf{v}\in W$ with reduced expression $\mathbf{v}=s_{1}...s_{n}$.
We have that 
\begin{eqnarray} \label{reduction}
\left|\alpha_{i}^{-1}\mathbf{v}\right|-\left|\alpha_{i}\right| = \sum_{j=0}^{n-1}\left(\left|\left(s_{j}...s_{1}\alpha_{i}\right)^{-1}s_{j+1}\right|-\left|\left(s_{j}...s_{1}\alpha_{i}\right)^{-1}\right|\right)\text{.}
\end{eqnarray}
Theorem \ref{coxeteraction} implies that for every
$j=0,...,n-1$ the sequences $(s_{j}...s_{1}\alpha_{i})_{i\in\mathbb{N}}$
and $(s_{j}...s_{1}\beta_{i})_{i\in\mathbb{N}}$ are equivalent infinite
geodesic paths. By \eqref{reduction} it hence suffices to show that $\lim_{i}\left(\left|\alpha_{i}^{-1}s\right|-\left|\alpha_{i}\right|\right)=\lim_{i}\left(\left|\beta_{i}^{-1}s\right|-\left|\beta_{i}\right|\right)$
for all $s\in S$. Since $\alpha$ and $\beta$ are equivalent we either have $s\leq\alpha_{i},\beta_{i}$ for
$i$ large enough or $s\nleq\alpha_{i},\beta_{i}$ for $i$ large
enough. In the first case, 
\[
\lim_{i\rightarrow\infty}\left(\left|\alpha_{i}^{-1}s\right|-\left|\alpha_{i}\right|\right)=1=\lim_{i\rightarrow\infty}\left(\left|\beta_{i}^{-1}s\right|-\left|\beta_{i}\right|\right)
\]
and in the second one
\[
\lim_{i\rightarrow\infty}\left(\left|\alpha_{i}^{-1}s\right|-\left|\alpha_{i}\right|\right)=(-1)=\lim_{i\rightarrow\infty}\left(\left|\beta_{i}^{-1}s\right|-\left|\beta_{i}\right|\right)\text{.}
\]
We get that $\phi$ is indeed well-defined.

\emph{Continuity:} Let $\mathbf{v}\in W$ with reduced expression
$\mathbf{v}=s_{1}...s_{n}$. The equality \eqref{reduction} implies that for
every $z\in\overline{W}$, $(\phi(z))(\mathbf{v})=\sum_{j=0}^{n-1}\phi(s_{j}...s_{1}.z)(s_{j+1})\text{.}$
It hence suffices to show that for every $s\in S$ and every sequence
$(z^{i})_{i\in\mathbb{N}}\subseteq\overline{W}$ converging to a point
$z\in\overline{W}$ the equality $(\phi(z))(s)=\lim_{i}(\phi(z^{i}))(s)$
holds. A straightforward modification of the argument above implies
the desired statement.

\emph{Surjectivity:} The surjectivity is clear.

\emph{Injectivity:} Let $z,z^{\prime}\in\overline{W}$ with $z\neq z^{\prime}$.
Then there exists $\mathbf{v}\in W$ with $\mathbf{v}\leq z$ but
$\mathbf{v}\nleq z^\prime$. Let $\alpha$ be a (possibly finite) geodesic
path representing $z$ and $\beta$ a (possibly finite) geodesic path
representing $z^\prime$ with $\alpha_{0}=\beta_{0}=e$. Then, 
\[
(\phi(z))(\mathbf{v})=\lim_{i\rightarrow\infty}\left(\left|\alpha_{i}^{-1}\mathbf{v}\right|-\left|\alpha_{i}\right|\right)=-\left|\mathbf{v}\right|\neq\lim_{i\rightarrow\infty}\left(\left|\beta_{i}^{-1}\mathbf{v}\right|-\left|\beta_{i}\right|\right)=(\phi(z^{\prime}))(\mathbf{v})
\]
and hence $\phi(z)\neq\phi(z^{\prime})$. The claim follows.
\end{proof}

\begin{remark}
In general it is not true that for a connected rooted graph
$(K,o)$ the map $K\rightarrow C(K,o)$, $x\mapsto f_{x}:=d_{K}(x,\cdot)-d_{K}(x,o)$
extends to a homeomorphism between $\overline{(K,o)}$ and $\widehat{K}$;
not even in the locally finite case or in the setting of Proposition \ref{characterization}. Indeed, as mentioned above
the horofunction compactification does not depend on the choice of
the base point whereas $\overline{(K,o)}$ can depend on the choice
of the root $o$ (see for instance Remark \ref{rootdependence}).
\end{remark}

\begin{definition}
An action of a discrete group $G$ on a compact
space $Y$ is called \emph{amenable} if there exists a net of continuous
maps $m_{i}:Y\rightarrow\text{Prob}\left(G\right)$ such that for
each $g\in G$ 
\[
\lim_{i\rightarrow\infty}\left(\sup_{y\in Y}\left\Vert g.m_{i}^{y}-m_{i}^{g.y}\right\Vert _{1}\right)=0\text{,}
\]
 where $g.m_{i}^{y}(g^{\prime}):=m_{i}^{y}(g^{-1}g^{\prime})$, $g^{\prime}\in G$.
Here, $\text{Prob}(G)$ denotes the space of probability measures
on the group $G$.
\end{definition}

The main result in \cite[Section 5]{Lecureux}
states that for a Coxeter system $(W,S)$ the action of $W$ on the
maximal combinatorial compactification $\mathcal{C}_{\text{sph}}(\Sigma)$
is amenable. The argument makes use of a construction by Dranishnikov
and Januszkiewicz, see \cite{DJ}. Similar constructions
appear in \cite{Fe} and \cite{Vinberg}.
Since $\mathcal{C}_{1}(W,S)$ is a closed subset of $\mathcal{C}_{\text{sph}}(\Sigma)$
we deduce the following corollary.

\begin{corollary} \label{Total amenability}
Let $(W,S)$ be a finite rank Coxeter system.
Then the actions $W\curvearrowright\overline{(W,S)}$ and $W\curvearrowright\partial(W,S)$
are amenable.
\end{corollary}


\subsection{Smallness at infinity}

\begin{definition} [{\cite[Definition 5.1.6]{BrownOzawa}}] \label{infty}
A \emph{compactification} of a group $G$ is a compact Hausdorff space $\overline{G}$ containing $G$ as an open dense subset. It is called (left) \emph{equivariant} if the left translation action of $G$ on $G$ extends to a continuous action on $\overline{G}$. The compactification $\overline{G}$ is said to be \emph{small at infinity} if for every net $(g_{i})_{i\in I}\subseteq G$ converging to a point $z\in\overline{G}\setminus G$ and every $g^{\prime}\in G$, one has that $g_{i}g^{\prime}\rightarrow z$.
\end{definition}

By what we have shown earlier it is clear that for every finite rank Coxeter system $(W,S)$ the corresponding space $\overline{(W,S)}$ is indeed an equivariant compactification in the sense of Definition \ref{infty}. 

\begin{definition}
We call a finite rank Coxeter system $(W,S)$ \emph{small at infinity} if $\overline{(W,S)}$ is small at infinity. If the generating set $S$ is clear, we will also say that $W$ is small at infinity.
\end{definition}

As we will see in Subsection \ref{Akemann}, the notion of smallness at infinity has implications for the Hecke operator algebras of the system. The aim of this subsection is a characterization of Coxeter groups that are small at infinity.

\begin{theorem} \label{infinity}
Let $(W,S)$ be a finite rank Coxeter system. Then the following statements are equivalent:
\begin{enumerate}[label=(\arabic*)]
\item $W$ is small at infinity;
\item $\#C_{W}(s)<\infty$ for every $s\in S$.
\end{enumerate}

\noindent Here $C_{W}(s):=\left\{ \mathbf{w}\in W\mid s\mathbf{w}=\mathbf{w}s\right\}$  denotes the centralizer of $s$ in $W$.
\end{theorem}

\begin{proof}
``$\ensuremath{\left(1\right)\Rightarrow\left(2\right)}$'': Let $s\in S$ be a generator with $\#C_{W}(s)=\infty$. By the compactness of $\overline{W}$ one can find a sequence $(\mathbf{w}_{i})_{i\in\mathbb{N}}\subseteq C_{W}(s)$ converging to a boundary point $z\in\partial W$. It can be chosen in such a way that $s\nleq\mathbf{w}_{i}$ for every $i\in\mathbb{N}$. But then $\mathbf{w}_{i}s\nrightarrow z$ since $s \nleq z$, i.e. $W$ is not small at infinity.

``$\ensuremath{\left(2\right)\Rightarrow\left(1\right)}$'': Let $W$ not be small at infinity. Choose a convergent sequence $(\mathbf{w}_{i})_{i\in\mathbb{N}}\subseteq W$ with limit point $z\in\partial W$ and an element $\mathbf{v}\in W$ such that $\mathbf{w}_{i}\mathbf{v}\nrightarrow z$. One can assume that $\mathbf{v}=s$ for some $s\in S$ and that there exist $\mathbf{w}\in W$, $i_{0}\in\mathbb{N}$ with $\mathbf{w}\leq\mathbf{w}_{i}$ and $\mathbf{w}\nleq\mathbf{w}_{i}s$ for all $i\geq i_{0}$. Further, we can assume that $s$ always cancels the first letter of $\mathbf{w}_{i}$. Indeed, for $i\geq i_{0}$, $\mathbf{w}_{i}$ is of the form $\mathbf{w}_{i}=\mathbf{w}\mathbf{u}_{i}$ with $\left|\mathbf{w}\mathbf{u}_{i}\right|=\left|\mathbf{w}\right|+\left|\mathbf{u}_{i}\right|$ and the multiplication of $\mathbf{w}\mathbf{u}_{i}$ with $s$ cancels some letter in the reduced expression $t_{1}...t_{n}$ for $\mathbf{w}$. As $\mathbf{w}$ consists of finitely many letters, by possibly going over to some subsequence, we can assume that multiplication by $s$ always cancels the same letter, say $t_{j}$, in the expression. Then, by possibly replacing $\mathbf{w}_{i}$ by $(t_{1}...t_{j-1})^{-1}\mathbf{w}_{i}$, we can further assume that $s$ cancels the first letter of $\mathbf{w}_{i}$. Call this letter $t$. We get that for $i\geq i_{0}$, $\mathbf{w}_{i}$ is of the form $\mathbf{w}_{i}=t\mathbf{v}_{i}$ where $\left|t\mathbf{v}_{i}\right|=\left|\mathbf{v}_{i}\right|+1$ and $\mathbf{w}_{i}s=\mathbf{v}_{i}$. This implies
\begin{eqnarray}
\nonumber
s=\mathbf{w}_{i}^{-1}t\mathbf{w}_{i}=(\mathbf{w}_{i_{0}}^{-1}\mathbf{w}_{i})^{-1}s(\mathbf{w}_{i_{0}}^{-1}\mathbf{w}_{i})\text{,}
\end{eqnarray}
i.e. $\mathbf{w}_{i_{0}}^{-1}\mathbf{w}_{i}\in C_{W}(s)$ for every $i\geq i_{0}$. We get that $\#C_{W}(s)=\infty$.
\end{proof}

Reflection centralizers of Coxeter groups have been studied in \cite{Allcock} and \cite{Brink}. The main theorem in \cite{Brink} gives the description of the centralizer $C_{W}(s)$ of a generator $s$ in a Coxeter group $W$ as a semidirect product of its reflection subgroup by the fundamental group of the connected component of the odd Coxeter diagram of $W$ containing $s$. In combination with Theorem \ref{infinity} this has the following immediate consequence.

\begin{corollary} \label{cycle}
Let $(W,S)$ be a finite rank Coxeter system for which the corresponding odd Coxeter diagram contains a cycle. Then $\overline{(W,S)}$ is not small at infinity.
\end{corollary}

Let us collect some other consequences of Theorem \ref{infinity}.

The following proposition relies on the well-known fact that irreducible affine Coxeter groups arise as subgroups generated by (affine) reflections associated with crystallographic root systems (for details see \cite{Humphreys}). Recall that a \emph{crystallographic root system} $\Phi$ is a set of finitely many vectors that span a real Euclidean space $V$ and satisfy certain geometrical properties. For every $\alpha\in\Phi$, $i\in\mathbb{Z}$ the set $H_{\alpha,i}:=\left\{ x\in V\mid\left\langle x,\alpha\right\rangle =i\right\}$  defines an affine \emph{hyperplane} in $V$. Write $r_{\alpha,i}$ for the (unique) non-trivial isometry of $V$ that fixes $H_{\alpha,i}$. Then the set $R:=\left\{ r_{\alpha,i}\mid\alpha\in\Phi,i\in\mathbb{Z}\right\}$  generates an affine Coxeter group that has $R$ as its set of reflections. Every irreducible affine Coxeter group arises in that way. A \emph{translation} in $V$ is a map of the form $t_{v}:x\mapsto x+v$ for some $v\in V$. Note that for each $\alpha\in\Phi$ the product $r_{\alpha,1}r_{\alpha,0}$ is a non-zero translation in the direction of $\alpha$ with $r_{\alpha,1}r_{\alpha,0}=t_{\alpha^{\vee}}$, $\alpha^{\vee}:=\frac{2}{\left\langle \alpha,\alpha\right\rangle }\alpha$.

\begin{proposition} \label{affine}
An irreducible Coxeter system of affine type is small at infinity if and only if it is the infinite dihedral group.
\end{proposition}

\begin{proof}
Let $(W,S)$ be an irreducible affine Coxeter system and denote the associated crystallographic root system by $\Phi$. The Coxeter diagram is of one of the following forms: $(\tilde{A}_{n})_{n\geq2}$, $(\tilde{B}_{n})_{n\geq3}$, $(\tilde{C}_{n})_{n\geq2}$, $(\tilde{D}_{n})_{n\geq4}$, $(\tilde{E}_{n})_{6\leq n\leq8}$, $\tilde{F}_{4}$, $\tilde{G}_{2}$, $\tilde{I}_{1}$.

\emph{Case 1:} If the Coxeter system is of the form $(\tilde{A}_{n})_{n\geq2}$, $(\tilde{B}_{n})_{n\geq3}$, $(\tilde{C}_{n})_{n\geq2}$, $(\tilde{D}_{n})_{n\geq4}$, $(\tilde{E}_{n})_{6\leq n\leq8}$, $\tilde{F}_{4}$ or $\tilde{G}_{2}$, then $\left|S\right|\geq3$. Therefore, the reflection hyperplane $H_{\alpha,i}$ with $\alpha\in\Phi$, $i\in\mathbb{Z}$ corresponding to a generator $s\in S$ is at least $1$-dimensional and one finds an element $\beta\in\Phi$ that is linearly independent from $\alpha$. We have that $\gamma:=\beta-\left\langle \alpha,\beta\right\rangle \alpha^{\vee}\in\Phi$ and the translation $t_{\beta^{\vee}+\gamma^{\vee}}=t_{\beta^{\vee}}t_{\gamma^{\vee}}$ corresponds to an infinite order element in the Coxeter group $W$. By
\begin{eqnarray} \nonumber
\beta^{\vee}+\gamma^{\vee}=\frac{2}{\left\langle \beta,\beta\right\rangle }\left(2\beta-\frac{2\left\langle \alpha,\beta\right\rangle }{\left\langle \alpha,\alpha\right\rangle }\alpha\right)\in H_{\alpha,0}\text{,}
\end{eqnarray}
the translation $t_{\beta^{\vee}+\gamma^{\vee}}$ stabilizes the hyperplane $H_{\alpha,i}$, hence the element commutes with $s$. It follows from Theorem \ref{infinity} that $W$ is not small at infinity.

\emph{Case 2:} If $(W,S)$ is infinite dihedral, i.e. $W=\left\langle s,t\mid s^{2}=t^{2}=e\right\rangle$, then obviously $\#C_{W}(s)=\#C_{W}(t)=2$.
\end{proof}

Recall that by Corollary \ref{hyperbolic 2} for every word hyperbolic Coxeter system $(W,S)$ the map \begin{eqnarray} \nonumber \tilde{\phi}:\overline{(W,S)}\rightarrow W\cup\partial_{h}W \end{eqnarray} given by $\tilde{\phi}(\mathbf{w})=\mathbf{w}$ for $\mathbf{w}\in W$ and $\tilde{\phi}(\left[\alpha\right])=\left[\alpha\right]_{h}$ for an infinite geodesic path $\alpha$ is well-defined, continuous, $W$-equivariant and surjective with $\tilde{\phi}(\partial(W,S))=\partial_{h}W$. The injectivity of $\tilde{\phi}$ gives information on whether or not the system is small at infinity, as the next theorem illustrates.

\begin{theorem} \label{hyperbolic 3}
Let $(W,S)$ be a finite rank Coxeter system. Then $\left(W,S\right)$ is small at infinity if and only if $W$ is word hyperbolic and the map $\tilde{\phi}$ (resp. its restriction $\tilde{\phi}|_{\partial(W,S)}$) from Corollary \ref{hyperbolic 2} is a homeomorphism.
\end{theorem}

\begin{proof}
``$\ensuremath{\Rightarrow}$'': Let $(W,S)$ be small at infinity and assume that the system is not word hyperbolic. By Moussong's characterization of word hyperbolic Coxeter groups \cite[Theorem 17.1]{Moussong}, $S$ contains a subset $T\subseteq S$ such that $(W_{T},T)$ is either of affine type with $\#T\geq3$ or the Coxeter system decomposes as $(W_{T},T)=(W_{T^{\prime}}\times W_{T^{\prime\prime}},T^{\prime}\cup T^{\prime\prime})$ with both $W_{T^{\prime}}$ and $W_{T^{\prime\prime}}$ infinite. In the first case we deduce with Proposition \ref{affine} that $W$ is not small at infinity which contradicts our assumption. In the second case the same contradiction follows from Theorem \ref{infinity} and $C_{W}(s)\supseteq W_{s}\times W_{T^{\prime\prime}}$ for every $s\in T^{\prime}$. Hence, $(W,S)$ must be word hyperbolic. It remains to show that the map $\tilde{\phi}$ is injective. For this, let $\alpha$ and $\beta$ be infinite geodesic paths with $\left[\alpha\right]_{h}=\left[\beta\right]_{h}$. By $\sup_{i}\left|\alpha_{i}^{-1}\beta_{i}\right|<\infty$, the set $\left\{ \alpha_{i}^{-1}\beta_{i}\mid i\in\mathbb{N}\right\} \subseteq W$ is bounded with respect to the word metric on $W$. We hence find a strictly increasing sequence $(i_{k})_{k\in\mathbb{N}}\subseteq\mathbb{N}$ and an element $\mathbf{w}\in W$ with $\alpha_{i_{k}}^{-1}\beta_{i_{k}}=\mathbf{w}$ for all $k\in\mathbb{N}$. But $(W,S)$ is small at infinity, so $\left[\beta\right]=\lim_{k}\beta_{i_{k}}=\lim_{k}\alpha_{i_{k}}\mathbf{w}=\left[\alpha\right]$. This implies that $\tilde{\phi}$ is indeed injective.

``$\ensuremath{\Leftarrow}$'': It is well-known that the hyperbolic compactification of a word hyperbolic group is small at infinity, see for instance \cite[Proposition 5.3.18]{BrownOzawa}. Hence, if $(W,S)$ is word hyperbolic and the map $\tilde{\phi}$ is a homeomorphism, then $(W,S)$ is small at infinity.
\end{proof}

\begin{proposition} \label{free}
Let $(W,S)$ be a finite rank Coxeter system that is a free product of finite Coxeter groups, meaning that $S$ is the disjoint union of non-empty subsets $S_{1},...,S_{n}\subseteq S$ whose corresponding special subgroups $W_{S_{1}}, ..., W_{S_{n}}$ are all finite with $W=W_{S_{1}}\ast...\ast W_{S_{n}}$. Then $(W,S)$ is small at infinity.
\end{proposition}

\begin{proof}
Let $(W,S)$ be an irreducible Coxeter system that is a free product of finite Coxeter groups. The corresponding Cayley graph $\text{Cay}(W,S)$ is locally finite and hyperbolic. By Theorem \ref{hyperbolic 3} it suffices to prove the injectivity of the map $\tilde{\phi}$. Let $\alpha$ and $\beta$ be two infinite geodesic paths with $\alpha\sim_{h}\beta$. For every $i\in\mathbb{N}$ let $s_{1}...s_{i}$ be a reduced expression for $\alpha_{i}$ and let $t_{1}...t_{i}$ be a reduced expression for $\beta_{i}$. It is clear that $s_{1}$ and $t_{1}$ must lie in the same component of the free product. The same is true for $s_{2}$, $t_{2}$, .... As the free product components $W_{S_{1}}$, ..., $W_{S_{n}}$ are finite, there exists exists $i\in\mathbb{N}$ such that $s_{1},...,s_{i}$ (and hence $t_{1},...,t_{i}$) all lie in the same component and such that $s_{i+1}$ (resp. $t_{i+1}$) lies in a different component. By $\sup_{j}|\alpha_{j}^{-1}\beta_{j}|<\infty$ we then get $s_{i}...s_{1}t_{1}...t_{i}=e$. Proceeding like this, one concludes that there exists an increasing sequence $(i_{k})_{k\in\mathbb{N}}\subseteq\mathbb{N}$ with $\alpha_{i_{k}}=\beta_{i_{k}}$ for every $k\in\mathbb{N}$. This implies $\left[\alpha\right]=\left[\beta\right]$, i.e. $\tilde{\phi}$ is injective.
\end{proof}

\begin{corollary}
An irreducible finite rank right-angled Coxeter system is small at infinity if and only if it is a free product of finite Coxeter groups.
\end{corollary}

\begin{proof}
``$\ensuremath{\Leftarrow}$'': This follows from Lemma \ref{free}.
``$\ensuremath{\Rightarrow}$'': Let $(W,S)$ be an irreducible right-angled Coxeter system that is not a free product of finite Coxeter groups. One easily checks that $S$ contains elements $r,s,t$ with coefficients $m_{rs}=m_{rt}=2$ and $m_{st}=\infty$. In particular, $C_{W}(r)\supseteq\left\langle s,t\right\rangle \cong\mathbf{D}_{\infty}$ where $\mathbf{D}_{\infty}$ denotes the infinite dihedral group. But then $\#C_{W}(r)=\infty$, so $W$ is not small at infinity by Theorem \ref{infinity}.
\end{proof}

\begin{remark}
Not every Coxeter system that is small at infinity is a free product of finite Coxeter groups. Consider for instance the group $W$ represented by
\begin{eqnarray}
\nonumber
\left\langle r,s,t\mid m_{rs}=3,m_{rt}=2,m_{st}=\infty\right\rangle \text{.}
\end{eqnarray}
It is irreducible and non-affine. Obviously, all of its reflection centralizers are finite, so $W$ is small at infinity. However, $W$ can not be decomposed into a non-trivial free product since, by $m_{rs}=3$ and $m_{rt}=2$, the generators $r$, $s$ and $t$ would all have to sit in the same component of that decomposition.
\end{remark}


\subsection{Boundary actions of Coxeter groups}

The notion of (topological) boundary actions was introduced by Furstenberg in \cite{Fu1} and \cite{Fu2} in the context of rigidity questions of semisimple Lie groups. It recently gained a lot of attention based on results by Kalantar and Kennedy \cite{KalantarKennedy} who established a connection between the dynamical properties of the Furstenberg boundary of a given group and the question for simplicity, uniqueness of trace and tightness of nuclear embedding of the corresponding reduced group C$^\ast$-algebra. A series of breakthrough works followed (see e.g. \cite{BKKO}, \cite{Haagerup}, \cite{Adam} and also \cite{BK}, \cite{HK})

In this subsection we will study two classes of Coxeter groups $\left(W,S\right)$ whose corresponding boundary $\partial(W,S)$ is a $W$-boundary in the sense of Furstenberg. We will further consider related properties of the action $W\curvearrowright\partial(W,S)$ which relate to the operator algebras of the group $W$. We will pick up some of the implications in Section \ref{applications}.

\begin{definition}
Let $G$ be a discrete group acting (continuously) on a compact Hausdorff space $X$.
\begin{itemize}
\item The action is called \emph{minimal} if for every $x\in X$ the $G$-orbit $G.x:=\left\{ g.x\mid g\in G\right\}$  is dense in $X$.
\item The action is called \emph{strongly proximal} if for every probability measure $\nu\in \text{Prob}(X)$ the weak-$\ast$ closure of the $G$-orbit $G.\nu$ contains a point mass $\delta_{x}\in \text{Prob}(X)$ for some $x\in X$.
\item $X$ is called a \emph{G-boundary} if the action of $G$ on $X$ is both minimal and strongly proximal. In that case the action is called a \emph{boundary action}.
\end{itemize}
\end{definition}

Furstenberg proved in \cite[Proposition 4.6]{Fu2} that every discrete group $G$ admits a unique $G$-boundary $\partial_{F}G$ that is universal in the sense that every other $G$-boundary is a continuous $G$-equivariant image of $\partial_{F}G$. It is called the \emph{Furstenberg boundary} of the group $G$.

In the case of an irreducible right-angled Coxeter system we can completely characterize when the corresponding action on the boundary is a boundary action. Note that the only Coxeter group generated by one element is the finite group $\mathbb{Z}_2$ whose boundary is empty.

\begin{theorem} \label{r.a.}
Let $\left(W,S\right)$ be a finite rank right-angled irreducible Coxeter system. Then the following statements hold:
\begin{itemize}
\item If $\left|S\right|=2$, then the action $W\curvearrowright\partial(W,S)$ is minimal but not strongly proximal;
\item If $\left|S\right|\geq3$, then the action $W\curvearrowright\partial(W,S)$ is a boundary action.
\end{itemize}
\end{theorem}

\begin{proof}
In the case $\left|S\right|=2$ the Coxeter group $W$ is the infinite dihedral group
\begin{eqnarray}
\nonumber
\mathbf{D}_{\infty}=\left\langle s,t\mid s^{2}=t^{2}=e\right\rangle
\end{eqnarray}
whose boundary $\partial\mathbf{D}_{\infty}$ consists of the two points $z_{1}:=stst...$ and $z_{2}:=tsts...$. It is clear that the action $\mathbf{D}_{\infty}\curvearrowright\partial\mathbf{D}_{\infty}$ is minimal. It is not strongly proximal since for the probability measure $\mu:=\frac{1}{2}(\delta_{z_{1}}+\delta_{z_{2}})\in \text{Prob}(\partial\mathbf{D}_{\infty})$ the equalities $s.\mu=t.\mu=\mu$ hold, i.e. $\overline{W.\mu}=\left\{ \mu\right\}$.

Let us now assume that $(W,S)$ is a right-angled irreducible Coxeter system with $\left|S\right|\geq3$. Recall that if we have cancellation of the form $s_1...s_n = s_1 ...\widehat{s_i} ... \widehat{s_j}...s_n$ for $s_1 ,...,s_n \in S$, then $s_i = s_j$ and $s_i$ commutes with every letter in the reduced expression for $s_{i+1}...s_{j-1}$ (see the remark after Theorem \ref{cancellation}). In the following we will often implicitly make use of this property.

\emph{Minimality:} Let $\alpha$ and $\beta$ be arbitrary infinite geodesic paths with $\alpha_{0}=\beta_{0}=e$. We have to show that $\left[\beta\right]\in\overline{W.\left[\alpha\right]}$. Since $S$ is finite, we find $t\in S$ and a strictly increasing sequence $(i_{k})_{k\in\mathbb{N}}\subseteq\mathbb{N}$ with $t\leq_{L}\beta_{i_{k}}$ for every $k\in\mathbb{N}$. Further, let $t^{\prime}:=\alpha_{1}\in S$ and choose a path $s_{0}...s_{n+1}$ in the Coxeter diagram of $(W,S)$ that connects $t^{\prime}$ and $t$, meaning that $s_{0},...,s_{n+1}\in S$ with $s_{0}=t^{\prime}$, $s_{n+1}=t$ and $m_{s_{j}s_{j+1}}=\infty$ for $j=0,...,n$. We claim that $(\beta_{i_{k}}s_{n}...s_{1}).\left[\alpha\right]\rightarrow\left[\beta\right]$. Indeed, by the choice of $s_{1},...,s_{n}$ one gets $\beta_{i_{k}}\leq(\beta_{i_{k}}s_{n}...s_{1})\alpha_{j}\leq(\beta_{i_{k}}s_{n}...s_{1}).\left[\alpha\right]$ for all $j,k\in\mathbb{N}$, so for every $\mathbf{w}\in W$ with $\mathbf{w}\leq\left[\beta\right]$ one eventually has $(\beta_{i_{k}}s_{n}...s_{1}).\left[\alpha\right]\in\mathcal{U}_{\mathbf{w}}=\left\{ z\in\overline{W}\mid\mathbf{w}\leq z\right\}$.

Now let $\mathbf{w}\in W$ with $\mathbf{w}\nleq\left[\beta\right]$ and let $\mathbf{w}=t_{1}...t_{n}$ be a reduced expression for $\mathbf{w}$. We have to show that $\mathbf{w}\nleq(\beta_{i_{k}}s_{n}...s_{1}).\left[\alpha\right]$ eventually. Assume that this is not the case. By possibly going over to a subsequence we can then assume that $\mathbf{w}\leq(\beta_{i_{k}}s_{n}...s_{1}).\left[\alpha\right]$ for infinitely many $k\in\mathbb{N}$. Let us proceed inductively:
\begin{itemize}
\item By the choice of $s_{1},...,s_{n}$ one either has $t_{1}\leq\beta_{i_{k}}$ or $t_{1}=s_{n}$ and $t_{1}$ commutes with every letter of $\beta_{i_{k}}$. Only the first case is possible since $m_{s_{n},t}=\infty$, so $t_{1}\leq\beta_{i_{k}}$.
\item Further, one either has $t_{1}t_{2}\leq\beta_{i_{k}}$ or $t_{2}=s_{n}$ and $t_{2}$ commutes with every letter of $t_{1}\beta_{i_{k}}$. In the second case we would get that $t_{1}=t$ and that $t$ commutes with every letter of $\beta_{i_{k}}$. But for $k\geq1$ the letter $t$ appears more than once in the reduced expression for $\beta_{i_{k}}$ which leads to a contradiction. Hence, $t_{1}t_{2}\leq\beta_{i_{k}}$ for $k\geq1$.
\end{itemize}
Proceeding like this, we get that $\mathbf{w}\leq\beta_{i_{k}}$ for large enough $k$, in contradiction to $\mathbf{w}\nleq\left[\beta\right]$. Theorefore, for every $\mathbf{w}\in W$ with $\mathbf{w}\nleq\left[\beta\right]$, $(\beta_{i_{k}}s_{n}...s_{1}).\left[\alpha\right]\in\mathcal{U}_{\mathbf{w}}^{c}$ eventually. This implies that indeed $(\beta_{i_{k}}s_{n}...s_{1}).\left[\alpha\right]\rightarrow\left[\beta\right]$, i.e. $\left[\beta\right]\in\overline{W.\left[\alpha\right]}$.

\emph{Strong proximality:} We have to show that for every probability measure $\mu\in \text{Prob}(\partial W)$ there exists $z\in\partial W$ with $\delta_{z}\in\overline{W.\mu}$ where the closure is taken in the weak-$\ast$ topology. The argument is similar to the one above. Let $z\in\partial W$, choose a path $s_{1}...s_{n}$ in the Coxeter diagram of $(W,S)$ that covers the whole graph (i.e. $m_{s_{j}s_{j+1}}=\infty$ for $j=1,...,n-1$) with $m_{s_{1}s_{n}}=\infty$ and set $\mathbf{g}:=s_{1}...s_{n}$. Obviously, the sequences $\left(\mathbf{g}^{k}\right)_{k\in\mathbb{N}}$ and $\left(\mathbf{g}^{-k}\right)_{k\in\mathbb{N}}$ converge to boundary points $\mathbf{g}^{\infty}$ and $\mathbf{g}^{-\infty}$. We either have $s_{1}\leq\mathbf{g}^{k}.z$ for some $k\in\mathbb{N}$ or $z=\mathbf{g}^{-\infty}$. In the first case, $\mathbf{g}^{k}.z\rightarrow\mathbf{g}^{\infty}$ and in the second case $\mathbf{g}^{k}.z\rightarrow\mathbf{g}^{-\infty}$. This implies that for $\mu\in \text{Prob}(\partial W)$ there exists $\lambda\in\left[0,1\right]$ with
\begin{eqnarray}
\nonumber
\lambda\delta_{\mathbf{g}^{\infty}}+\left(1-\lambda\right)\delta_{\mathbf{g}^{-\infty}}=\lim_{k\rightarrow\infty}\mathbf{g}^{k}.\mu\in\overline{W.\mu}\text{.}
\end{eqnarray}
Now, choose a second path $t_{1}...t_{m}$ in the Coxeter diagram of $(W,S)$ that covers the whole graph (i.e. $m_{t_{j}t_{j+1}}=\infty$ for $j=1,...,m-1$) with $t_{1}\notin\left\{ s_{1},s_{n}\right\}$ , $m_{t_{1}t_{m}}=\infty$ and set $\mathbf{h}:=t_{1}...t_{m}$. Again, the sequences $\left(\mathbf{h}^{k}\right)_{k\in\mathbb{N}}$ and $\left(\mathbf{h}^{-k}\right)_{k\in\mathbb{N}}$ converge to boundary points $\mathbf{h}^{\infty}$ and $\mathbf{h}^{-\infty}$. Further, $\mathbf{h}^{k}.\mathbf{g}^{\infty}\rightarrow\mathbf{h}^{\infty}$ and $\mathbf{h}^{k}.\mathbf{g}^{-\infty}\rightarrow\mathbf{h}^{\infty}$ from which we conclude that
\begin{eqnarray}
\nonumber
\delta_{\mathbf{h}^{\infty}}=\lim_{k\rightarrow\infty}\left(\lambda\delta_{\mathbf{g}^{\infty}}+\left(1-\lambda\right)\delta_{\mathbf{g}^{-\infty}}\right)\in\overline{W.\mu}\text{.}
\end{eqnarray}
The claim follows.
\end{proof}

For Coxeter systems which are small at infinity a characterization of the form as in Theorem \ref{r.a.} is possible as well. Note that by Theorem \ref{infinity} and Proposition \ref{affine} the only amenable finite rank irreducible Coxeter groups that are small at infinity are either the finite ones or the infinite dihedral group which is already covered by Theorem \ref{r.a.}.

\begin{theorem} \label{s.a.i.}
Let $\left(W,S\right)$ be a non-amenable finite rank Coxeter system that is small at infinity. Then the action $W\curvearrowright\partial(W,S)$ is a boundary action.
\end{theorem}

\begin{proof}
By Theorem \ref{hyperbolic 3} the group $W$ is word hyperbolic and the boundary $\partial(W,S)$ coincides with the hyperbolic boundary $\partial_{h}W$. It is well-known that the action of a non-amenable word hyperbolic group is a boundary action (see for instance \cite{KalantarKennedy}). This proves the statement.
\end{proof}

\begin{remark} \label{limitelements}
Let $(W,S)$ be a right-angled irreducible Coxeter system with $3\leq\left|S\right|<\infty$. Note that by the same argument as in the proof of Theorem \ref{r.a.} the action $W\curvearrowright\overline{(W,S)}$ is strongly proximal. Indeed, the elements $\mathbf{g}$ and $\mathbf{h}$ appearing in the proof of Theorem \ref{r.a.} have the property that the limits $\mathbf{g}^{\pm\infty}:=\lim\mathbf{g}^{\pm l}$ and $\mathbf{h}^{\pm\infty}:=\lim\mathbf{h}^{\pm l}$ exist and that $\mathbf{g}^{k}.z\rightarrow\mathbf{g}^{\infty}$ for every $z\in\overline{(W,S)}\setminus\left\{ \mathbf{g}^{-\infty}\right\}$  and $\mathbf{h}^{k}.z\rightarrow\mathbf{h}^{\infty}$ for every $z\in\overline{(W,S)}\setminus\left\{ \mathbf{h}^{-\infty}\right\}$. Further, $\mathbf{h}^{-\infty}\neq\mathbf{g}^{\pm\infty}$. We deduce that the action $W\curvearrowright\overline{(W,S)}$ is strongly proximal. If the Coxeter system $(W,S)$ is non-amenable and small at infinity, the strong proximality of the action $W\curvearrowright\overline{(W,S)}$ also holds. It follows from Theorem \ref{hyperbolic 3} and \cite[Corollaire 20]{Harpe}.
\end{remark}

\begin{definition}
Let $G$ be a discrete group acting (continuously) on a compact Hausdorff space $X$. The action is \emph{topologically free} if for every $g\in G\setminus\left\{ e\right\}$  the set $X^{g}:=\left\{ x\in X\mid g.x=x\right\}$  has no inner points.
\end{definition}

\begin{lemma} \label{topfree}
Let $(W,S)$ be a finite rank Coxeter system. Then the natural action of $W$ on its compactification $\overline{(W,S)}$ is topologically free.
\end{lemma}

\begin{proof}
The statement immediately follows from the fact that $W$ is a dense subset of $\overline{W}$.
\end{proof}

Again, in the right-angled case we can characterize when the corresponding action of the Coxeter groups on its boundary is topologically free. The argument requires a technical lemma.

\begin{lemma} \label{technical}
Let $\left(W,S\right)$ be a finite rank right-angled irreducible Coxeter system. For $\mathbf{w}\in W\setminus\left\{ e\right\}$, $z\in\partial(W,S)$ with $\mathbf{w}.z=z$ there exist elements $\mathbf{u},\mathbf{v}\in W$ with $\mathbf{w}=\mathbf{u}\mathbf{v}^{-1}$, $\left|\mathbf{w}\right|=\left|\mathbf{u}\right|+\left|\mathbf{v}\right|$ and $\mathbf{u},\mathbf{v}\leq z$.
\end{lemma}

\begin{proof}
Let $\mathbf{w}\in W$, $z\in\partial(W,S)$ be elements with $\mathbf{w}.z=z$ and let $\mathbf{w}=s_{1}...s_{n}$ be a reduced expression for $\mathbf{w}$. We claim that for every $1\leq k\leq n$ we find integers $i_{1}<...<i_{l}$ and $j_{1}<...<j_{m}$ such that $\left\{ n-k+1,...,n\right\} =\left\{ i_{1},...,i_{l},j_{1},...,j_{m}\right\}$, $\mathbf{w}=(s_{1}...s_{n-k})(s_{i_{1}}...s_{i_{l}})(s_{j_{1}}...s_{j_{m}})$ is a reduced expression for $\mathbf{w}$, $s_{j_{m}}...s_{j_{1}}\leq z$ and $s_{i_{1}}...s_{i_{l}}\leq(s_{i_{1}}...s_{i_{l}})(s_{j_{1}}...s_{j_{m}}).z$. We prove this by induction over $k$.

\begin{itemize}
\item For $k=1$ we have that either $s_{n}\leq z$ or $s_{n}\nleq z$. In the first case set $l=0$, $m=1$ and $j_{1}=n$. Then, $\mathbf{w}=\left(s_{1}...s_{n-1}\right)s_{j_{1}}$ is a reduced expression for $\mathbf{w}$ with $s_{j_{1}}\leq z$. In the second case set $l=1$, $m=0$ and $i_{1}=n$. Then again, $\mathbf{w}=(s_{1}...s_{n-1})s_{i_{1}}$ is a reduced expression for $\mathbf{w}$ with $s_{i_{1}}\leq s_{i_{1}}.z$. We get that for $k=1$ the claimed statement holds.
\item Now assume that the claim holds for $k\in\mathbb{N}$, i.e. we have $i_{1}<...<i_{l}$ and $j_{1}<...<j_{m}$ with $\left\{ n-k+1,...,n\right\} =\left\{ i_{1},...,i_{l},j_{1},...,j_{m}\right\}$ such that $\mathbf{w}=(s_{1}...s_{n-k})(s_{i_{1}}...s_{i_{l}})(s_{j_{1}}...s_{j_{m}})$ is a reduced expression for $\mathbf{w}$, $s_{j_{m}}...s_{j_{1}}\leq z$ and $s_{i_{1}}...s_{i_{l}}\leq(s_{i_{1}}...s_{i_{l}})(s_{j_{1}}...s_{j_{m}}).z$. Now, either $s_{n-k}\leq(s_{i_{1}}...s_{i_{l}})(s_{j_{1}}...s_{j_{m}}).z$ or $s_{n-k}\nleq(s_{i_{1}}...s_{i_{l}})(s_{j_{1}}...s_{j_{m}}).z$. In the first case, since $s_{n-k}(s_{i_{1}}...s_{i_{l}})(s_{j_{1}}...s_{j_{m}})$ is reduced and $s_{i_{1}}...s_{i_{l}}\leq(s_{i_{1}}...s_{i_{l}})(s_{j_{1}}...s_{j_{m}}).z$, we get that $s_{n-k}$ commutes with $s_{i_{1}}...s_{i_{l}}$ and $s_{n-k}\leq(s_{j_{1}}...s_{j_{m}}).z$. Hence,
\begin{eqnarray} \nonumber
\left\{ n-k,...,n\right\} =\left\{ i_{1},...,i_{l},n-k,j_{1},...,j_{m}\right\} \text{, }
\end{eqnarray}
$\mathbf{w}=(s_{1}...s_{n-k-1})(s_{i_{1}}...s_{i_{l}})(s_{n-k}s_{j_{1}}...s_{j_{m}})$ is a reduced expression for $\mathbf{w}$, $s_{j_{m}}...s_{j_{1}}s_{n-k}\leq z$ and $s_{i_{1}}...s_{i_{l}}\leq(s_{i_{1}}...s_{i_{l}})(s_{n-k}s_{j_{1}}...s_{j_{m}}).z$. In the second case,
\begin{eqnarray}
\nonumber
\left\{ n-k,...,n\right\} =\left\{ n-k,i_{1},...,i_{l},j_{1},...,j_{m}\right\},
\end{eqnarray}
$\mathbf{w}=(s_{1}...s_{n-k-1})(s_{n-k}s_{i_{1}}...s_{i_{l}})(s_{j_{1}}...s_{j_{m}})$ is a reduced expression for $\mathbf{w}$, $s_{j_{1}}...s_{j_{m}}\leq z$ and $s_{n-k}s_{i_{1}}...s_{i_{l}}\leq(s_{n-k}s_{i_{1}}...s_{i_{l}})(s_{j_{1}}...s_{j_{m}}).z$. In both cases we get that the claim also holds for $k+1$.
\end{itemize}
This completes the induction argument.

For $k=n$ we get that there exist $i_{1}<...<i_{l}$ and $j_{1}<...<j_{m}$ with $\left\{ 1,...,n\right\} =\left\{ i_{1},...,i_{l},j_{1},...,j_{m}\right\}$  such that $\mathbf{w}=(s_{i_{1}}...s_{i_{l}})(s_{j_{1}}...s_{j_{m}})$ is a reduced expression for $\mathbf{w}$, $s_{j_{m}}...s_{j_{1}}\leq z$ and $s_{i_{1}}...s_{i_{l}}\leq(s_{i_{1}}...s_{i_{l}})(s_{j_{1}}...s_{j_{m}}).z=\mathbf{w}.z=z$. The lemma then follows via $\mathbf{u}:=s_{i_{1}}...s_{i_{l}}$ and $\mathbf{v}:=s_{j_{m}}...s_{j_{1}}$.
\end{proof}

\begin{proposition} \label{r.a.2}
Let $\left(W,S\right)$ be a right-angled irreducible Coxeter system with $2\leq\left|S\right|<\infty$. Then the action $W\curvearrowright\partial(W,S)$ is topologically free if and only if $\left|S\right|\geq 3$.
\end{proposition}

\begin{proof}
``$\ensuremath{\Rightarrow}$'': Again, for $\left|S\right|=2$ the Coxeter group $W$ is the infinite dihedral group
\begin{eqnarray}
\nonumber
\mathbf{D}_{\infty}=\left\langle s,t\mid s^{2}=t^{2}=e\right\rangle
\end{eqnarray}
with boundary $\partial\mathbf{D}_{\infty}=\left\{ z_{1},z_{2}\right\}$  where $z_{1}:=stst...$ and $z_{2}:=tsts...$. Obviously, $\partial\mathbf{D}_{\infty}$ carries the discrete topology and $(\partial\mathbf{D}_{\infty})^{st}=\left\{ z_{1},z_{2}\right\}$. Hence, the action is not topologically free.

``$\ensuremath{\Leftarrow}$'': Let $\left|S\right| \geq 3$ and assume that the action is not topologically free. We find $\mathbf{w}\in W\setminus\left\{ e\right\}$  such that $(\partial W)^{\mathbf{w}}$ contains an inner point. Without loss of generality we can assume that $\mathbf{w}$ with that property has minimal length. Fix some inner point $z\in(\partial W)^{\mathbf{w}}$. By Lemma \ref{technical} there exist $\mathbf{u},\mathbf{v}\in W$ with $\mathbf{w}=\mathbf{u}\mathbf{v}^{-1}$, $\left|\mathbf{w}\right|=\left|\mathbf{u}\right|+\left|\mathbf{v}\right|$ and $\mathbf{u},\mathbf{v}\leq z$. Let $\mathbf{u}=s_{1}...s_{n}$, $\mathbf{v}=t_{1}...t_{m}$ be reduced expressions for $\mathbf{u}$, $\mathbf{v}$. Without loss of generality one can assume that $n\leq m$. We claim that every letter of $\mathbf{u}$ commutes with every letter of $\mathbf{v}$ and that the letters are pairwise different.
\begin{itemize}
\item If $s_{1}=t_{1}$, then $(\partial W)^{(s_{2}...s_{n})(t_{m}...t_{2})}=s_{1}.(\partial W)^{\mathbf{w}}$. But we assumed $\mathbf{w}$ to have minimal length, so $s_{1}\neq t_{1}$. By $s_{1},t_{1}\leq z$ for $z\in(\partial W)^{\mathbf{w}}$ we further get $m_{s_{1}t_{1}}=2$.
\item If $s_{2}=t_{1}$, then $(\partial W)^{\left(s_{1}s_{3}...s_{n}\right)\left(t_{m}...t_{2}\right)}=s_{2}.(\partial W)^{\mathbf{w}}$. Again, by the minimality of $\mathbf{w}$ we get $s_{2}\neq t_{1}$ with $m_{s_{2}t_{1}}=2$. In the same way, $t_{2}\neq s_{1}$, $m_{s_{1}t_{2}}=2$ and $s_{2}\neq t_{2}$, $m_{s_{2}t_{2}}=2$.
\item ....
\end{itemize}
Proceeding like this we find that every letter of $\mathbf{u}$ commutes with every letter of $\mathbf{v}$ and that the letters are pairwise different.\\

\emph{Claim.} We have $\left|\mathbf{u}^{n}\right|=n\left|\mathbf{u}\right|$, $\left|\mathbf{v}^{n}\right|=n\left|\mathbf{v}\right|$ and $\mathbf{u}^{n},\mathbf{v}^{n}\leq z$ for every $n\in\mathbb{N}$.

\emph{Proof of Claim.} Let $\alpha$ be an infinite geodesic path representing $z$ with $\alpha_{0}=e$. By $\mathbf{u},\mathbf{v}\leq z$ and the above we can assume that $\alpha_{l}=t_{1}...t_{m}s_{1}...s_{n}\mathbf{w}_{l}$ for $l\geq m+n+1$ with $\left|\alpha_{l}\right|=\left|\mathbf{u}\right|+\left|\mathbf{v}\right|+\left|\mathbf{w}_{l}\right|$. The identities $\mathbf{u}\leq z$ and $z=\mathbf{u}\mathbf{v}^{-1}z$ imply that for every $i\in\left\{ 1,...,n-1\right\}$  one has $s_{i}s_{i+1}...s_{n}\left(t_{m}...t_{1}\right)z=\left(s_{i-1}...s_{1}\right)z\geq s_{i}$, so for $l$ large enough $s_{i}\leq s_{i}s_{i+1}...s_{n}s_{1}...s_{n}\mathbf{w}_{l}=s_{i}s_{i+1}...s_{n}\mathbf{u}\mathbf{w}_{l}$. We get $\left|s_{i}\left(s_{i+1}...s_{n}\right)\mathbf{u}\right|=\left|\left(s_{i+1}...s_{n}\right)\mathbf{u}\right|+1$ and hence (via induction over $i$, starting with $i=n$) that $\left|\mathbf{u}^{2}\right|=2\left|\mathbf{u}\right|$. This implies $\left|\mathbf{u}^{n}\right|=n\left|\mathbf{u}\right|$ for every $n\in\mathbb{N}$ and in a similar way $\left|\mathbf{v}^{n}\right|=n\left|\mathbf{v}\right|$ for every $n\in\mathbb{N}$. Now, since each letter of $\mathbf{u}$ commutes with each letter of $\mathbf{v}$, we have $\mathbf{u}^{-n}z=\mathbf{v}^{-n}z\geq\mathbf{u}$ for every $n\in\mathbb{N}$. Inductively we get that $\mathbf{u}^{n}\leq z$ for every $n\in\mathbb{N}$. In a similar way, $\mathbf{v}^{n}\leq z$ for every $n\in\mathbb{N}$. The claim follows.\\

The claim in particular implies that $\mathbf{v}^{-1}.z=z$ and hence $z\in (\partial W)^{\mathbf{u}}$. But then $\mathbf{w}=\mathbf{u}$ and $\mathbf{v}=e$ by the minimality of $\mathbf{w}$ and $n\leq m$. Heuristically, $z$ starts with arbitrarily large powers of $\mathbf{w}$, but there can also appear other expressions in front of $z$. To make this precise, for every $i\in\mathbb{N}$ one can find $\mathbf{w}_{i}\in W$ with $\mathbf{w}_{i}\mathbf{w}=\mathbf{w}\mathbf{w}_{i}$ and $\left|\mathbf{w}_{i}\mathbf{w}\right|=\left|\mathbf{w}_{i}\right|+\left|\mathbf{w}\right|$ such that $\mathbf{w}_{i}\mathbf{w}^{i}\rightarrow z$. Let $s,t\in S$ with $s\leq_{L}\mathbf{w}$, $m_{st}=\infty$ and write $\left(ts\right)^{\infty}:=\lim_{k}(ts)^{k}\in\partial W$. Assume that $\mathbf{w}$ is not of the form $\mathbf{w}=(st)^{l}$ for some $l\in\mathbb{N}$. Then $\mathbf{w}_{i}\mathbf{w}^{i}(ts)^{\infty}\notin(\partial W)^{\mathbf{w}}$ for every $i\in\mathbb{N}$. But $\mathbf{w}_{i}\mathbf{w}^{i}(ts)^{\infty}\in\partial W\setminus(\partial W)^{\mathbf{w}}$ is a sequence converging to $z$ which contradicts our assumption that $z$ is an inner point. Hence, $\mathbf{w}=(st)^{l}$ for some $l\in\mathbb{N}$. By the minimality of $\mathbf{w}$, $l=1$, so in particular $\mathbf{w}_{i}(st)^{i}\rightarrow z$. Because $\left|S\right|\geq3$ one can find $r\in S$ such that either $m_{sr}=\infty$ or $m_{tr}=\infty$. If $m_{sr}=\infty$, then $\mathbf{w}_{i}(st)^{i}s(rs)^{\infty}\in\partial W\setminus(\partial W)^{\mathbf{w}}$ is a sequence converging to $z$ and if $m_{tr}=\infty$, then $\mathbf{w}_{i}(st)^{i}(rt)^{\infty}\in\partial W\setminus(\partial W)^{\mathbf{w}}$ is a sequence converging to $z$ where $(rs)^{\infty}:=\lim_{k}(rs)^{k}$ and $(rt)^{\infty}:=\lim_{k}(rt)^{k}$. In both cases $z$ turns out not to be an inner point, in contradiction to our assumption. Hence, the action $W\curvearrowright\partial W$ must be topologically free.
\end{proof}

\begin{remark}
The proof of Proposition \ref{r.a.2} is direct and only uses combinatorial arguments. We chose to present it that way because of its self-containedness. However, the same statement can also be shown by an operator algebraic approach. Indeed, if $(W,S)$ is an irreducible right-angled Coxeter system with $3\leq\left|S\right|<\infty$, then $W$ is C$^\ast$-simple (see for instance \cite{Fe}, \cite{DeLaHarpe}, \cite{Cornulier} or \cite{Mario}). The C$^\ast$-simplicity and the minimality of the action $W\curvearrowright\partial(W,S)$ then imply with \cite[Theorem 7.1]{BKKO} that the reduced crossed product $C(\partial(W,S))\rtimes_{r}W$ is simple. By Corollary \ref{Total amenability} and \cite[Theorem 4.3.4]{BrownOzawa} the reduced crossed product coincides with the universal one. The topological freeness of the action $W\curvearrowright\partial(W,S)$ hence follows with \cite[Theorem 2]{Archbold}.
\end{remark}

\begin{lemma}
Let $\left(W,S\right)$ be a finite rank non-amenable Coxeter system that is small at infinity. Then the action $W\curvearrowright\partial(W,S)$ is topologically free.
\end{lemma}

\begin{proof}
As in the proof of Theorem \ref{s.a.i.}, the group $W$ is word hyperbolic and the boundary $\partial(W,S)$ coincides with the hyperbolic boundary $\partial_{h}W$. The topological freeness then follows from \cite[Corollaire 20]{Harpe}.
\end{proof}

An extension of the results above to broader classes (or even a complete characterization) of Coxeter systems $(W,S)$ whose respective boundary defines a boundary in the sense of Furstenberg and whose respective action $W\curvearrowright\partial(W,S)$ is topologically free would be very interesting.

Note that an action of a group $G$ on a compact Hausdorff space $X$ is minimal if and only if $C(X)$ does not contain any non-trivial $G$-invariant ideal. We close this section with a result that relates to the ideal structure of the C$^\ast$-algebra $C(\partial(W,S))$. Recall that by Proposition \ref{C*} (and Remark \ref{remark}), $\pi(\mathcal{D}(W,S))\cong C(\partial(W,S))$ via $\pi(P_{\mathbf{w}})\mapsto\chi_{\mathcal{U}_{\mathbf{w}}\cap\partial(W,S)}$ where $\mathcal{D}(W,S):=\mathcal{D}(\text{Cay}(W,S),e)$.

\begin{proposition}
Let $\left(W,S\right)$ be a finite rank Coxeter system and let $I$ be a non-zero ideal in $\pi(\mathcal{D}(W,S))$. Then $I$ intersects non-trivially with the $\ast$-algebra $\text{Span}\left\{ \pi(P_{\mathbf{w}})\mid\mathbf{w}\in W\right\} \subseteq\pi(\mathcal{D}(W,S))$.
\end{proposition}

\begin{proof}
Let $I$ be a non-zero ideal in $C(\partial W)\cong\pi(\mathcal{D}(W,S))$ and assume that $I$ intersects the $\ast$-algebra $\text{Span}\left\{ \chi_{\mathcal{U}_{\mathbf{w}}\cap\partial W}\mid\mathbf{w}\in W\right\}$ trivially. Denote the quotient map $C(\partial W)\rightarrow C(\partial W)/I$ by $\rho$. Let further $x:=\sum_{\mathbf{w}\in W}\lambda_{\mathbf{w}}\chi_{\mathcal{U}_{\mathbf{w}}\cap\partial W}\in C(\partial W)$ with $\lambda_{\mathbf{w}} \in \mathbb{C}$ be a non-zero element where we assume that the sum is finite. The space $\partial W$ is compact, therefore there exists $z\in\partial W$ with
\begin{eqnarray}
\nonumber
\left\Vert x\right\Vert =\left|\sum_{\mathbf{w}\in W:\mathbf{w}\leq z}\lambda_{\mathbf{w}}\right|.
\end{eqnarray}
Define the finite set $\mathfrak{S}:=\left\{ \mathbf{v}\in W\mid\lambda_{\mathbf{v}}\neq0\text{ and }\mathbf{v}\nleq z\right\}$ and let $(\alpha_{i})_{i\in\mathbb{N}}\subseteq W$ be an infinite geodesic path representing the element $z$. Then, for every $i\in\mathbb{N}$ the continuous function $\mathbf{P}_{i}:=\chi_{\mathcal{U}_{\alpha_{i}}\cap\partial W}\prod_{\mathbf{v}\in\mathfrak{S}}\chi_{\mathcal{U}_{\mathbf{v}}^{c}\cap\partial W}\in C(\partial W)$ is a projection with $\rho(\mathbf{P}_{i})\neq0$. Indeed, $\mathbf{P}_{i}(z)=1$ implies that $\mathbf{P}_{i}\neq0$ and hence $\rho(\mathbf{P}_{i})\neq0$ since $\mathbf{P}_{i}\in\text{Span}\left\{ \chi_{\mathcal{U}_{\mathbf{w}}\cap\partial W}\mid\mathbf{w}\in W\right\}$. We get that
\begin{eqnarray}
\nonumber
\left\Vert \rho(x)\right\Vert &\geq& \lim_{i\rightarrow\infty}\left\Vert \sum_{\mathbf{w}\in W}\lambda_{\mathbf{w}}\rho(\chi_{\mathcal{U}_{\mathbf{w}}\cap\partial W}\mathbf{P}_{i})\right\Vert \\
\nonumber
&=& \lim_{i\rightarrow\infty}\left\Vert \sum_{\mathbf{w}\in W:\mathbf{w}\notin\mathfrak{S}}\lambda_{\mathbf{w}}\rho(\chi_{\mathcal{U}_{\mathbf{w}\vee\alpha_{i}}\cap\partial W}\prod_{\mathbf{v}\in\mathfrak{S}}\chi_{\mathcal{U}_{\mathbf{v}}^{c}\cap\partial W})\right\Vert\\
\nonumber
&=& \lim_{i\rightarrow\infty}\left\Vert \left(\sum_{\mathbf{w}\in W:\mathbf{w}\leq z}\lambda_{\mathbf{w}}\right)\rho(\mathbf{P}_{i})\right\Vert\\
\nonumber
&=& \left\Vert x\right\Vert \text{.}
\end{eqnarray}
But then $\rho$ must be isometric, i.e. $I=0$ in contradiction to our assumption. We deduce the claim.
\end{proof}


\section{Applications to Hecke C$^\ast$-algebras} \label{applications} \label{section 4}

In the following section we will apply our earlier results to (operator algebras associated with) Hecke algebras by studying certain embeddings of Hecke C$^\ast$-algebras, the (strong) Akemann-Ostrand property of Hecke-von Neumann algebras associated with Coxeter groups which are small at infinity and properties that are widely related to injective envelopes of Hecke C$^\ast$-algebras.

Let $(W,S)$ be a finite rank Coxeter system. Recall that for $\mathbf{w}\in W$ we defined $P_{\mathbf{w}}\in\ell^{\infty}(W)\subseteq\mathcal{B}(\ell^{2}(W))$ to be the orthogonal projection onto the subspace
\begin{eqnarray}
\nonumber
\overline{\text{Span}\left\{ \delta_{\mathbf{v}}\mid\mathbf{v}\in W\text{ with }\mathbf{w}\leq\mathbf{v}\right\} }\subseteq\ell^{2}(W)\text{,}
\end{eqnarray}
with $P_{e}=1$. Further, we denoted the quotient map of $\mathcal{B}(\ell^{2}(W))$ onto $\mathcal{B}(\ell^{2}(W))/\mathcal{K}$ by $\pi$ where $\mathcal{K}:=\mathcal{K}(\ell^{2}(W))$ is the space of compact operators on $\ell^{2}(W)$. For every $q=(q_{s})_{s\in S}\in\mathbb{R}_{>0}^{(W,S)}$, $s\in S$ the operator $T_{s}^{(q)}$ can be written as $T_{s}^{(q)}=T_{s}^{(1)}+p_{s}(q)P_{s}$ and the map $q_{s}\mapsto \frac{q_{s}-1}{\sqrt{q_{s}}}$ is injective on $\mathbb{R}_{>0}$. This implies that for all $q^{1},q^{2}\in\mathbb{R}_{>0}^{(W,S)}$ with $q_{s}^{1}\neq q_{s}^{2}$, $s\in S$ the C$^\ast$-subalgebra $\mathfrak{A}(W)$ of $\mathcal{B}(\ell^{2}(W))$ generated by $C_{r,q^{1}}^{\ast}(W)$ and $C_{r,q^{2}}^{\ast}(W)$ does not depend on the choice of different parameters $q^{1}$ and $q^{2}$. It is the smallest C$^\ast$-subalgebra of $\mathcal{B}(\ell^{2}(W))$ that contains all Hecke C$^\ast$-algebras of the system $(W,S)$ and there exists a natural isomorphism $\iota:\mathfrak{A}(W)\cong C(\overline{W})\rtimes_{r}W$ via $\iota(P_{\mathbf{w}})=\chi_{\mathcal{U}_{\mathbf{w}}}$, $\iota(T_{\mathbf{w}}^{(1)})=\lambda_{\mathbf{w}}$ for $\mathbf{w} \in W$. Here, $C(\overline{W})\rtimes_{r}W\subseteq\mathcal{B}(\ell^{2}(W)\otimes\ell^{2}(W))$ denotes the reduced crossed product associated with the canonical action $W\curvearrowright\overline{W}$. The isomorphism is being implemented by conjugation with the unitary $U\in\mathcal{B}(\ell^{2}(W) \otimes \ell^2(W))$ defined by $U(\delta_{\mathbf{v}}\otimes\delta_{\mathbf{w}}):=\delta_{\mathbf{v}}\otimes\delta_{\mathbf{w}\mathbf{v}}$ for $\mathbf{v},\mathbf{w}\in W$ and Proposition \ref{C*}. In the same way, there exists an isomorphism $\kappa: \pi (\mathfrak{A}(W)) \cong C(\partial W)\rtimes_{r}W$ with $\kappa\circ\pi(P_{\mathbf{w}})=\chi_{\mathcal{U}_{\mathbf{w}}\cap C(\partial W)}$, $\kappa\circ \pi (T_{\mathbf{w}}^{(1)})=\lambda_{\mathbf{w}}$ for $\mathbf{w}\in W$. Corollary \ref{Total amenability} and \cite[Theorem 4.3.4]{BrownOzawa} imply the following statement.

\begin{corollary} \label{nuclear}
Let $(W,S)$ be a finite rank Coxeter system. Then the C$^\ast$-algebras $\mathfrak{A}(W)$ and $\pi(\mathfrak{A}(W))$ are nuclear.
\end{corollary}

The existence of the maps $\iota$ and $\kappa$ provides a direct link between the topological spaces $\overline{W}$, $\partial W$ (or their respective abelian C$^\ast$-algebras) and the Hecke C$^\ast$-algebras $C_{r,q}^{\ast}(W)$, $q\in\mathbb{R}_{>0}^{(W,S)}$. The aim of the following section is, among other things, to collect implications of the previous results that follow from this observation.

Let us first investigate, when the restriction of $\kappa\circ\pi$ to $C_{r,q}^{\ast}(W)$, $q\in \mathbb{R}_{>0}^{(W,S)}$ factors to an embedding of $C_{r,q}^{\ast}(W)$ into $C(\partial W)\rtimes_{r}W$. For this we will need the following inequality.

\begin{lemma} \label{inequality}
Let $(W,S)$ be a finite rank Coxeter system, $q=(q_{s})_{s\in S}\in\mathbb{R}_{>0}^{(W,S)}$ and $\xi\in\ell^{2}(W)$. Then, 
\begin{eqnarray}
\nonumber
\left(\prod_{i=1}^{n}\min\{ q_{s_{i}}^{1/2}, q_{s_{i}}^{-1/2}\} \right)\left\Vert \xi\right\Vert _{2}\leq\Vert T_{\mathbf{w}}^{(q)}\xi\Vert _{2}\leq\left(\prod_{i=1}^{n}\max\{ q_{s_{i}}^{1/2}, q_{s_{i}}^{-1/2}\} \right)\left\Vert \xi\right\Vert _{2}
\end{eqnarray}
for all $\mathbf{w}\in W$ where $\mathbf{w}=s_{1}...s_{n}$ is a reduced expression for $\mathbf{w}$.
\end{lemma}

\begin{proof}
For $q=(q_{s})_{s\in S}\in\mathbb{R}_{>0}^{(W,S)}$, $\xi \in \ell^2(W)$ we have
\begin{eqnarray}
\nonumber
\Vert T_{s}^{\left(q\right)}\xi\Vert _{2}^{2} &=& \left\langle (T_{s}^{\left(q\right)})^{2} \xi \text{, } \xi \right\rangle\\
\nonumber
&=& \left\langle (1+p_{s}(q)T_{s}^{\left(q\right)}) \xi \text{, } \xi \right\rangle\\
\nonumber
&=& \left\Vert \xi\right\Vert _{2}^{2}+p_{s}(q)\left\langle T_{s}^{\left(q\right)}\xi \text{, }\xi \right\rangle
\end{eqnarray}
for every $s\in S$. If we assume that $0<q_{s}\leq1$ one gets $p_s(q) \leq 0$ and hence
\begin{eqnarray}
\nonumber
\left\Vert \xi\right\Vert _{2}^{2} &=& \Vert T_{s}^{\left(q\right)}\xi\Vert _{2}^{2}-p_{s}(q)\langle T_{s}^{\left(q\right)}\xi \text{, }\xi \rangle\\
\nonumber
& \leq& \Vert T_{s}^{\left(q\right)}\xi\Vert _{2}^{2}-p_{s}(q)\Vert T_{s}^{\left(q\right)}\xi\Vert _{2}\left\Vert \xi\right\Vert _{2}
\end{eqnarray}
and hence
\begin{eqnarray}
\nonumber
\left\Vert \xi\right\Vert _{2}^{2}+p_{s}(q)\Vert T_{s}^{\left(q\right)}\xi\Vert _{2}\left\Vert \xi\right\Vert _{2}-\Vert T_{s}^{\left(q\right)}\xi\Vert _{2}^{2}\leq0\text{.}
\end{eqnarray}
By solving the quadratic equation this implies $\left\Vert \xi\right\Vert _{2}\leq\frac{1}{\sqrt{q_{s}}}\Vert T_{s}^{(q)}\xi\Vert _{2}$. If $q_{s}\geq1$, one has in the same way $\left\Vert \xi\right\Vert _{2}\leq\sqrt{q_{s}}\Vert T_{s}^{(q)}\xi\Vert _{2}$. The left inequality then follows via induction. The right inequality is immediate from $\Vert T_{s}^{\left(q\right)}\Vert=\max\{ q_{s}^{1/2}, q_s^{-1/2}\}$.
\end{proof}

For a finite rank Coxeter system $(W,S)$ and $z\in\mathbb{C}^{(W,S)}$ define $z_{\mathbf{w}}:=z_{s_{1}}...z_{s_{n}}$ where $\mathbf{w}=s_{1}...s_{n}$ is a reduced expression for $\mathbf{w}\in W$. Again, this does not depend on the choice of the reduced expression. The growth series of $W$ is the power series in $z$ given by \begin{eqnarray} \nonumber W(z):=\sum_{\mathbf{w}\in W}z_{\mathbf{w}}\text{.} \end{eqnarray} We denote its region of convergence by $\mathcal{R}$ and set
\begin{eqnarray}
\nonumber
\mathcal{R}^{\prime}:=\left\{ (q_{s}^{\epsilon_{s}})_{s\in S}\mid q\in\mathcal{R}\cap\mathbb{R}_{>0}^{(W,S)},\epsilon\in\left\{ -1,1\right\} ^{(W,S)}\right\} \text{.}
\end{eqnarray}

For right-angled Coxeter systems the following statement is an immediate consequence of the results in \cite{Raum} (see also \cite{Gar}). It further follows that for those Coxeter groups the intersection $\mathcal{N}_{q}(W)\cap\mathcal{K}$ is exactly one or zero dimensional.

\begin{theorem} \label{compacts}
Let $(W,S)$ be a finite rank Coxeter system with $W$ being infinite. Further let $q\in\mathbb{R}_{>0}^{(W,S)}$. Then $\mathcal{N}_{q}(W)\cap\mathcal{K}\neq0$ if and only if $q\in\mathcal{R}^{\prime}$.
\end{theorem}

\begin{proof}
``$\ensuremath{\Rightarrow}$'': The construction of the one-dimensional projection in the proof of \cite[Theorem 5.3]{Gar} translates to the multi-parameter case (compare also with \cite[Lemma 19.2.5]{Davis} but note that our notational conventions differ slightly). It implies that for $q\in\mathcal{R}\cap\mathbb{R}_{>0}^{(W,S)}$ the intersection $\mathcal{N}_{q}(W)\cap\mathcal{K}$ is non-trivial. For general $q\in\mathcal{R}^{\prime}$ note that the isomorphism in \cite[Proposition 4.7]{Mario} is unitarily implemented and extends to the von Neumann algebraic level. Hence, the non-triviality of $\mathcal{N}_{q}(W)\cap\mathcal{K}$ follows from the above.

``$\ensuremath{\Leftarrow}$'': First let $q\in\mathbb{R}_{>0}^{(W,S)}\setminus\mathcal{R}$ with $0<q_{s}\leq1$ for every $s\in S$ and assume that $\mathcal{N}_{q}(W)\cap\mathcal{K}\neq0$. Then, $\mathcal{N}_{q}(W)\cap\mathcal{K}$ contains a non-zero positive operator and also its finite-rank spectral projections. Since $\mathcal{N}_{q}(W)\cong\mathcal{N}_{q}^{r}(W)$ via $T_{\mathbf{w}}^{(q)}\mapsto T_{\mathbf{w}}^{(q),r}=U^{\ast}T_{\mathbf{w}}^{(q)}U$ for $\mathbf{w}\in W$, where $U\in\mathcal{B}(\ell^{2}(W))$ is the unitary that maps an orthonormal basis element $\delta_{\mathbf{v}}$ to $\delta_{\mathbf{v}^{-1}}$, this implies that $\mathcal{N}_{q}^{r}(W)$ also contains a finite-rank projection which we denote by $P$. Since $P$ commutes with the elements in $\mathcal{N}_{q}(W)$, the Hilbert subspace $\mathcal{H}:=P\ell^{2}(W)$ is invariant under $\mathcal{N}_{q}(W)$. Let $(\xi_{i})_{i=1,...,n}$ be an orthonormal basis of $P\ell^{2}(W)$. Then, by Lemma \ref{inequality},
\begin{eqnarray}
\nonumber
\left\Vert P\delta_{e}\right\Vert _{2}^{2}\leq q_{\mathbf{w}}^{-1}\Vert T_{\mathbf{w}}^{\left(q\right)}P\delta_{e}\Vert _{2}^{2}=\sum_{i=1}^{n}q_{\mathbf{w}}^{-1}|\langle \xi_{i},T_{\mathbf{w}}^{\left(q\right)}P\delta_{e}\rangle |^{2}=\sum_{i=1}^{n}q_{\mathbf{w}}^{-1}\left|\left\langle \xi_{i},\delta_{\mathbf{w}}\right\rangle \right|^{2}
\end{eqnarray}
for every $\mathbf{w}\in W$. Let us distinguish two cases:\\

\begin{itemize}
\item \emph{Case 1:} Assume that there exists a constant $C>0$ such that for every $\mathbf{w}\in W$ there exists some $1\leq i\leq n$ with $q_{\mathbf{w}}^{-1}\left|\left\langle \xi_{i},\delta_{\mathbf{w}}\right\rangle \right|^{2}>C$. We get that
\begin{eqnarray}
\nonumber
\sum_{i=1}^{n}\left\Vert \xi_{i}\right\Vert _{2}^{2}=\sum_{i=1}^{n}\sum_{\mathbf{w}\in W}\left|\left\langle \xi_{i},\delta_{\mathbf{w}}\right\rangle \right|^{2}>C\sum_{\mathbf{w}\in W}q_{\mathbf{w}}\text{.}
\end{eqnarray}
But $q\in\mathbb{R}_{>0}^{(W,S)}\setminus\mathcal{R}$, so the sum on the right-hand side diverges in contradiction to $\sum_{i=1}^{n}\left\Vert \xi_{i}\right\Vert _{2}^{2}<\infty$. This implies that there exists no such constant $C$.
\item \emph{Case 2:} Assume that there exists a sequence $(\mathbf{w}_{j})_{j\in\mathbb{N}}\subseteq W$ with $q_{\mathbf{w}_{j}}^{-1}\left|\left\langle \xi_{i},\delta_{\mathbf{w}_{j}}\right\rangle \right|^{2}\rightarrow0$ for $1\leq i\leq n$. Then,
\begin{eqnarray}
\nonumber
\left\Vert P\delta_{e}\right\Vert _{2}^{2}\leq\sum_{i=1}^{n}q_{\mathbf{w}_{j}}^{-1}\left|\left\langle \xi_{i},\delta_{\mathbf{w}_{j}}\right\rangle \right|^{2}\rightarrow0\text{,}
\end{eqnarray}
i.e. $ P\delta_{e}=0$. But then $P=0$ since $P\delta_{\mathbf{w}}=PT_{\mathbf{w}}^{(q)}\delta_{e}=T_{\mathbf{w}}^{(q)}P\delta_{e}=0$ for every $\mathbf{w}\in W$. This is a contradiction to our assumption.
\end{itemize}

Since both cases lead to a contradiction, the intersection $\mathcal{N}_{q}(W)\cap\mathcal{K}$ must be trivial. Again, for general $q\in\mathbb{R}_{>0}^{(W,S)}\setminus\mathcal{R}^{\prime}$ the statement follows with \cite[Proposition 4.7]{Mario}.
\end{proof}

\begin{corollary} \label{embedding2}
Let $(W,S)$ be a finite rank Coxeter system. For $q\in\mathbb{R}_{>0}^{(W,S)}\setminus\mathcal{R}^{\prime}$ the map $\kappa\circ(\pi|_{\mathfrak{A}(W)}):\mathfrak{A}(W)\rightarrow C(\partial W)\rtimes_{r}W $ restricts to an embedding of $C_{r,q}^{\ast}(W)$ into $C(\partial W)\rtimes_{r}W$.
\end{corollary}


\subsection{The Akemann-Ostrand property for Hecke-von Neumann algebras} \label{Akemann}

In this subsection we will make use of a method used by Higson and Guentner (see \cite{HiGu}) in the context of word hyperbolic groups, to show the (strong) Akemann-Ostrand property for certain Hecke-von Neumann algebras. The same approach has been made in \cite[Section 5]{Caspers}. However, the proof presented there contains a gap since for general word hyperbolic Coxeter systems $(W,S)$ the space $\partial(W,S)$ does not identify with the hyperbolic boundary $\partial_{h}W$. We correct it in the case of Coxeter groups which are small at infinity.

\begin{definition} [{\cite[Definition 2.6]{Isono}}] \label{Definition}
Let $\mathcal{M}$ be a von Neumann algebra and $(\mathcal{M},\mathcal{H},J,\mathfrak{P})$ a standard form for $\mathcal{M}$. We say that $\mathcal{M}$ satisfies the \emph{strong Akemann-Ostrand condition} (\emph{strong condition $(\mathcal{A}O)$}) if there exist unital C$^\ast$-subalgebras $A\subseteq\mathcal{M}$, $\mathcal{C}\subseteq\mathcal{B}(\mathcal{H})$ such that:

\begin{enumerate}[label=(\arabic*)]
\item $A$ is exact and $\sigma$-weakly dense in $\mathcal{M}$;
\item $\mathcal{C}$ is nuclear and contains $A$;
\item  The set of commutators $\left[\mathcal{C},JAJ\right]:=\left\{ \left[c,JaJ\right]\mid c\in\mathcal{C},a\in A\right\}$  is contained in the compact operators $\mathcal{K}$.
\end{enumerate}
\end{definition}

\begin{theorem} \label{AO}
Let $(W,S)$ be a finite rank Coxeter system that is small at infinity and let $q\in\mathbb{R}_{>0}^{(W,S)}$. Then the Hecke-von Neumann algebra $\mathcal{N}_{q}(W)\subseteq\mathcal{B}(\ell^{2}(W))$ satisfies the strong condition $\left(\mathcal{AO}\right)$.
\end{theorem}

\begin{proof}
Set $A:=C_{r,q}^{\ast}(W)$ and $\mathcal{C}:=\mathfrak{A}(W)$. Property (1) of Definition \ref{Definition} follows from \cite[Theorem 6.1]{Mario}. Further, the nuclearity of $\mathcal{C}$ is clear by Corollary \ref{nuclear} (or follows from Corollary \ref{hyperbolic 2} and Theorem \ref{hyperbolic 3}). It remains to show that $\left[C,JAJ\right]\subseteq\mathcal{K}$ where $JAJ=C_{r,q}^{\ast,r}(W)$. Note that $\mathfrak{A}(W)$ is the unital C$^\ast$-subalgebra of $\mathcal{B}(\ell^{2}(W))$ generated by all operators $T_{s}^{(1)}$, $s\in S$ and $P_{\mathbf{w}}$, $\mathbf{w}\in W$. One can further write $T_{s}^{(q),r}=T_{s}^{(1),r}+p_{s}(q)P_{s}^{r}$ for all $s\in S$. We assumed $W$ to be small at infinity, so by \cite[Lemma 5.3.17]{BrownOzawa}
\begin{eqnarray}
\nonumber
\left[T_{s}^{(1)},P_{\mathbf{w}}^{r}\right]\in\mathcal{K} \text{, }\left[P_{\mathbf{w}},T_{s}^{(1),r}\right]\in\mathcal{K}
\end{eqnarray}
for $s\in S$, $\mathbf{w}\in W$ and further
\begin{eqnarray}
\nonumber
\left[T_{s}^{(1)},T_{t}^{(1),r}\right]=\left[P_{\mathbf{v}},P_{\mathbf{w}}^{r}\right]=0
\end{eqnarray}
for all $s,t\in S$ and $\mathbf{v},\mathbf{w}\in W$. Therefore, $\left[C,JAJ\right]\subseteq\mathcal{K}$.
\end{proof}

\begin{remark}
Theorem \ref{AO} implies in combination with \cite[Remark 2.7]{Isono} that Hecke-von Neumann algebras of Coxeter systems that are small at infinity satisfy Ozawa's property $(\mathcal{AO})$ (see \cite{Ozawa}) and Isono's property $(\mathcal{AO})^{+}$ (see \cite{Isono2}). Hence, we get from \cite[Theorem 6]{Ozawa} that these von Neumann algebras are solid, meaning that the relative commutant of any diffuse von Neumann subalgebra is injective. Further, if the Hecke-von Neumann algebra is a $\text{II}_{1}$-factor satisfying the weak-$\ast$ completely bounded approximation property, \cite[Theorem A]{Isono2} implies that it is strongly solid. The results in \cite{Isono2} rely on \cite{PopaVaes} and \cite{OzawaPopa}.
\end{remark}

Garncarek observed in \cite[Section 6]{Gar} that the interpolated free group factors $\mathcal{L}(\mathbb{F}_{t})$, $t\in\mathbb{R}_{>1}$ (cf. \cite{Dykema}, \cite{Radulescu}) can be realized as Hecke-von Neumann algebras of free products of finite right-angled Coxeter groups. For instance, for the Coxeter group $W:=(\mathbb{Z}_{2})^{\ast l}$ with $l\geq3$ one has $\mathcal{N}_{q}(W)\cong\mathcal{L}(\mathbb{F}_{2lq(1+q)^{-2}})$ for all $q\in\left[(l-1)^{-1},1\right]$. The interpolated free group factors are closely related to the famous free factor problem. By \cite{Dykema}, \cite{Radulescu}, they are either all isomorphic to each other or they are all non-isomorphic. Ozawa and Popa showed in \cite{OzawaPopa} that the interpolated free group factors are strongly solid which strengthens earlier indecomposability results by Voiculescu \cite{Voiculescu} and Ozawa \cite{Ozawa}.

The following corollary is an immediate consequence of Proposition \ref{free}, Theorem \ref{AO} and the discussion above. The statement is known to experts.

\begin{corollary}
For every $t\in\mathbb{R}_{>1}$ the interpolated free group factors $\mathcal{L}(\mathbb{F}_{t})$ satisfies the strong condition $\left(\mathcal{AO}\right)$.
\end{corollary}

In the context of group algebras, property $(\mathcal{AO})$ has a number of interesting applications (see for instance \cite{Delaroche}). In particular, it relates to Connes's notion of fullness, introduced in \cite{Connes}. Recall that a factor $\mathcal{M}$ is said to be \emph{full} if for every bounded net $(x_{i})_{i\in I}\subseteq\mathcal{M}$ with $\lim_{i}\left\Vert \varphi(x_{i}\cdot)-\varphi(\cdot x_{i})\right\Vert =0$ for all $\varphi\in\mathcal{M}_{\ast}$, there exists a bounded net $(z_{i})_{i\in I}\subseteq\mathbb{C}$ with $x_{i}-z_{i}\rightarrow0$ in the strong operator topology. In the case of type $\text{II}_{1}$-factors this definition is equivalent to $\mathcal{M}$ not having Murray and von Neumann's \emph{property Gamma} (see \cite{Neumann}). In \cite{Connes2} Connes proved that a $\text{II}_{1}$-factor $\mathcal{M}$ is full if and only if $C^{\ast}(\mathcal{M},\mathcal{M}^{\prime})\cap\mathcal{K}(L^{2}(\mathcal{M}))\neq0$ where $L^{2}(\mathcal{M})$ denotes the GNS-space associated with the tracial state of $\mathcal{M}$.

Compare the following proposition with the results in \cite{Skandalis} and \cite{Delaroche}. The proof is close to \cite[Proposition 6.19]{Delaroche}.

\begin{proposition} \label{II1}
Let $(W,S)$ be a finite rank non-amenable Coxeter system which is small at infinity and let $q\in\mathbb{R}_{>0}^{(W,S)}$. Then, 
\begin{eqnarray}
\nonumber
C^{\ast}\left(C_{r,q}^{\ast}(W),C_{r,q}^{\ast,r}(W)\right)\cap\mathcal{K}\neq0\text{.}
\end{eqnarray}
If the corresponding Hecke-von Neumann algebra is a $\text{II}_{1}$-factor, then $\mathcal{N}_{q}(W)$ is full and
\begin{eqnarray}
\nonumber
C_{r,q}^{\ast}(W)\otimes C_{r,q}^{\ast,r}(W)\cong\pi\left(C^{\ast}\left(C_{r,q}^{\ast}(W),C_{r,q}^{\ast,r}(W)\right)\right)\text{.}
\end{eqnarray}
\end{proposition}

\begin{proof}
Set $\mathcal{A}:=C^{\ast}(C_{r,q}^{\ast}(W),C_{r,q}^{\ast,r}(W))\subseteq\mathcal{B}(\ell^{2}(W))$. By (the proof of) Theorem \ref{AO} the map $C_{r,q}^{\ast}(W)\odot C_{r,q}^{\ast,r}(W)\rightarrow\mathcal{B}(\ell^{2}(W))/\mathcal{K}$ given by $x\otimes y\mapsto xy + \mathcal{K}$ is continuous with respect to the minimal tensor norm. Denote the corresponding extension by $\rho$. Let $\mu:C_{r,q}^{\ast}(W)\otimes_{max}C_{r,q}^{\ast,r}(W)\rightarrow\mathcal{B}(\ell^{2}(W))$ and $Q:C_{r,q}^{\ast}(W)\otimes_{max}C_{r,q}^{\ast,r}(W)\rightarrow C_{r,q}^{\ast}(W)\otimes C_{r,q}^{\ast,r}(W)$ be the canonical maps. Then, $\pi\circ\mu=\rho\circ Q$. Since by our assumption $W$ is non-affine one can find an element $x\in C_{r,q}^{\ast}(W)\otimes_{max}C_{r,q}^{\ast,r}(W)$ with $\mu(x)\neq0$ and $Q(x)=0$. Indeed, if no such element exists then $\ker(Q)\subseteq\ker(\mu)$ and therefore the map $C_{r,q}^{\ast}(W)\odot C_{r,q}^{\ast,r}(W)\rightarrow\mathcal{B}(\ell^{2}(W))$ given by $T_{\mathbf{v}}^{(q)}\otimes T_{\mathbf{w}}^{(q),r}\mapsto T_{\mathbf{v}}^{(q)}T_{\mathbf{w}}^{(q),r}$ is continuous with respect to the minimal tensor norm. With \cite[Theorem 6.2.7]{BrownOzawa} and \cite[Theorem 6.2]{Mario} this leads to a contradiction. So let $x$ be an element with $\mu(x)\neq0$, $Q(x)=0$. Then, $0\neq\mu(x)\in\mathcal{A}\cap\mathcal{K}$ because $\pi\circ\mu(x)=\rho\circ Q(x)=0$.

In the case of a $\text{II}_{1}$-factor the fullness of $\mathcal{N}_{q}(W)$ follows from the discussion above. For the deduction of the existence of the isomorphism it suffices to show that $\rho$ is isometric. It is well-known (see for instance \cite[Cor. 4.1.10]{Dixmier}) that a C$^\ast$-algebra acting irreducibly on a Hilbert space $\mathcal{H}$ that intersects non-trivially with the compact operators on $\mathcal{H}$ contains all compact operators. Since, by the factoriality of $\mathcal{N}_{q}(W)$, the commutant of $\mathcal{A}$ is trivial, one hence gets $\mathcal{K}\subseteq\mathcal{A}$. We claim that 
\begin{eqnarray}
\nonumber
|||\cdot|||:\mathcal{N}_{q}(W)\odot\mathcal{N}_{q}(W)\rightarrow\mathbb{R}_{+}\text{, }\sum_{i}x_{i}\otimes y_{i}\mapsto\left\Vert \sum x_{i}(Jy_{i}J)+\mathcal{K}\right\Vert
\end{eqnarray}
defines a C$^\ast$-norm on $\mathcal{N}_{q}(W)\odot\mathcal{N}_{q}(W)$ where $J$ is the modular conjugation operator. Indeed, the only property that is not obvious is the definiteness of $|||\cdot|||$. It follows from the fact that the norm closure of
\begin{eqnarray}
\nonumber
\left\{ x\in\mathcal{N}_{q}(W)\odot\mathcal{N}_{q}(W)\mid|||x|||=0\right\} \subseteq\mathcal{N}_{q}(W)\odot\mathcal{N}_{q}(W)
\end{eqnarray}
is an ideal in $\mathcal{N}_{q}(W)\otimes\mathcal{N}_{q}(W)$, that $\mathcal{N}_{q}(W)$ is a $\text{II}_{1}$-factor (i.e. simple as a C$^\ast$-algebra) and that the spatial tensor product of two simple C$^\ast$-algebras is simple. Hence, $|||\cdot|||$ defines a C$^\ast$-norm. In particular, it majorizes the minimal tensor norm on $\mathcal{N}_{q}(W)\odot\mathcal{N}_{q}(W)$ so $\rho$ is indeed isometric.
\end{proof}


\subsection{Injective envelopes of Hecke C$^\ast$-algebras}

In \cite{KalantarKennedy} Kalantar and Kennedy built a connection between Hamana's theory of (C$^\ast$-dynamical) injective envelopes (see \cite{Hamana1}, \cite{Hamana2}, \cite{Hamana3}, \cite{Hamana4}) and Furstenberg's notion of boundary actions (see \cite{Fu1}, \cite{Fu2}). They used this connection to reformulate the longstanding open problem to determine which discrete groups are C$^\ast$-simple (meaning that the corresponding reduced group C$^\ast$-algebra is simple) in terms of the structure of the action of the group on its Furstenberg boundary. Their work led to a number of important results (see e.g. \cite{BKKO}, \cite{Haagerup}, \cite{Kennedy} and \cite{Adam}), some of which we will make use of to pick up operator algebraic implications of our earlier results.

One of the main results in \cite{KalantarKennedy} states that a discrete group $G$ is C$^\ast$-simple if and only if it admits a topologically free action on some $G$-boundary. This in particular implies that non-amenable word hyperbolic groups are C$^\ast$-simple (and have unique tracial state). In the context of right-angled Coxeter groups the theorem leads (in combination with Theorem \ref{r.a.} and Proposition \ref{r.a.2}) to a new proof of a well-known C$^\ast$-simplicity and trace-uniqueness result (see \cite{Fe}, \cite{DeLaHarpe}, \cite{Cornulier} and \cite{Mario}).

\begin{corollary}
Let $(W,S)$ be a right-angled irreducible Coxeter system with $3\leq\left|S\right|<\infty$. Then the reduced group C$^\ast$-algebra $C_{r}^{\ast}(W)$ is simple and has unique tracial state.
\end{corollary}

Also in the understanding of simplicity and trace-uniqueness of reduced crossed products the content of \cite{KalantarKennedy} led to new insights. The corresponding results apply to the C$^\ast$-algebras $\mathfrak{A}(W)$ and $\pi(\mathfrak{A}(W))$ which, as we observed earlier, identify with C$^\ast$-algebraic crossed products.

\begin{corollary} \label{simplicity}
Let $(W,S)$ be a finite rank Coxeter system. Assume that $W$ is either small at infinity and non-amenable or that the system is irreducible and right-angled with $\left|S\right|\geq3$. Then $\pi(\mathfrak{A}(W))$ is simple. In particular, $\mathfrak{A}(W)$ contains $\mathcal{K}\cap\mathfrak{A}(W)$ as its unique non-trivial ideal. Further, $C(\overline{(W,S)})$ and $C(\partial(W,S))$ carry no $W$-invariant tracial states and both $\mathfrak{A}(W)$ and $\pi(\mathfrak{A}(W))$ are traceless.
\end{corollary}

\begin{proof}
The simplicity of $\pi(\mathfrak{A}(W))$ follows from the fact that $\pi(\mathfrak{A}(W))\cong C(\partial W)\rtimes_{r}W$ in combination with Theorem \ref{s.a.i.}, Theorem \ref{r.a.} and \cite[Corollary 7.5]{BKKO}. Let $I\vartriangleleft\mathfrak{A}(W)$ be a non-trivial ideal. By Lemma \ref{topfree}, Corollary \ref{Total amenability} and \cite[Theorem 2]{Archbold}, $I$ intersects non-trivially with the unital C$^\ast$-algebra $\mathcal{D}(W,S)$ generated by all $P_\mathbf{w}$, $\mathbf{w}\in W$. Since $\mathcal{K}\cap\mathcal{D}(W,S)$ is a $W$-equivariant ideal in $\mathcal{D}(W,S)$ it is easy to see that $\mathcal{K}\cap\mathcal{D}(W,S)\subseteq I\cap\mathcal{D}(W,S)$, hence $\mathcal{K}\cap\mathfrak{A}(W)\subseteq I$. But by the simplicity of $\pi(\mathfrak{A}(W))$, $\mathcal{K}\cap\mathfrak{A}(W)$ is a maximal non-trivial ideal, so $I=\mathcal{K}\cap\mathfrak{A}(W)$. We get that the C$^\ast$-algebra $\mathfrak{A}(W)$ contains $\mathcal{K}\cap\mathfrak{A}(W)$ as its unique non-trivial ideal.

To show the remaining statements, it suffices to show that $C(\overline{W})$ carries no $W$-invariant tracial state. Indeed, if $C(\overline{W})$ carries no $W$-invariant tracial state then $C(\partial W)$ obviously also carries no $W$-invariant state. That $\mathfrak{A}(W))\cong C(\overline{W})\rtimes_{r}W$ and $\pi(\mathfrak{A}(W))\cong C(\partial W)\rtimes_{r}W$ are both traceless then follows with \cite[Corollary 4]{Kennedy}. So let us show that $C(\overline{W})$ carries no $W$-invariant state. For this, assume that $\tau$ is such a state. Remark \ref{limitelements} implies that $\delta_{z}\in\overline{W.\left\{ \tau\right\} }=\left\{ \tau\right\}$  for some $z\in\partial W$, i.e. $\tau=\delta_{z}$. But $\delta_{z}$ is obviously not $W$-invariant. This leads to a contradiction.
\end{proof}

One of the main ideas in \cite{KalantarKennedy} is the observation of the fact that for a discrete group $G$ the $G$-injective envelope $I_{G}(\mathbb{C})$ of $\mathbb{C}$ (i.e. the unique $G$-injective and $G$-essential extension of $\mathbb{C})$ carries a natural C$^\ast$-algebra structure for which $I_{G}(\mathbb{C})\cong C(\partial_{F}G)$. Here, $\partial_{F}G$ denotes the Furstenberg boundary of the group $G$. The construction of $\partial_{F}G$ by means of Hamana's theory of $G$-injective envelopes implies some powerful rigidity results.

In \cite{Ozawa2} Ozawa conjectured that for every exact C$^\ast$-algebra $\mathcal{A}$ there is a nuclear C$^\ast$-algebra $N(\mathcal{A})$ such that $\mathcal{A}\subseteq N(\mathcal{A})\subseteq I(\mathcal{A})$. Here $I(\mathcal{A})$ denotes the injective envelope of $\mathcal{A}$. (For more information on operator systems and ($G$-)injective envelopes we refer to \cite{Hamana1}, \cite{Hamana2},\cite{Hamana3}, \cite{Hamana4} and Paulsen's book \cite{Paulsen}.) Embeddings of this form have the striking advantage that properties of the larger C$^\ast$-algebra (for instance simplicity and primeness) are reflected by the properties of $\mathcal{A}$. Ozawa proved his conjecture in the case of reduced group C$^\ast$-algebras of word hyperbolic groups $G$ by choosing $N(C_{r}^{\ast}(G))$ to be the crossed product $C(\partial_{h}G)\rtimes_{r}G$. Kalantar and Kennedy extended his result in \cite[Section 4]{KalantarKennedy} to general exact group C$^\ast$-algebras by replacing the crossed product by $C(\partial_{h}G)$ by the crossed product $C(\partial_{F}G)\rtimes_{r}G$. However, in full generality Ozawa's conjecture remains a major open problem.

The following corollary provides an embedding of certain Hecke C$^\ast$-algebras which is similar to the one above.

\begin{proposition} \label{envelope}
Let $(W,S)$ be a finite rank Coxeter system. Assume that $W$ is either small at infinity or that the system is irreducible and right-angled with $\left|S\right|\geq3$. Then $I_{W}(C(\partial(W,S)))=C(\partial_{F}W)$ where $I_{W}(C(\partial(W,S)))$ denotes the $W$-injective envelope of $C(\partial(W,S))$. Further, for every $q\in\mathbb{R}_{>0}^{(W,S)}\setminus\mathcal{R}^{\prime}$ there are natural embeddings
\begin{eqnarray}
\nonumber
C_{r,q}^{\ast}(W)\hookrightarrow C(\partial(W,S))\rtimes_{r}W\hookrightarrow C(\partial_{F}W)\rtimes_{r}W\hookrightarrow I(C_{r}^{\ast}(W))\text{.}
\end{eqnarray}
In particular, $I(C_{r,q}^{\ast}(W))\hookrightarrow I(C_{r}^{\ast}(W))$ for every $q\in\mathbb{R}_{>0}^{(W,S)}\setminus\mathcal{R}^{\prime}$.
\end{proposition}

\begin{proof}
By Theorem \ref{r.a.} and Theorem \ref{s.a.i.}, $\partial W$ is a $W$-boundary. Hence, $\partial W$ is a continuous $W$-equivariant image of the Furstenberg boundary $\partial_{F}W$ (see \cite[Proposition 4.6]{Fu2}). This induces a $W$-equivariant embedding $C(\partial W)\hookrightarrow C(\partial_{F}W)$. The equality $I_{W}(C(\partial(W,S)))=C(\partial_{F}W)$ then follows in the same way as in the proof of \cite[Corollary 5.5]{KalantarKennedy}. We further deduce the existence of the chain $C_{r,q}^{\ast}(W)\hookrightarrow C(\partial W)\rtimes_{r}W\hookrightarrow C(\partial_{F}W)\rtimes_{r}W\hookrightarrow I(C_{r}^{\ast}(W))$ of inclusions from Corollary \ref{embedding2}, from \cite[Theorem 3.4]{Hamana4} and by extending the $W$-equivariant embedding $C(\partial W)\hookrightarrow C(\partial_{F}W)$ to an embedding $C(\partial W)\rtimes_{r}W\hookrightarrow C(\partial_{F}W)\rtimes_{r}W$ of the corresponding crossed products.

It remains to show that $I(C_{r,q}^{\ast}(W)) \hookrightarrow I(C_{r}^{\ast}(W))$ for every $q\in\mathbb{R}_{>0}^{(W,S)}\setminus\mathcal{R}^{\prime}$. But this is clear since, by the injectivity of $I(C_{r}^{\ast}(W))$ and $C_{r,q}^{\ast}(W)\hookrightarrow I(C_{r}^{\ast}(W))$, the injective envelope $I(C_{r,q}^{\ast}(W))$ is contained in $I(C_{r}^{\ast}(W))$ as an operator system. Hence, every completely positive projection $\theta:\mathcal{B}(\ell^{2}(W))\rightarrow I(C_{r}^{\ast}(W))$ restricts to the identity on $I(C_{r,q}^{\ast}(W))$. But the C$^\ast$-algebra structure of $I(C_{r}^{\ast}(W))$ is given by the Choi-Effros product associated with $\theta$, so this induces an embedding $I(C_{r,q}^{\ast}(W))\hookrightarrow I(C_{r}^{\ast}(W))$. 
\end{proof}

\begin{remark}
Proposition \ref{envelope} holds for all Coxeter systems whose action $W\curvearrowright\partial(W,S)$ is a boundary action. Considering Ozawa's conjecture it would be interesting to know if the embedding $I(C_{r,q}^{\ast}(W))\hookrightarrow I(C_{r}^{\ast}(W))$, $q\in\mathbb{R}_{>0}^{(W,S)}\setminus\mathcal{R}^{\prime}$ is always surjective, i.e. if $I(C_{r,q}^{\ast}(W))$ does not depend on the choice of the parameter $q$. In that case, $C_{r,q}^{\ast}(W)$ would turn out to be a prime C$^\ast$-algebra for all $q\in\mathbb{R}_{>0}^{(W,S)}\setminus\mathcal{R}^{\prime}$ (see \cite[Theorem 3.4]{Hamana4}).
\end{remark}

A complete classification of Coxeter systems (and ranges of parameters $q$) that give rise to Hecke-von Neumann algebras which are $\text{II}_{1}$-factors is still an open problem. Partial results have been obtained in \cite{Gar}, \cite{Mario} and \cite{Raum}. Considering Proposition \ref{II1}, a factoriality result would be particularly interesting in the case of systems which are small at infinity. We close this section with the following proposition which treats a similar question.

\begin{proposition}
Let $(W,S)$ be a finite rank Coxeter system that is small at infinity. Then $C_{r,q}^{\ast}(W)\cap C_{r,q}^{\ast,r}(W)=\mathbb{C}1$ for every $q\in\mathbb{R}_{>0}^{(W,S)}\setminus\mathcal{R}^{\prime}$.
\end{proposition}

\begin{proof}
Let $(W,S)$ be a finite rank Coxeter system that is small at infinity, let $q\in\mathbb{R}_{>0}^{(W,S)}\setminus\mathcal{R}^{\prime}$ and $x\in C_{r,q}^{\ast}(W)\cap C_{r,q}^{\ast,r}(W)$. By the same argument as in the proof of Theorem \ref{AO} we have that for every $y\in\mathfrak{A}(W)$ the commutator $\left[x,y\right]$ is compact. This implies that $\pi(x)$ is in the center of $\pi(\mathfrak{A}(W))$. But by Corollary \ref{simplicity} the C$^\ast$-algebra $\pi(\mathfrak{A}(W))$ is simple, so in particular its center is trivial. We get that $\pi(x)\in\mathbb{C}1$. Since by Corollary \ref{embedding2} the quotient map $\pi:\mathcal{B}(\ell^{2}(W))\rightarrow\mathcal{B}(\ell^{2}(W))/\mathcal{K}$ restricts to an embedding of $C_{r,q}^{\ast}(W)$ into $\pi(\mathfrak{A}(W))$, we get that $x\in\mathbb{C}1$.
\end{proof}


\subsection*{Acknowledgements}

It is a pleasure to thank my supervisor Martijn Caspers. The present paper is to a large amount the result of many fruitful discussions in which he explained notions and concepts to me and gave me feedback on the direction of my work. I am also grateful to Sven Raum and Adam Skalski for their feedback on an earlier draft of this paper. Finally, I want to thank Stefaan Vaes for very useful communication on interpolated free group factors.

I acknowledge support by the NWO project ``\emph{The structure of Hecke-von Neumann algebras}'', 613.009.125.



\end{document}